\documentclass[12pt,a4paper]{article}
\usepackage[OT1]{fontenc}
\usepackage{amsthm,amsmath,amssymb,amsfonts,color, enumerate}
\usepackage[authoryear]{natbib}

\usepackage[colorlinks, allcolors=blue]{hyperref}
\usepackage{multibib}
\newcites{suppl}{References}
\usepackage{graphicx}
\usepackage{array}
\usepackage{float}
\usepackage{titling}
\usepackage{multirow}

\def\spacingset#1{\renewcommand{\baselinestretch}%
{#1}\small\normalsize} \spacingset{1}

\newcommand{\blind}{0}

\numberwithin{equation}{section}

\newtheoremstyle{dotless}{}{}{\itshape}{}{\bfseries}{}{ }{}
\theoremstyle{dotless}
\newtheorem{theorem}{Theorem}[section]
\newtheorem{corollary}[theorem]{Corollary}
\newtheorem{assump}[theorem]{Assumption}
\newtheorem{lemma}[theorem]{Lemma}
\newtheorem{Proposition}[theorem]{Proposition}

\newtheoremstyle{definition}{}{}{}{}{\bfseries}{}{ }{}
\theoremstyle{definition}

\newtheorem{remark}[theorem]{Remark}
\newtheorem{algorithm}[theorem]{Algorithm}
\newtheorem{example}[theorem]{Example}

\newcommand{\Pb}{\mathbb{P}}
\newcommand{\E}{\mathbb{E}}

\newcommand{\R}{\mathbb{R}}
\newcommand{\U}{\mathbb{U}}

\newcommand{\V}{\mathbb{V}}
\newcommand{\Z}{\mathbb{Z}}
\newcommand{\N}{\mathbb{N}}

\newcommand{\IF}{\mathcal{IF}}
\newcommand{\floor}[1]{\lfloor #1 \rfloor}
\newcommand{\convd}{\overset{\mathcal{D}}{\Longrightarrow}}
\newcommand{\eqd}{\overset{\mathcal{D}}{=}}
\newcommand{\convp}{\overset{\mathbb{P}}{\Longrightarrow}}
\newcommand{\limm}{\lim_{m\to\infty}}

\newcommand{\hattheta}{\hat{\theta}}
\newcommand{\hatF}{\hat{F}}
\newcommand{\hatD}{\hat{D}}
\newcommand{\hatDSN}{\hat{D}^{\rm SN}}
\newcommand{\hatmu}{\hat{\mu}}
\newcommand{\op}{o_{\mathbb{P}}}
\newcommand{\hatFij}{\hat{F}_i^j}
\newcommand{\hatqijbeta}{[\hat{q_{\beta}}]_i^j}
\DeclareMathOperator*{\argmax}{\arg\!\max}
\DeclareMathOperator{\Cov}{Cov}
\DeclareMathOperator{\Cor}{Cor}
\DeclareMathOperator{\Var}{Var}
\DeclareMathOperator{\vech}{vech}
\newcommand{\tabj}{\newline\newline\newline}
\newcommand{\rev}[1]{\textcolor{red}{#1}}

\begin{document}

\title{A likelihood ratio approach to sequential change point detection for a general class of parameters}
\if0\blind
{
\author{ {\small Holger Dette, Josua G\"osmann}\\
 {\small Ruhr-Universit\"at Bochum }\\
 {\small Fakult\"at f\"ur Mathematik} \\
 {\small44780 Bochum} \\
 {\small Germany }
}
} \fi
\maketitle

\begin{abstract}
In this paper we propose a new approach for sequential monitoring of a general class of parameters of a $d$-dimensional time series, which can be estimated by approximately linear functionals of the empirical distribution function.
We consider a closed-end-method, which is motivated by the likelihood ratio test principle and compare the new method with two alternative procedures.
We also incorporate self-normalization such that estimation of the long-run variance is not necessary.
We prove that for a large class of testing problems the new detection scheme has asymptotic level $\alpha$ and is consistent.
The asymptotic theory is illustrated for the important  cases of monitoring a change in the mean, variance and correlation.
By means of a simulation study it is demonstrated that the new test performs better than the currently available procedures for these problems.
Finally the methodology is illustrated by a small data example investigating index prices from the dot-com bubble.
\end{abstract}



\newpage
\spacingset{1.45}
\section{Introduction}

An important problem in statistical modeling of time series is the problem of testing for structural stability as changes in the data generating process may have a substantial impact on statistical inference developed under the assumption of stationarity.
Because of its importance there exists a large amount of literature, which develops tests for structural breaks in various models and we refer to \cite{Aue2013} and \cite{Jandhyala2013} for more recent reviews of the literature.
There are essentially two ways how the problem of change point analysis is addressed.
A large portion of the literature discusses a-posteriori change point analysis, where the focus is on the detection of structural breaks given a historical data set [see \cite{Davis1995}, \cite{Csorgo1997}, \cite{Aue2009}, \cite{Jirak2015} among many others].
On the other hand, in many applications, such as engineering, medicine or risk management data arrives steadily and therefore several authors have addressed the problem of sequentially monitoring changes in a parameter.
\cite{Page1954}, \cite{Hinkley1971}, \cite{Moustakides1986}, \cite{Nikiforov1987},  and \cite{Lai1995} among others developed detection schemes for models with an infinite time horizon.
These methods always stop and have been partly summarized by the term 'statistical process control' (SPC).
SPC-statistics, for example CUSUM- or Shewart-Controlcharts, have then been traditionally compared by the average run length (ARL).
Our approach differs from this literature substantially as we work under the paradigm introduced by \cite{Chu1996}.
These authors proposed an alternative monitoring ansatz in the context of testing the structural stability of the parameters in a linear model, which on the one hand allows to control (asymptotically) the type I error (if no changes occur) and on the other hand provides the possibility of power analysis. \cite{Horvath2004}, \cite{Fremdt2014} extended this approach for linear models with infinite time horizon, while \cite{Aue2012}, \cite{Wied2013}, \cite{Pape2016} developed monitoring procedures for changes in a capital asset pricing model, correlation and variance under the assumption of a finite time horizon.

In this paper we propose a general alternative sequential test in this context, which is applicable for change point analysis of a $p$-dimensional parameter of a $d$-dimensional time series if a historical data set from a stable phase is available and then data arrives consecutively.
Our approach differs from the current methodology as it is motivated by a likelihood ratio test for a structural change.
To be precise, \cite{Wied2013} proposed to compare estimates from the historical data set, say $X_1, \ldots , X_m$, with estimators from the sequentially observed data $X_{m+1}, \ldots , X_{m+k}$ [see also \cite{Chu1996,Horvath2004,Aue2012,Pape2016}], while \cite{Fremdt2014} and \cite{Kirch2018} suggested to compare the estimate from the historical data set with estimates from the sequentially observed data $X_{m+j+1}, \ldots , X_{m+k}$ (for all $j=0,\ldots , k-1$).
In contrast, motivated by the likelihood ratio principle, our approach sequentially compares the estimates from the samples $X_1, \ldots , X_{m+j}$  and   $X_{m+j+1}, \ldots , X_{m+k}$ (for all $j=0,\ldots , k-1$).

Moreover, we also propose a self-normalized test, which avoids the problem of estimating the long-run variance.
While the concept of self-normalization has been studied intensively for a-posteriori change point analysis [see \cite{Shao2010}, \cite{shao2015} and \cite{Zhang2018} among many others], to our best knowledge self-normalization has not been studied in the context of sequential monitoring.

The statistical model and the  change point problem for a general parameter of the marginal
distribution are introduced in Section \ref{sec2}, where we also provide the motivation for the statistic used in the sequential scheme (see Example \ref{ex1}) and a discussion of the alternative methods.
The asymptotic properties of the new monitoring scheme are investigated in Section \ref{sec3}.
In particular we prove a result, which allows to control (asymptotically) the probability of indicating a change in the parameter although there is in fact structural stability (type I error).
Moreover, we also show that the new test is consistent and investigate the concept of self-normalization in this context.
These asymptotic considerations require several assumptions, which are stated in the general context and verified in Section \ref{sec4} for the case of monitoring changes in the mean and variance matrix.
In Section \ref{sec5} the finite sample properties of the new procedure are investigated by means of a simulation study, and we also demonstrate the (empirical) superiority of our approach. 
Here we also illustrate our approach by a small real data example investigating the log-returns of NASDAQ and S\&P 500 during the dot-com bubble.
Finally, Section \ref{sec:concout} gives a brief conclusion and an outlook defining subjects for future research.
All proofs are deferred to an online supplement.

\section{Sequential change point testing}\label{sec2}

Consider a $d$-dimensional time series $\{X_t\}_{t\in \mathbb{Z}}$, where $X_t$ has distribution function $F_t$, and denote by $\theta_t = \theta(F_t)$ a $p$-dimensional parameter of interest of the distribution of $F_t$.
We are taking the sequential point of view and assume that there exists a historical period of length, say $m \in \N$, such that the process is stable in the sense
\begin{align}\label{eq:stability}
\theta_1 = \theta_2 = \dots = \theta_m~.
\end{align}
We are interested to monitor if the parameter $\theta_{m+k}$ changes in the future $m+k \geq m+1$.
The sequence $X_1,\dots ,X_{m}$ is usually referred to as a historical or initial training data set, see for example \cite{Chu1996}, \cite{Horvath2004}, \cite{Wied2013} or \cite{Kirch2018}, among many others.
Based on this stretch of ``stable'' observations a sequential procedure should be conducted to test the hypotheses
\begin{align}\label{globalhypo0}
H_0&:\; \theta_1 = \dots = \theta_m = \theta_{m+1} =\theta_{m+2} = \ldots ~,
\end{align}
against the alternative that the parameter $\theta_{m+k^{\star}}$ changes for some $ k^{\star} \geq 1$, that is
\begin{align}\label{globalhypo1}
 H_1&:\; \exists k^{\star} \in \N:\;\;
\theta_1 = \dots =\theta_{m+k^{\star}-1} \neq \theta_{m+k^{\star}} = \theta_{m+k^{\star}+1} = \ldots ~,
\end{align}
In order to motivate our approach in particular the test statistic used in the detection scheme, which will be used in the proposed sequential test, we begin with a very simple example of a change in the mean.

\begin{example} \label{ex1}
Consider a sequence $\{ X_t\}_{t\in \Z}$ of independent, $d$-dimensional normal distributed random variables with (positive definite) variance matrix $\Sigma$ and mean vectors
\begin{align} \label{meanvec}
\mu_t =\E[X_t] =  \theta (F_t)  = \int_{\R^d} x  dF_t(x)  \in \R^d ~,~~j=1,2, \ldots .
\end{align}
During the monitoring procedure, we propose to successively test the hypotheses
\begin{align}\label{singlehypo:mean}
\begin{split}
H_0&:\; \mu_1 = \dots = \mu_m = \mu_{m+1} = \dots =\mu_{m+k} ~,\\
\mbox{ versus } \quad H_A^{(k)}&: \exists\; j \in \{ 0,\dots, k-1 \}: \;\; \mu_1 = \dots = \mu_{m+j}  \neq \mu_{m+j+1} =\dots =\mu_{m+k}
\end{split}
\end{align}
based on the sample $X_1 , \ldots , X_{m+k}$.
Under the assumptions made in this example we can easily derive the likelihood ratio
\begin{align*}
\Lambda_m(k)
&= \dfrac{ \sup\limits_{\mu \in \R^d}  \prod\limits_{t=1}^{m+k}f(X_t, \mu)}
{ \sup\limits_{\substack{j \in \{0,\dots,k-1\} \\ \mu^{(1)}, \mu^{(2)} \in \R^d}}
\prod\limits_{t=1}^{m+j}f(X_t, \mu^{(1)}) \cdot \prod\limits_{t=m+j+1}^{m+k}f(X_t, \mu^{(2)})}~,
\end{align*}
where $f ( \cdot , \mu ) $ denotes the density of a  normal distribution with mean $\mu$ and variance matrix $\Sigma$ (note that the first $m$ observations are assumed to be mean-stable).
A careful calculation now proves the identity
\begin{align}\label{eq:loglikelihood}
-2 \log \big( \Lambda_m(k) \big)
& = \max_{j=0}^{k-1} \dfrac{(m+j)(k - j)}{m+k}\big(\hatmu_1^{m+j} - \hatmu_{m+j+1}^{m+k}\big)^\top
\Sigma^{-1}\big(\hatmu_1^{m+j} - \hatmu_{m+j+1}^{m+k}\big) \\
& = \max_{j=0}^{k-1}\dfrac{(m+k)(m+j)}{(k - j)}
\big(\hatmu_1^{m+j} - \hatmu_{1}^{m+k}\big)^\top
\Sigma^{-1}\big(\hatmu_1^{m+j} - \hatmu_{1}^{m+k}\big)~,\nonumber
\end{align}
where $v^\top$ denotes the transposed of the vector $v$ (usually considered as a column vector) and
\begin{align*}
\hat{\mu}_i^j
= \frac{1}{j-i+1}\sum_{t=i}^{j} X_t
\end{align*}
is the mean of the observations $X_i,\ldots,X_j$.
Consequently the null hypothesis  $H_0^{(k)}$ should be rejected in favor of the alternative $H_A^{(k)}$ for large values of the statistic
\begin{align} \label{testa}
\max_{j=0}^{k-1} (m+j)(k - j)\big(\hatmu_1^{m+j} - \hatmu_{m+j+1}^{m+k}\big)^\top
\Sigma^{-1}\big(\hatmu_1^{m+j} - \hatmu_{m+j+1}^{m+k}\big)~.
\end{align}
However, as pointed out in \cite{Csorgo1997}, the asymptotic properties of a likelihood ratio type statistic of the type \eqref{testa} are difficult to study.
For this reason we propose to use a weighted version of \eqref{testa} and consider the statistic
\begin{align}\label{testb}
&\max_{j=0}^{k-1} (m + j)^2(k - j)^2\big(\hatmu_1^{m+j} - \hatmu_{m+j+1}^{m+k}\big)^\top
\Sigma^{-1}\big(\hatmu_1^{m+j} - \hatmu_{m+j+1}^{m+k}\big)\\
= & \max_{j=0}^{k-1} \Big( (k-j) \sum_{t=1}^{m+j} X_t  - (m+j) \sum_{t=m+j+1}^{m+k} X_t \Big)
\Sigma^{-1}
\Big( (k-j) \sum_{t=1}^{m+j} X_t  - (m+j) \sum_{t=m+j+1}^{m+k} X_t \Big)~,
\nonumber
\end{align}
for which (after appropriate normalization) weak convergence of a corresponding sequential empirical process can be established.
Note that the right-hand side in \eqref{testb} corresponds to the well known CUSUM statistic, which has become a standard tool for change point detection in a retrospective setting.
\end{example}

\medskip
\noindent
Motivated by the previous example we propose to use the statistic
\begin{align}\label{genstat}
\hatD_m(k)
= m^{-3} \max_{j=0}^{k-1}  \,
(m+j)^2(k - j)^2
\big( \hattheta_{1}^{m+j}  - \hattheta_{m+j+1}^{m+k}\big)^\top
\hat \Sigma_m^{-1}
\big( \hattheta_{1}^{m+j} - \hattheta_{m+j+1}^{m+k} \big) ~,
\end{align}
for monitoring changes in the parameter $\theta_j$, where $\hattheta_i^j = \theta(\hatF_i^j)$ denotes the estimator obtained from the empirical distribution function
\begin{align}\label{eq:Fij}
\hatF_i^j(z)
= \dfrac{1}{j-i+1}\sum_{t=i}^j I\{X_t \leq z\}~
\end{align}
of  the  observations $X_i,\dots,X_j$ and the matrix $\hat \Sigma_m $ corresponds to an estimator of a long-run variance based on $X_1,\dots,X_m$.
The scaling by $m^{-3}$ will be necessary to obtain weak convergence in the sequel.

We use the sequence $\{\hat{D}_m(k)\}_{k \in \mathbb{N}}$ in combination with an increasing threshold function $w(\cdot)$ as a monitoring scheme.
More precisely,  let $T\in \R_+$ denote a constant factor (with $Tm \in \N$) defining the window of monitoring, then we reject the null hypothesis in \eqref{globalhypo0} at the first time $k \in \{1,2,\ldots , Tm\}$ for which the detector $\hat{D}_m(k)$ exceeds the threshold function $w: [0,T] \to \R_{+} $, that is
\begin{align}\label{ineq:reject}
\hat{D}_m(k) > w(k/m)~.
\end{align}
This definition yields a stopping rule defined by
\begin{align*}
\tau_m
= \inf \Big\{ 1 \leq k \leq{Tm}\;|\; \hat{D}_m(k) > w(k/m) \Big\}~,
\end{align*}
(if the set $\{ 1 \leq k \leq{Tm}\;|\; \hat{D}_m(k) > w(k/m) \} $ is empty we define $\tau_m =\infty$).
The threshold function has to be chosen such that the test has asymptotic level $\alpha$, that is
\begin{align}\label{stmt:levelalpha}
\limsup_{m \to \infty}\; \Pb_{H_0} \big( \tau_m < \infty \big)
&= \limsup_{m \to \infty}\; \Pb_{H_0} \bigg( \max_{k=1}^{mT}
\dfrac{\hat{D}_m(k)}{w(k/m)} > 1 \bigg)
\leq \alpha~,
\end{align}
and is consistent, i.e.
\begin{align}\label{stmt:power}
\lim_{m \to \infty}\; \Pb_{H_1} \big( \tau_m < \infty \big)
= 1~.
\end{align}
Following \cite{Aue2012} we call this procedure a closed-end method, because monitoring is only performed in the interval $[m+1, mT]$.

To our best knowledge, the detection scheme defined by \eqref{ineq:reject} has not been considered in the literature so far.
However, our approach is related to the work of \cite{Chu1996,Horvath2004,Aue2012,Wied2013} and \cite{Pape2016}, who investigated sequential monitoring schemes for various parameters (such as the correlation, the variance or the parameters of the capital asset pricing model).
In the general situation considered in this section their approach uses
\begin{align}\label{scheme:wied}
\hat Q_m(k)
= \dfrac{k^2}{m}\big (\hattheta_1^{m} - \hattheta_{m+1}^{m+k}  \big )
^\top  \hat \Sigma_m^{-1}
(\hattheta_1^{m} - \hattheta_{m+1}^{m+k}  \big )
\end{align}
as a basic statistic in the sequential procedure.
Note that a sequential scheme based on this statistic measures the differences between the estimator $\hattheta_1^{m} $ from the initial data and the estimator $\hattheta_{m+1}^{m+k}$ from all observations excluding the training sample.
As a consequence - in particular in the case of a rather late change - the estimator $\hattheta_{m+1}^{m+k}$ may be corrupted by observations before the change point, which might lead to a loss of power.
Another related procedure uses the statistic
\begin{align}\label{scheme:kirch}
\hat P_m(k)
= \max_{j=0}^{k-1} \dfrac{(k-j)^2}{m} \big ( \hattheta_1^{m} - \hattheta_{m+j+1}^{m+k} \big )  ^\top  \hat \Sigma_m^{-1} \big ( \hattheta_1^{m} - \hattheta_{m+j+1}^{m+k} \big)
\end{align}
and was recently suggested by \cite{Fremdt2014} and reconsidered by \cite{Kirch2018}.
These authors compare the estimate from the data $X_1,\ldots , X_m$ with estimates from the data $X_{m+j+1},\ldots , X_{m+k}$ (for different values of $j$). This method may lead to a loss in power in problems with a small sample of historical data and a rather late change point.
In contrast our approach compares the estimates of the parameters before and after all potential positions of a change point $ j \in \{m+1, \ldots , m+k \}$.\\
In this paper, we argue that the performance of the change point tests can be improved by replacing $\hattheta_1^{m}$ by $\hattheta_1^{m+j}$ inside the maximum, which would directly lead to a scheme of the form \eqref{genstat}.
Here, we would already like to point to our simulation study in Section \ref{sec5}, which contains many cases where a sequential detection scheme based on the statistic $\hat D_m$ outperforms schemes based on $\hat Q_m $ or $\hat P_m $.

\section{Asymptotic properties} \label{sec3}

In the subsequent discussion we use the following notation.
We denote  by $\ell^{\infty}(V_1,V_2)$ the space of all bounded functions $f: V_1 \to V_2$ equipped with sup-norm, where  $V_1, V_2$ are normed linear spaces.
The symbols $\convp$ and $\convd$ mean convergence in probability and weak convergence (in the space under consideration), respectively.
The process $\{W(s)\}_{s \in [0,T+1]}$ will usually represent a standard $p$-dimensional Brownian motion.
For a vector $v \in \R^d $, we denote by $|v| = \big ({\sum_{i=1}^d v_{i}^2} \big) ^{1/2}$ its euclidean norm.
\subsection{Weak convergence}
\label{sec3a}

Throughout this paper we denote by $\theta = \theta(F)$ a $p$-dimensional functional of the $d$-dimensional distribution function $F$ and define its influence function (assuming its existence) by
\begin{align}\label{def:inflfunc}
\IF(x,F,\theta)
= \lim_{\varepsilon \searrow 0} \dfrac{\theta((1-\varepsilon)F + \varepsilon\delta_x) - \theta(F)}{\varepsilon}~,
\end{align}
where $\delta_x(z) = I\{x \leq z\}$ is the distribution function of the Dirac measure at the point $x \in \R^d$ and the inequality in the indicator is understood component-wise.
Throughout this section we make the following assumptions, which will be verified for several important examples in Section \ref{sec4} [see also \cite{Shao2010} for similar regularity conditions].

\begin{assump}\label{assump:meta-1}
Under the null hypothesis \eqref{globalhypo0} we assume that the times series $\{X_t\}_{t\in \mathbb{N}}$ is strictly stationary with $\E[\IF(X_1,F,\theta)] = 0$ and that the weak convergence
\begin{align}
\dfrac{1}{\sqrt{m}} \sum_{t=1}^{\floor{ms}} \IF(X_t,F,\theta)
\convd \sqrt{\Sigma_F} W(s)~,
\end{align}
holds in the space $\ell^{\infty}([0,T+1],\R^p)$ as $m \to \infty$, where the long-run variance matrix is defined by (assuming convergence of the series)
\begin{align}\label{eq:meta-covar}
 \Sigma_F
= \sum_{t \in \mathbb{Z}}
\Cov\big(\IF(X_0, F, \theta),
~\IF(X_t, F, \theta) \big)
 \in \R^{p\times p}
\end{align}
and $\{W(s)\}_{s \in [0,T+1]}$ is a $p$-dimensional (standard) Brownian motion.
\end{assump}

\medskip
\begin{assump}\label{assump:meta-2}
The remainder terms
\begin{align}\label{eq:meta-remainder}
R_{i,j}= \hattheta_i^j - \theta(F) - \dfrac{1}{j-i+1} \sum_{t=i}^j \IF(X_t, F, \theta)
\end{align}
in the linearization of $\hattheta_i^j - \theta(F) $ satisfy
\begin{align}\label{eq:meta-remainderrate}
\sup_{1\leq i < j \leq n} (j-i+1) |R_{i,j}| = \op(n^{1/2})~.
\end{align}
\end{assump}

\medskip
\noindent
For the statement of our first result we introduce the notations (throughout this paper we use the convention $\hattheta_z^u = 0 $, whenever $z > u$)
\begin{align} \label{utilde}
\tilde{\U}(\ell,z,u)
& := (u-z)(z-\ell)\big( \hattheta_{\ell+1}^z - \hattheta_{z+1}^u \big)~,
\\
\label{eq:meta-cusum}
\U(z,u) &: = \tilde{\U}(0,z,u) =
(u-z)z\big( \hattheta_{1}^z - \hattheta_{z+1}^u \big) ~,
\end{align}
 and denote by
\begin{align}\label{eq:twoset}
\Delta_2 &= \{ (s,t) \in [0,T+1]^2\;\vert\; s \leq t\}~,
\\
\Delta_3 &= \big \{ (r,s,t) \in [0,T+1]^3 ~|~ r \leq s \leq t \big \}.
\label{eq:triset}
\end{align}
the $2$-dimensional triangle and the $3$-dimensional oblique pyramid in $[0,T+1]^2 $ and\\ $[0,T+1]^3$, respectively.

\medskip

\begin{theorem}\label{thm:meta-main}
Let Assumptions \ref{assump:meta-1} and \ref{assump:meta-2} be satisfied.
If the null hypothesis in \eqref{globalhypo0} holds, then as $m \to \infty$
\begin{align}\label{conv:kern}
\big \{
m^{-3/2}  \tilde \U(\floor{mr}, \floor{ms}, \floor{mt}) \big \}_{(r,s,t) \in \Delta_3}
&\convd \Sigma^{1/2}_F \big \{   B(s,t) + B(r,s) -B(r,t)  \big \}_{(r,s,t) \in \Delta_3}
\end{align}
in the space $\ell^{\infty}(\Delta_3, \R^p)$,
where the process  $\{B(s,t)\}_{(s,t) \in \Delta_2}$ is defined by
\begin{align}\label{eq:Bst}
B(s,t) = tW(s) - sW(t) ~,~~ (s,t) \in \Delta_2,
\end{align}
and $\{ W (s) \}_{s\in [0,T+1]}$ denotes a $p$-dimensional Brownian motion on the interval $[0,T+1]$.
\end{theorem}

\begin{remark}\label{rem:Bst}
As a by-product of Theorem \ref{thm:meta-main} and the representation \eqref{eq:meta-cusum} we obtain the weak convergence of the double-indexed CUSUM-process \eqref{eq:meta-cusum}, that is ($m \to \infty$)
\begin{align}\label{conv:cusum}
\big \{ m^{-3/2}\cdot\U(\floor{ms},\floor{mt})
\big \}_{(s,t) \in \Delta_2}
\convd  \{  \Sigma_{F}^{1/2}B(s,t)\}_{(s,t) \in \Delta_2}~,
\end{align}
where $
\Delta_2 $
denotes the $2$-dimensional triangle in $[0,T+1]^2$ and the process $B$ is defined in \eqref{eq:Bst}.
In particular the covariance structure of
the process $B$ is given by
\begin{align*}
\Cov \big (B(s_1, t_1), B(s_2, t_2) \big )
&= t_1t_2(s_1 \wedge s_2) - t_1s_2(s_1 \wedge t_2) - s_1t_2 (t_1 \wedge s_2) + s_1s_2(t_1 \wedge t_2)~.
\end{align*}
Consequently, the process $\{B(s,t)\}_{ (s,t) \in \Delta_2}$ can be considered as a natural extension of the standard Brownian bridge as for fixed $t$ the process $\{B(s,t)\}_{ s \in [0,t] }$ is a Brownian bridge on the interval $[0,t]$.
\end{remark}

\noindent
Observing the definition \eqref{eq:meta-cusum} the statistic \eqref{genstat} allows the representation
\begin{align} \label{dhat}
\hatD_m(k)
= m^{-3} \max_{j=0}^{k-1}
| \U^\top(m+j,m+k) \hat \Sigma_m^{-1} \U(m+j,m+k) |~,
\end{align}
and we obtain the following Corollary as a consequence of Theorem \ref{thm:meta-main}.
\begin{corollary}\label{cor:meta-lvl1}
Let the assumptions of Theorem \ref{thm:meta-main} be satisfied.
If the null hypothesis in \eqref{globalhypo0} holds, and $\hat \Sigma_m$ denotes a consistent, non-singular estimator of the long-run variance $\Sigma_F$ defined in \eqref{eq:meta-covar}, then as $m \to \infty$
\begin{align*}
\max_{k=1}^{Tm} \dfrac{\hatD_m(k)}{w(k/m)} \convd
\sup_{t \in [1,T+1]} \sup_{s\in [1,t]}
\dfrac{B(s,t)^\top B(s,t)}{w(t-1)}~.
\end{align*}
for any threshold function $w: [0,T] \to \R^+ $, which is increasing.
\end{corollary}

\noindent By the result in Corollary \ref{cor:meta-lvl1} it is reasonable to choose for a given level $\alpha$ a threshold function $w_{\alpha}(\cdot)$, such that
\begin{align}\label{ineq:threshold}
\Pb \bigg( \sup_{t \in [1,T+1]} \sup_{s \in [1,t]}
\dfrac{ B(s,t)^\top B(s,t)}{w_{\alpha}(t-1)} > 1 \bigg)
= \alpha
\end{align}
and to reject the null hypothesis $H_0$ in \eqref{globalhypo0} at time $k$, if
\begin{align}\label{ineq:rejectnew}
\hat{D}_m(k) > w_{\alpha}(k/m)~.
\end{align}
By Corollary \ref{cor:meta-lvl1} this test has asymptotic level $\alpha$, that is
\begin{align*}
\limm \Pb_{H_0}\Big( \max_{k=1}^{Tm} \dfrac{\hat{D}_m(k)}{w_{\alpha}(k/m)} > 1 \Big)
& = \alpha~
\end{align*}
(if the  assumptions of
Theorem \ref{thm:meta-main} hold and $w_{\alpha}$ satisfies \eqref{ineq:threshold}).
The choice of $w_{\alpha}(\cdot)$ has been investigated by several authors [see \cite{Chu1996}, \cite{Aue2009} and \cite{Wied2013} among others] and we will compare different options by means of a simulation study in Section \ref{sec5}.
Note that one can take any function (which is increasing and bounded from below by a positive constant) and multiply an appropriate constant such that \eqref{ineq:threshold} is fulfilled.
Next we discuss the consistency of the monitoring scheme \eqref{ineq:rejectnew}.
For this purpose  we consider the alternative hypothesis in \eqref{globalhypo1}, where the location of the change point is increasing with the length of the training sample, that is $m + k^* = \lfloor mc \rfloor $ for some $1 < c < T+1$.
Recalling the definition of $\hat D_m(k)$ in \eqref{genstat} and observing the inequality
\begin{align*}
\max_{k=1}^{Tm} \hat D_m(k)
\geq {\floor{mc}^2 \big(T(m+1) - \lfloor mc \rfloor \big)^2 \over m^3} \cdot
\frac{
\big( \hattheta_{1}^{\floor{mc}}  - \hattheta_{\floor{mc}+1}^{T(m+1)} \big )^\top
 \hat \Sigma_m^{-1}
\big( \hattheta_{1}^{\floor{mc}}  - \hattheta_{\floor{mc}+1}^{T(m+1)} \big )}{w_{\alpha}(T)}
\end{align*}
it is intuitively clear that the statistic $\max_{k=1}^{Tm} \hat D_m(k)$ converges to infinity, provided that $\hattheta_{1}^{\floor{mc}}$ and  $\hattheta_{\floor{mc}+1}^{T(m+1)}$ are consistent estimates of the parameter $\theta$ before and after the change point $m+k^*=\floor{mc}$ and $\hat \Sigma_m$
converges to a positive definite $ p \times p$ matrix.
The following Theorem \ref{thm:meta-power} makes these heuristic arguments more precise.
Its proof requires several assumptions, which are stated first.
The result might be even correct under slightly weaker assumptions.
However, in the form stated below we can also prove consistency of a sequential scheme based on a self-normalized version of $\hatD_m (k)$ (see Theorem \ref{cor:meta-lvl1SN} in Section \ref{sec3b}).

\begin{assump}\label{assump:meta-power}
If the alternative hypothesis $H_1$ defined in \eqref{globalhypo1} holds we assume that the change occurs at position $m+k^*=\floor{mc}$ for some $c \in (1,T+1)$.
Moreover, let  $\{ Z_t(1)\}_{t \in \mathbb{Z}}$ and $\{ Z_t(2)\}_{t \in \mathbb{Z}}$  denote (strictly) stationary $\R^d$-valued processes with marginal distribution functions $F^{(1)}$ and $F^{(2)}$, respectively, such that \begin{align*}
\theta(F^{(1)})
\neq \theta(F^{(2)}) ~,
\end{align*}
and that for each $m \in \mathbb{N}$
\begin{align}
\label{mix3a}
(X_1,X_2,\dots,  X_{\floor{mc}} )
&\eqd (Z_1(1), \dots, Z_{\lfloor mc \rfloor}(1))~,\\
\label{mix3b} (X_{\floor{mc}+1}, \dots, X_{\floor{mT}})
&\eqd (Z_{\floor{mc}+1}(2), \dots, Z_{\floor{mT}}(2))~.
\end{align}
Note, that formally the process $\{X_{t}\}_{t=1, \ldots \floor{mT}} $ is a triangular array, that is $X_{t}=X_{m,t}$, but we do not reflect this in our notation.
Further, we assume that there exist two (standard) Brownian motions $W_1$ and $W_2$ such that the joint weak convergence
\begin{align*}
\left(
\begin{array}{c}
\big\{\frac{1}{\sqrt{m}} \sum_{t=1}^{\floor{ms}} \IF(Z_t(1),F^{(1)},\theta)  \big\}_{s\in [0,c]} \\
\big\{\frac{1}{\sqrt{m}} \sum_{t=\floor{mc}+1}^{\floor{ms}} \IF(Z_{t}(2), F^{(2)},\theta)  \big\}_{s\in [c,T+1]}
\end{array}
\right)~~\convd ~
\left(
\begin{array}{c}
 \big\{\sqrt{\Sigma_{F^{(1)}}} W_1(s)\big\}_{s\in [0,c]} \\
 \big\{\sqrt{\Sigma_{F^{(2)}} } \big( W_2(s) - W_2(c) \big)\big\}_{s\in [c,T+1]}
\end{array}
\right)
\end{align*}
holds,  where $\Sigma_{F^{(1)}}$ and $\Sigma_{F^{(2)}}$ denote positive definite matrices defined in the same way as \eqref{eq:meta-covar}, that is
\begin{align}\label{eq:meta-covartwo}
\Sigma_{F^{(\ell)}}
= \sum_{t \in \mathbb{Z}}
\Cov\big(\IF(Z_0 (\ell) , F^{(\ell)}, \theta),
~\IF(Z_t (\ell) , F^{(\ell)}, \theta) \big)
 \in \R^{p\times p}~,~\ell =1,2
\end{align}
and both phases in \eqref{mix3a} and \eqref{mix3b} fulfill Assumption \ref{assump:meta-2} for the corresponding expansion.
\end{assump}

\medskip

\begin{remark} \label{rem38}
~~
\begin{itemize}
\item[(a)]
The interpretation of Assumption \ref{assump:meta-power} is as follows: There exist two regimes and the process under consideration switches from one regime to the other.
\item[(b)]
The assumption of two stationary phases before and after the change point is commonly made in the literature to analyze change point tests under the alternative
[see for example \cite{Aue2009}, \cite{Dette2016}, \cite{Kirch2018} among others].
Note, that we do not assume that the two limiting processes $W_1$ and $W_2$ are independent.
\item[(c)]
Often Assumption \ref{assump:meta-power} is directly implied by the underlying change point problem. For example, in the situation of the mean vector introduced in \eqref{ex1} it is usually assumed that $X_t = \mu_t + \varepsilon_t$, where $\{\varepsilon_t\}_{t\in \Z}$ is a stationary process and $\mu_t =\mu^{(1)}$ if $t \leq \floor{mc}$ and $\mu_t =\mu^{(2)}$  if $t \geq \floor{mc}+1$. In this case Assumption \ref{assump:meta-power} is obviously satisfied.
Further examples are discussed in Section \ref{sec4}.
\end{itemize}
\end{remark}

\begin{theorem}\label{thm:meta-power}
Let Assumption \ref{assump:meta-power} be satisfied and let the threshold function $w_\alpha$ satisfy \eqref{ineq:threshold}.
Further assume that $\hat{\Sigma}_m$ denotes a consistent estimator of the long-run variance $\Sigma_{F^{(1)}}$ based on the  observations $X_1,\dots,X_m$.
Under the alternative hypothesis $H_1$ we have
\begin{align*}
\limm \Pb \bigg( \max_{k=0}^{Tm} \dfrac{ \hatD_m(k)}{w_{\alpha}(k/m)} > 1 \bigg)
= 1~.
\end{align*}
\end{theorem}
\medskip
\begin{remark} \label{rem39}
We can establish similar results for the statistics \eqref{scheme:wied} and \eqref{scheme:kirch}
proposed by \cite{Wied2013} among others and \cite{Fremdt2014} among others, respectively.
For example, if the assumptions of Theorem \ref{thm:meta-main} are satisfied we obtain the weak convergence
\begin{align*}
\max_{k=1}^{Tm} \dfrac{\hat Q_m(k)}{w(k/m)}
&\convd
\sup_{t \in [1,T+1]} \dfrac{B(t,1)^\top B(t,1)}{w(t-1)}~,\\
\max_{k=1}^{Tm} \dfrac{\hat P_m(k)}{w(k/m)}
& \convd
\sup_{t \in [1,T+1]} \sup_{s \in [1,t]}
\dfrac{\big(B(1,s) + B(t,1)\big)^\top \big(B(1,s) + B(t,1)\big)}{w(t-1)}
\end{align*}
under the null hypothesis, which can be used to construct an asymptotic level $\alpha$ monitoring scheme based on the statistics $\hat Q_m$ and $\hat P_m$, respectively.
Consistency of the corresponding tests follows along the arguments given in the proof of Theorem \ref{thm:meta-power}.
The details are omitted for the sake of brevity.
The finite sample properties of the three different tests will be investigated by means of a simulation study in Section \ref{sec5}.
\end{remark}

\subsection{Self-Normalization}
\label{sec3b}

The test proposed in Section \ref{sec3a} requires an estimator of the long-run variance $\hat \Sigma_m$, and we discuss commonly used estimates for this purpose in Section \ref{sec4}.
However, it has been pointed out by several authors that this problem is not an easy one as the common estimates depend sensitively on a regularization parameter (for example a bandwidth), which might be difficult to select in practice.
An alternative to long-run variance estimation is the concept of self-normalization, which will be investigated in this section.
This approach has been studied intensively for a-posteriori change point analysis [see \cite{Shao2010} and \cite{shao2015} among many others], but - to our best knowledge - self-normalization has not been considered in the context of sequential monitoring.
In the following we discuss a self-normalized version of the statistic $\hatD_m$ proposed in this paper (see equation \eqref{genstat}).
Self-normalization of the statistic $\hat P_m$ in \eqref{scheme:kirch} will be briefly discussed in Remark \ref{rem:nwied}.\\
To be precise, we define a self-normalizing matrix
\begin{align}
\label{eq:meta-snprocess}
\begin{split}
\V(z,u)
&= \sum_{j=1}^{z}j^2(z-j)^2
\big(\hattheta_1^{j} - \hattheta_{j+1}^{z}\big)
\big(\hattheta_1^{j} - \hattheta_{j+1}^{z}\big)^\top\\
&+ \sum_{j=z+1}^{u} (u-j)^2(j-z)^2
\big(\hattheta_{z+1}^{j} - \hattheta_{j+1}^{u} \big)
\big(\hattheta_{z+1}^{j} - \hattheta_{j+1}^{u} \big)^\top
\end{split}
\end{align}
and replace the estimate $\hat \Sigma_m $ of the long-run variance in \eqref{genstat} by the matrix $\frac {1}{m^4} \V (m+j,m+k)$.
This yields the self-normalized statistic
\begin{align} \label{genstatSN}
\hatDSN  (k)
= m \max_{j=0}^{k-1}  \,
(m+j)^2(k - j)^2
\big( \hattheta_{1}^{m+j}  - \hattheta_{m+j+1}^{m+k}\big)^\top
\V^{-1}(m+j,m+k)
\big( \hattheta_{1}^{m+j} - \hattheta_{m+j+1}^{m+k} \big) ~.
\end{align}


\medskip
\begin{theorem}\label{cor:meta-lvl1SN}
Let $w: [0,T] \to \R^+ $ denote any threshold function, which is increasing and let the assumptions of Theorem \ref{thm:meta-main} be satisfied.
If the null hypothesis in \eqref{globalhypo0} holds, then as $m \to  \infty$
\begin{align}
\begin{split}
\max_{k=1}^{Tm} \dfrac{\hatDSN_m(k)}{w(k/m)} \convd
\sup_{t \in [1,T+1]} \sup_{s\in [1,t]}
\dfrac{| \tilde B(s,t) |}{w(t-1)}~,
\end{split}
\end{align}
where
\begin{align*}
\tilde{B}(s,t)
= B^\top(s,t)\big(N_1(s) + N_2(s,t)\big)^{-1}B(s,t)~,
\end{align*}
the process $\{B(s,t)\}_{(s,t) \in \Delta_2}$ is defined in \eqref{eq:Bst} and $\{N_1(s)\}_{s\in [0,T+1]}$ and $\{N_2(s,t)\}_{(s,t) \in \Delta_2}$ are given by
\begin{align}\label{eq:denominator}
\begin{split}
N_1(s)
&= \int\limits_0^sB(r,s)B^\top(r,s)dr~,\\
N_2(s,t)
&= \int_s^t \big(B(r,t) + B(s,r) - B(s,t) \big)\big(B(r,t) + B(s,r) - B(s,t) \big)^\top dr~.
\end{split}
\end{align}

\end{theorem}

\medskip
\noindent
The monitoring rule is now defined in the same way as described in Section \ref{sec3a} determining a threshold function $w_{\alpha}(\cdot)$, such that
\begin{align}\label{ineq:thresholdSN}
\Pb \bigg( \sup_{t \in [1,T+1]} \sup_{s \in [1,t]} \dfrac{ | \tilde{B}(s,t) |}{w_{\alpha}(t-1)} > 1 \bigg) = \alpha
\end{align}
for a given level $\alpha$,
and rejecting the null hypothesis $H_0$ in \eqref{globalhypo0} at the time $k$, if
\begin{align}\label{ineq:rejectnewSN}
\hatDSN_m(k) > w_{\alpha}(k/m)~.
\end{align}
By Theorem \ref{cor:meta-lvl1SN} this test has asymptotic level $\alpha$ and our next result shows that this procedure is also consistent.

\medskip
\begin{theorem}
\label{thm:meta-powerSN}
Let Assumption \ref{assump:meta-power} be satisfied and let the threshold function $w_\alpha$ satisfy \eqref{ineq:thresholdSN}.
Under the alternative hypothesis $H_1$ we have
\begin{align*}
\limm \Pb \Big( \max_{k=0}^{Tm} \dfrac{ \hatDSN_m(k)}{w_{\alpha}(k/m)} > 1 \Big) = 1~.
\end{align*}
\end{theorem}

\medskip
\begin{remark} \label{rem:nwied}
The statistic $\hat{P}_m(k)$ defined in \eqref{scheme:kirch} can be self-normalized in a similar manner, that is
\begin{align}\label{scheme:kirchsn}
\hat{P}^{\rm SN}_m(k)
= m^3 \max_{j=0}^{k-1} (k-j)^2 \big ( \hattheta_1^{m} - \hattheta_{m+j+1}^{m+k} \big )  ^\top \V^{-1}(m+j,m+k) \big ( \hattheta_1^{m} - \hattheta_{m+j+1}^{m+k} \big ) ~.
\end{align}
If the null hypothesis holds and the assumptions of Theorem \ref{thm:meta-main} are satisfied it can be shown using \eqref{conv:cusum} and similar arguments as given in the proof of Theorem \ref{cor:meta-lvl1SN} that
\begin{align}\label{conv:kirchSN}
\begin{split}
\max_{k=1}^{Tm} \dfrac{\hat{P}^{\rm  SN}_m(k)}{w(k/m)} \convd
\sup_{t \in [1,T+1]} \sup_{s\in [1,t]}
\dfrac{| \tilde{B}^{\rm  SN}(s,t) |}{w(t-1)}~,
\end{split}
\end{align}
where the process $\tilde{B}^{\rm  SN}$ is defined by
{\begin{align*}
\tilde{B}^{\rm  SN}(s,t)
= \big( B(1,s) + B(t,1)\big)^\top \big(N_1(s) + N_2(s,t)\big)^{-1} \big( B(1,s) + B(t,1)\big)~.
\end{align*}}
Similarly, consistency follows along the lines given in the proof of Theorem  \ref{thm:meta-powerSN}.
The details are omitted for the sake of brevity.\\
On the other hand a statistic like $\hat{Q}_m(k)$ defined in \eqref{scheme:wied} cannot be self-normalized in a straightforward manner as it does not employ a maximum, which is necessary to separate points before and after the (possible) change point.
In particular one cannot use the matrix $\V$ in \eqref{eq:meta-snprocess}, which is based on such a separation. Obviously, a self-normalization approach without separation could be constructed but this would lead to a severe loss in power  and is therefore not discussed here.
We refer the reader to \cite{Shao2010} for a comprehensive discussion of this problem.
The finite sample properties of both self-normalized methods $\hat{D}^{\rm SN}_m(k)$ and $\hat{P}^{\rm SN}_m(k)$ will be compared by means of a simulation study in Section \ref{sec5}~.
\end{remark}

\subsection{Implementation} \label{sec3c}

We will close this section with a description of the algorithm to detect changes in the functional $\theta(F)$ employing self-normalization.
In the following paragraph $\hat S_m (k) $ denotes any of the statistics $\hatD_m(k), \hat{D}^{\rm SN}_m(k), \hat{P}_m(k), \hat{P}^{\rm SN}_m(k)$ and $\hat{Q}_m(k)$ discussed in Section \ref{sec3a} and \ref{sec3b}.

\begin{algorithm}
Let  $\{X_1,\dots,X_m\}$ denote the ``stable'' training data
satisfying \eqref{eq:stability}.
\begin{itemize}
\item[](\textit{Initialization}) Choose the factor $T$ to determine how much longer the monitoring can be performed. Further choose a threshold function $w_{\alpha}$ such that the probability of type I error is asymptotically $\alpha$.
\item[](\textit{Monitoring}) If $X_{m+k}$ has been observed, compute the the statistic $\hat S_m (k) $ and reject the null hypothesis of no change in the parameter $\theta (F)$ if $\hat S_m (k) > w_{\alpha}(k/m)$.
In this case stop monitoring.
Otherwise, repeat this comparison with the next observation $X_{m+k+1}$.
\item[](\textit{Stop}) If there has been no rejection at time $m+mT$, stop monitoring with observation $X_{m+mT}$ and conclude that no change has occurred within the monitoring period.
\end{itemize}
\end{algorithm}

\section{Some specific change point problems}
\label{sec4}

In this section we illustrate how the assumptions of Section \ref{sec3} can be verified for concrete functionals.
Exemplarily we consider the mean, variance and quantiles, but similar arguments could be given for other functionals under consideration.
To be precise consider a time series $\{ X_t \}_{t\in \mathbb{Z}}$, which forms a physical system in the sense of \cite{Wu2005}, that is
\begin{align}\label{eq:physicalsys}
X_{t}=
\begin{cases}
g(\varepsilon_t,\varepsilon_{t-1},\dots)\;\;\text{if}\;\; t < \floor{mc}~,\\
h(\varepsilon_t,\varepsilon_{t-1},\dots)\;\;\text{if}\;\; t  \geq \floor{mc}~,
\end{cases}
\end{align}
where $\{ \varepsilon_t \}_{t \in \mathbb{Z}} $ denotes a sequence of i.i.d. random variables with values in some measure space $\mathbb{S}$ such that the functions  $g,h: \mathbb{S}^{\mathbb{N}} \to \R^d $ are measurable.
The functions $g$ and $h$ determine the phases of the physical system before and after the change at position $\floor{mc}$ with  $c>1$, respectively.
Under the null hypothesis we will always assume that $g$ and $h$ coincide, which yields that the (whole) times series $\{ X_t \}_{t\in \mathbb{Z}}$ is strictly stationary.
In the case $g \neq h$ the random variables $X_t$ form a triangular array, but for the sake of readability we do not reflect this in our notation.
In order to adapt the concept of physical dependence to the situation considered in this paper, let $\varepsilon_0'$ be an independent copy of $\varepsilon_0$ and define the distances
\begin{align}\label{eq:physicalcoeff}
\begin{split}
\delta_{t,q}^{(1)} &=
\big(  \E [ | g(\varepsilon_t,\varepsilon_{t-1},\dots) - g(\varepsilon_t,\varepsilon_{t-1},\dots,\varepsilon_1, \varepsilon_0',\varepsilon_{-1},\dots) |^q]\big)^{1/q}~,\\
\delta_{t,q}^{(2)} &=
\big(  \E [ | h(\varepsilon_t,\varepsilon_{t-1},\dots) - h(\varepsilon_t,\varepsilon_{t-1},\dots,\varepsilon_1, \varepsilon_0',\varepsilon_{-1},\dots) |^q]\big)^{1/q}~,
\end{split}
\end{align}
which are used to quantify the (temporal) dependence within both phases of the physical system and $\delta_{t,q}^{(\ell)}$ measures the influence of $\varepsilon_0$ on the random variable $X_{t}$.
Further let
\begin{align}
\Theta_{q}^{(\ell)}
= \sum_{t=1}^\infty \delta_{t,q}^{(\ell)}~,\qquad \ell = 1,2~,
\end{align}
denote the sum of the coefficients (which might diverge).
Additionally, we define the (ordinary) long-run variance matrix of the phases before and after the change by
\begin{align}\label{eq:ordcovar}
\begin{split}
\Gamma(g) &= \sum_{t \in \mathbb{Z}} \Cov\big(g(\varepsilon_0,\varepsilon_{-1},\dots), g(\varepsilon_t,\varepsilon_{t-1},\dots)\big)~, \\
\Gamma(h) &= \sum_{t \in \mathbb{Z}} \Cov\big(h(\varepsilon_0,\varepsilon_{-1},\dots), h(\varepsilon_t,\varepsilon_{t-1},\dots)\big)~.
\end{split}
\end{align}

\subsection{Sequential testing for changes in the mean vector}
\label{sec41}

In this section we are interested in detecting changes in the mean
\begin{align} \label{meanvec1}
\mu(F_t) =\E[X_t] = \int_{\R^d} x  dF_t(x) ~,~~t=1,2, \ldots .
\end{align}
of a $d$-dimensional time series $\{X_t\}_{t\in \mathbb{Z}}$.
Sequential detection schemes for a change in the mean have been been
investigated by \cite{Chu1996}, \cite{Horvath2004} and \cite{Aue2009} among others.
We consider the closed-end-procedure developed in Section \ref{sec3} and assume the first $m$ observations $X_1,\dots ,X_{m}$ to be mean-stable.

\noindent As the mean functional \eqref{meanvec1} is linear  the influence function is given by
\begin{align*}
\IF(x,F,\mu)
= x - \mu(F)
= x - \E_F[X]~,
\end{align*}
and therefore Assumption \ref{assump:meta-2} is obviously satisfied (note that  $R_{i,j}=0$ for all $i,j$).
Assumption \ref{assump:meta-1} reduces to Donsker's invariance principle, that is
\begin{align}\label{conv:donsker}
\Big\{ \dfrac{1}{\sqrt{m}} \sum_{t=1}^{\floor{ms}} (X_t - \E[X_t]) \Big\}_{s \in [0,T+1]} \convd  \big \{ \Sigma^{1/2}_F W(s) \big \}_{s \in [0,T+1]}
\end{align}
in $\ell^{\infty}([0,T+1], \R^d)$ with $ \Sigma_F
= \sum_{t \in \mathbb{Z}}
\Cov\big(X_0,X_t\big)$,
which has been derived by \cite{Wu2005} for physical systems under the assumption
that $\Theta_{2}^{(1)} < \infty$ (see See Theorem 3 in this reference).
Note that for this functional the (ordinary) long-run variance matrix $\Gamma(g)$ and $\Sigma_F$ coincide (under stationarity).

If the alternative of a change in the mean at position $\floor{mc}$ for some $c \in (1,T+1) $ holds, we may assume that
\begin{align}\label{eq:shift}
h = g + \Delta\mu~,
\end{align}
where $ \Delta\mu = \E[X_{\floor{mc}}] - \E[X_{\floor{mc}-1}]. $
Consequently, we have $\Theta_{2}^{(1)} = \Theta_{2}^{(2)}$ and if $\Theta_{2}^{(1)} < \infty$, Assumption \ref{assump:meta-power} is also satisfied with $W_1 = W = W_2$ (see also the discussion at the end of Remark \ref{rem38}).
We summarize these observations in the following proposition.

\begin{Proposition}
Assume that \eqref{eq:physicalsys} holds with $\Theta_{2}^{(1)} < \infty$ and further let
$\hat{\Sigma}_m$ denote a consistent estimator of the (positive definite) long-run variance matrix $\Sigma_{F^{(1)}}$ (before the change) based on the observations $X_1,\dots,X_m$.
\begin{itemize}
\item[(a)]
If $g=h$, then the assumptions of Theorem \ref{cor:meta-lvl1} and \ref{cor:meta-lvl1SN} are satisfied
for the functional \eqref{meanvec1}.
In other words: The sequential tests for a change in the mean based on the statistics $\hatD$ or  $\hatDSN$ with $\theta(F_t)  = \int_{\R^d} x  dF_t(x) $
have asymptotic level $\alpha$.
\item[(b)]
Let representation \eqref{eq:shift} hold with $\Delta\mu\neq 0$, then the assumptions of Theorem \ref{thm:meta-power} and \ref{thm:meta-powerSN} are satisfied for the functional \eqref{meanvec1}.
In other words: The sequential tests for a change in the mean based on the statistics $\hatD$ or $\hatDSN$ are consistent.
\end{itemize}
\end{Proposition}
\medskip
\noindent
The finite sample properties of this test will be investigated in Section \ref{sec51}.

\subsection{Sequential testing for changes in the variance}\label{sec42}
In this section, we focus on detecting changes in the variance. Following \cite{Aue2009b}, who investigated this problem in the non-sequential case, we consider   a time series $\{X_t\}_{t\in \Z}$ with common mean $\mu = \mathbb{E}_F[X]$  and define the functional
\begin{align}\label{eq:varfunc}
V(F)
= \int_{\R^d} xx^\top dF(x) - \int_{\R^d} x dF(x) \int_{\R^d}x^\top dF(x)~.
\end{align}
A careful but straightforward calculation shows that the corresponding influence function is given by
\begin{align}\label{eq:infvar}
\begin{split}
\IF(x, F, V)
&= - \E_F[XX^\top] + 2\E_F[X]\E_F[X^\top] + xx^\top -\E_F[X]x^\top - x\E_F[X]^\top\\
&= (x - \E_F[X])(x- \E_F[X])^\top - V(F)~.
\end{split}
\end{align}
Hence the remainder term (under stationarity with $X_1 \sim F$) in equation \eqref{eq:meta-remainder} is given by
\begin{align}
R_{i,j}
&= V(\hat{F}_i^j) - V(F) - \dfrac{1}{j-i+1}\sum_{t=i}^j \IF(X_t, F, V)\nonumber\\
&= \int_{\R^d} xx^\top d\hat{F}_i^j(x) - \int_{\R^d} x d\hat{F}_i^j(x) \int_{\R^d}x^\top d\hat{F}_i^j(x)
- \dfrac{1}{j-i+1}\sum_{t=i}^j(X_i - \E[X_1])(X_i - \E[X_1])^\top\nonumber\\
&= - \int_{\R^d} x d\hat{F}_i^j(x) \int_{\R^d}x^\top d\hat{F}_i^j(x) + \int_{\R^d} x\E[X^\top_1] d\hat{F}_i^j(x)
+ \int_{\R^d} \E[X_1]x^\top d\hat{F}_i^j(x) - \E[X_1]\E[X_1^\top]\nonumber\\
&= - \bigg(\int_{\R^d} x- \E[X_1]d\hat{F}_i^j(x)\bigg)\bigg(\int_{\R^d} x^\top- \E[X^\top_1]d\hat{F}_i^j(x)\bigg)~. \label{exp:var}
\end{align}
Define $\vech(\cdot)$ to be the operator that stacks the columns of a symmetric $d \times d$-matrix above the diagonal as a vector of dimension $d(d+1)/2$.
As this operator is linear, it is obvious that expansion \eqref{exp:var} is equivalent to
\begin{align}\label{eq:vectorized}
\vech(R_{i,j})
&= \vech(V(\hat{F}_i^j)) - \vech(V) - \dfrac{1}{j-i+1}\sum_{t=i}^j \IF_v(X_t, F, V)~,
\end{align}
where $\IF_v$ is defined as
\begin{align} \label{eq:compsvar}
\IF_v(X_t, F, V)
&= \vech\big(\IF(X_t, F, V)\big)
= \IF\big(X_t, F, \vech(V)\big) \\
\nonumber  &  \\
\nonumber  &=
\begin{pmatrix}
(X_{t,1}-\E[X_{t,1}])^2 - \E\big[ (X_{t,1}-\E[X_{t,1}])^2\big]\\
(X_{t,1}-\E[X_{t,1}])(X_{t,2}-\E[X_{t,2}]) - \E\big[ (X_{t,1}-\E[X_{t,1}])(X_{t,2}-\E[X_{t,2}])\big]\\
(X_{t,2}-\E[X_{t,2}])^2 - \E\big[ (X_{t,2}-\E[X_{t,2}])^2\big]\\
(X_{t,1}-\E[X_{t,1}])(X_{t,3}-\E[X_{t,3}]) - \E\big[ (X_{t,1}-\E[X_{t,1}])(X_{t,3}-\E[X_{t,3}])\big]\\
(X_{t,2}-\E[X_{t,2}])(X_{t,3}-\E[X_{t,3}]) - \E\big[ (X_{t,2}-\E[X_{t,2}])(X_{t,3}-\E[X_{t,3}])\big]\\
(X_{t,3}-\E[X_{t,3}])^2 - \E\big[ (X_{t,3}-\E[X_{t,3}])^2\big]\\
\vdots
\end{pmatrix}~,
\end{align}
and  $X_{t,h}$ denotes the $h$-th component of the vector $X_t$.
We now provide sufficient conditions  such that the  general  theory in Section  \ref{sec3} is applicable for  the functional $\vech(V)$.
Assumption \ref{assump:meta-1} is satisfied if  the  time series $\{X_t\}_{t \in \Z}$ is stationary and the invariance principle
\begin{align}\label{conv:vardonsker}
\dfrac{1}{\sqrt{m}} \sum_{t=1}^{\floor{ms}} \IF_v(X_t, F, V)
\convd \sqrt{\Sigma_F} W(s)~,
\end{align}
holds in the space $\ell^{\infty}(\R,\R^{d^*})$,  where $W$ is a  $d^*=d(d+1)/2$-dimensional Brownian motion and $\Sigma_F$ is
defined in \eqref{eq:meta-covar} with $\IF = \IF_v$.
Invariance principles of the form \eqref{conv:vardonsker} are well known for many classes of weakly dependent time series.
The required assumptions for the underlying time series $\{X_t \}_{t \in \Z}$ are typically the same as for the mean - except for some extra moment conditions to cover the product structure of the random variables  in \eqref{eq:compsvar}.
Condition \eqref{eq:meta-remainderrate} in Assumption \ref{assump:meta-2}
reads as follows
\begin{align}\label{eq:remaindervar}
\sup_{1\leq i < j \leq n} \dfrac{1}{j-i+1}
\Big |\sum_{t=i}^j X_{t,k} - \E[X_{t,k}] \Big|
\Big|\sum_{t=i}^j X_{t,\ell} - \E[X_{t,\ell}] \Big|
= \op(n^{1/2})~
\end{align}
($1 \leq k,\ell \leq d^*$).
The validity of this assumption depends on the underlying dependence structure, in particular of the properties of the functions $g$ and $h$
in \eqref{eq:physicalsys}, and exemplarily we give sufficient conditions in the following result, which is proved in the online supplement.

\begin{Proposition}\label{prop:var}
Assume that \eqref{eq:physicalsys} holds with bounded functions $g$ and $h$ with 
$\delta_{t,4}^{(1)} = \mathcal{O}(\rho^t),~ \delta_{t,4}^{(2)} = \mathcal{O}(\rho^t)$
for some $\rho \in (0,1)$.
Let $\hat{\Sigma}_m$ denote a consistent estimator of the long-run variance $\Sigma_{F^{(1)}}$ (before the change) based on the observations $X_1,\dots,X_m$.
Further assume, that the covariance matrices $\Gamma(g)$ and $\Gamma(h)$ defined in \eqref{eq:ordcovar} are positive definite.
\begin{itemize}
\item[(a)]
If $g=h$, then the assumptions of Theorem \ref{cor:meta-lvl1} and \ref{cor:meta-lvl1SN} are satisfied
for the functional \eqref{eq:varfunc}.
In other words: The sequential tests for a change in the variance based on the statistics $\hatD$ or  $\hatDSN$ have asymptotic level $\alpha$.
\item[(b)]
If $ h =  A \cdot g$ for some non-singular matrix $A \in \R^{d \times d}$ with $A\cdot V(F^{(1)}) \cdot A^\top \neq V(F^{(1)})$, then the assumptions of Theorem \ref{thm:meta-power} and \ref{thm:meta-powerSN} are satisfied for the variance functional \eqref{eq:varfunc}.
In other words: The sequential tests for a change in the variance based on the statistics $\hatD$ or $\hatDSN$ are consistent.
\end{itemize}
\end{Proposition}

\begin{remark}
The assumption of bounded observations in Proposition \ref{prop:var} is crucial to prove  the estimate \eqref{eq:remaindervar}.
Essentially a proof of such a statement requires a version of Theorem 1 in \cite{Shao1995} for dependent random variables.
The main ingredient for a proof of Shao's result is an Erd\"os-Renyi-Law of large numbers in the case of dependent random variables, which - to the authors best knowledge - is only known for bounded random variables [see \cite{Kifer2017} for example].
On the other hand the assumption of bounded functions $g$ and $h$ in \eqref{eq:physicalsys} is not necessary for the functional \eqref{eq:varfunc} in the case of $M$-dependent time series.
\end{remark}

\subsection{Quantiles}
In this section we consider the quantile functional
\begin{align}\label{eq:funcquant}
\varphi_{\beta}(F) = F^{-}(\beta) := \inf \{ x \in \R\,|\, F(x) \geq \beta\}~,
\end{align}
where $\beta \in (0,1)$ is fixed and $F^{-}$ denotes the quantile function (or sometimes called generalized inverse function) for a distribution function $F$. 
Further recall the notation $\hatFij$ defined in \eqref{eq:Fij} for the empirical distribution function based on the subset of observations $X_i,\dots,X_j$.
For the sake of readability, we will simplify  the notation and denote the true and empirical quantiles by
\begin{align*}
q_{\beta} := \varphi_\beta(F)  = F^{-}(\beta)
\;\;\;
\text{and}
\;\;\;
\hatqijbeta :=\varphi_\beta(\hatFij) = (\hatFij)^{-}(\beta)~,
\end{align*}
respectively.
Considering a twice differentiable distribution function $F$ having derivative $f$ with $f(q_\beta)>0$, a straightforward but tedious calculation yields that the influence functional for $\varphi_{\beta}$ is given by
\begin{align*}
\IF(x, \varphi_{\beta},F) 
= \dfrac{\beta - I\{x \leq q_\beta\}}{f(q_\beta)}
=\begin{cases}
\dfrac{\beta - 1}{f(q_\beta)} \;\;\text{if}\;\; x \leq q_\beta~,\\[12pt]
\dfrac{\beta}{f(q_\beta)} \;\;\text{if}\;\; x > q_\beta~
\end{cases}
\end{align*}
[see, for example, \cite{Wasserman2010} for a proof of this statement under slightly stronger conditions].
This yields that the linearization in \eqref{eq:meta-remainder} for $\varphi_\beta$ is given by
\begin{align*}
\varphi_\beta(\hatFij) - \varphi_\beta(F)
= \hatqijbeta - q_{\beta}
&= \dfrac{1}{j-i+1}\sum_{t=i}^j \IF(X_t, F, \varphi_\beta) -  R_{i,j}\\
&= \dfrac{\beta - \hatFij(q_\beta)}{f(q_\beta)}  - R_{i,j}~,
\end{align*}
with remainder terms 
\begin{align*}
R_{i,j} = q_{\beta} - \hatqijbeta  - \dfrac{\beta - \hatFij(q_\beta)}{f(q_\beta)}~.
\end{align*}
\noindent The linearization stated above is also known as Bahadur expansion [see \cite{Bahadur1966}] and the investigation of the order of the remainder terms has been a major research topic.
In this section, we restrict ourselves to the case of independent observations, which makes the arguments less technical
[see the discussion in the online supplement].
Given independent and identically distributed observations, an application of Donsker's theorem immediately shows that Assumption \ref{assump:meta-1} is satisfied for $\varphi_\beta$, where it suffices to use the bound
\begin{align*}
| \IF(X_t, \theta_{\beta},F)|
= \bigg| \dfrac{\beta - I\{X_t \leq q_\beta\}}{f(q_\beta)}\bigg|
\leq 2/f(q_\beta)~.
\end{align*}
Establishing Assumption \ref{assump:meta-2} is substantially more complicated.
To the author's best knowledge uniform estimates for the remainder terms of the Bahadur expansion of the form \eqref{eq:meta-remainderrate} have not been investigated in the literature.
In the online supplement we prove the following  result, which is of independent interest.
\begin{theorem}\label{thm:quantremain}
Let $\{X_t\}_{t \in \Z}$ be a sequence of i.i.d. random variables with distribution function $F$, which is twice differentiable with 
\begin{align}\label{ineq:tail}
\sup_{x \in \mathbb{R}} \big(f(x) + |f'(x)|\big) < \infty
\;\;\;
\text{and}
\;\;\;
\Pb \big( |X_1| > x \big) \lesssim x^{-\lambda} 
\end{align}
for fixed $\lambda > 18/5$.
Further let $\beta \in (0,1)$ be fixed and assume that $f(q_{\beta})>0$.
It holds that
\begin{align*}
\dfrac{1}{\sqrt{n}} \max_{1\leq i < j \leq n} (j-i+1) | R_{i,j}|
= \dfrac{1}{\sqrt{n}} \max_{1\leq i < j \leq n} (j-i+1) \bigg| \hatqijbeta - q_\beta - \dfrac{\beta - \hatFij(q_\beta)}{f(q_\beta)} \bigg|
= \op(1)~.\\
\end{align*}
\end{theorem}
\noindent The assumption on the tails of the distribution of $|X_1|$ is crucial to control the error of the quantile estimators in case of small sample sizes, see Lemma \ref{lem:quantlem} in the online supplement.
The theorem above directly implies the following corollary.
\begin{corollary}\label{cor:quantiles}
Let $\{X_t\}_{t \in \Z}$ be a sequence of independent random variables and $F^{(1)}, F^{(2)}$ distributions functions fulfilling the assumptions of Theorem \ref{thm:quantremain}.
Assume that for a constant $1<c<T$
\begin{align}\label{eq:quantchange}
X_{t} \sim 
\begin{cases}
F^{(1)}\;\;\text{if}\;\; t < \floor{mc}~,\\
F^{(2)}\;\;\text{if}\;\; t \geq \floor{mc}~.
\end{cases}
\end{align}
Further let $\hat{\Sigma}_m$ denote a consistent estimator of the long-run variance $\Sigma_{F^{(1)}}$ (before a possible change) based on the observations $X_1,\dots,X_m$.
\begin{itemize}
\item[(a)]
If $F^{(1)} = F^{(2)}$, then the assumptions of Theorem \ref{cor:meta-lvl1} and \ref{cor:meta-lvl1SN} are satisfied for the functional \eqref{eq:funcquant}.
In other words: The sequential tests for a change in the $\beta$-quantile based on the statistics $\hatD$ or  $\hatDSN$ have asymptotic level $\alpha$.
\item[(b)]
If $\theta_{\beta}(F^{(1)}) \neq \theta_{\beta}(F^{(2)})$, then the assumptions of Theorem \ref{thm:meta-power} and \ref{thm:meta-powerSN} are satisfied for the functional \eqref{eq:funcquant}.
In other words: The sequential tests for a change in the $\beta$-quantile based on the statistics $\hatD$ or $\hatDSN$ are consistent.
\end{itemize}
\end{corollary}

\section{Finite sample properties}\label{sec5}
In this section, we investigate the finite sample properties of the new detection schemes  based on the statistics $\hat{D}$ and $\hat{D}^{\rm SN}$ in \eqref{genstat} and \eqref{genstatSN} and also provide a comparison to the detection schemes based on $\hat{Q}$, $\hat{P}$ and $\hat{P}^{\rm SN}$,
which are defined in \eqref{scheme:wied}, \eqref{scheme:kirch} and \eqref{scheme:kirchsn}, respectively.
For the choice of the threshold function we follow the ideas of \cite{Horvath2004}, \cite{Aue2009} and \cite{Wied2013} and consider the parametric family
\begin{align}\label{eq:thgamma}
w(t) = (t+1)^2 \cdot \max \bigg\{ \bigg(\dfrac{t}{t+1} \bigg)^{2\gamma} , \delta \bigg\}~,
\end{align}
where the parameter $\gamma$ varies in  the interval $[0,1/2)$ and $\delta>0$ is a small constant introduced to avoid problems in the denominator of the ratio  considered in \eqref{ineq:threshold}.
For the statistic  $\hat{Q}$ these threshold functions are motivated by the law of  iterated logarithm and are used to reduce the stopping delay under the alternative hypothesis [see \cite{Aue2009} or \cite{Wied2013}].

Note that we use the squared versions of the thresholds from the  cited references, since we consider statistics in terms of quadratic forms. To be precise consider three different threshold functions
\begin{enumerate}[(T1)]
\item $w(t) = c_{\alpha}$~,
\item $w(t) = c_{\alpha} (t+1)^2$~,
\item $w(t) = c_{\alpha} (t+1)^2 \cdot \max \Big\{ \Big(\dfrac{t}{t+1} \Big)^{1/2} , 10^{-10} \Big\}$~,
\end{enumerate}
where the constant $c_{\alpha}$ is chosen by Monte-Carlo simulations, such that
\begin{align}\label{eq:calpha}
\Pb \bigg( \sup_{t \in [1,T+1]} \dfrac{L(t)}{w(t-1)}  > 1 \bigg) = \alpha~.
\end{align}
Here $L$ denotes the limit process corresponding to the test statistic under consideration.
The reader should be aware that occasionally in the literature, authors define the threshold function without $c_{\alpha}$ and then name the product $c_\alpha \cdot w(t)$ the critical curve (for test level $\alpha$).

For the estimation of the long-run variance of a one-dimensional time series we use the well-known quadratic spectral kernel [see \cite{Andrews1991}], that is
\begin{align*}
\hat{\sigma}^2
= \hat{\gamma}_0 + 2\sum_{i=1}^{m-1} k \Big(\dfrac{i}{b_m} \Big) \hat{\gamma}_i~,
\end{align*}
where $\hat{\gamma}_i$ denotes the empirical lag $i$ autocovariance of $X_1,\dots,X_m$ and $b_m$ is the bandwidth for the underlying kernel $k$ given by
\begin{align*}
k(x) = \dfrac{25}{12\pi^2x^2} \bigg( \dfrac{ \sin(6\pi x/5)}{6\pi x/5} - \cos(6\pi x/5) \bigg)~.
\end{align*}
For multivariate data, we use its canonical extension replacing the estimated autocovariances $\hat{\gamma}$ by its corresponding multivariate counterparts. 
In our simulations we use the implementation of this estimator contained in the R-package 'sandwich' [see \cite{Zeileis2004}].
As mentioned before, we only use the data from the stable subset $X_1,\dots,X_m$ for the long-run variance estimate, while the bandwidth is chosen as $b_m = \log_{10}(m)$, which corresponds to the rule proposed in \cite{Aue2009b}.
Note that the long-run variance estimator is the same for the non self-normalized procedures $\hat{D}, \hat{P}, \hat{Q}$.
For the sake of brevity, we will only display situations where the parameter $T$ is fixed as $T=1$, i.e. the monitoring period will always have the same size as the historical data set.
Further the change will always occur at the center of the monitoring period, which is $k^* = m/2$.
For the case $T\geq 2$ and other change locations we obtained a similar picture.
The results can be found in Section \ref{addsim} of the online supplement.
All results that are presented here and in the online supplement are based on 5000 independent simulation runs.

\subsection{Changes in the mean}\label{sec51}

For the analysis of the new procedures in the problem of detecting changes in the mean we look at independent data, a MA(2)- and two AR(1)-processes, defined by
\begin{enumerate}[(M1)]
\item $X_{t} \sim \varepsilon_t$~,
\item $X_{t} = 0.1X_{t-1} + \varepsilon_t$~,
\item $X_{t} = \varepsilon_t + 0.3\varepsilon_{t-1} - 0.1\varepsilon_{t-2}$~,
\item $X_t = 0.3X_{t-1} + e_t$~,
\end{enumerate}
where $\{\varepsilon_t\}_{t\in \mathbb{Z}}$ and $\{e_t\}_{t\in \mathbb{Z}}$ are sequences of independent standard Gaussian and $\exp(1)$ distributed random variables, respectively.
In the case of the alternative hypothesis we consider the sequence
$$
X_t^\mu =
\begin{cases}
\;\;X_t  &\mbox{if }~ t < m + \lfloor {m \over 2} \rfloor \\
X_t + \mu & \mbox{if }~ t \geq  m + \lfloor {m \over 2} \rfloor
\end{cases}
$$
for various values of $\mu$.
For all discussed detection schemes the empirical rejection probabilities for the  models (M1) - (M4) and threshold functions (T1) - (T3) are shown in Figure \ref{fig:2} and \ref{fig:3} corresponding to the choice $m=50$ and $m=100$ as initial sample size.
The results can be summarized as follows.
The statistic $\hat{D}$ outperforms $\hat{P}$ and $\hat{Q}$ with respect to the power for all combinations of the  model and threshold function.
Further the statistic $\hat{P}$ shows a better performance with respect to power as $\hat{Q}$ in all cases under consideration.
For example, the plot in the left-upper corner of Figure \ref{fig:2} shows, that $\hat{D}$ already has empirical power close to $1$ ($0.95$) for a change of
size $1$, while $\hat{P}$ and $\hat{Q}$ have only empirical power of $0.84$ and $0.71$, respectively.
This relation is basically the same in all plots contained in Figures \ref{fig:2} and \ref{fig:3}.

In Figure \ref{fig:rejTime} we present the average rejection time of the different procedures as a function of the size of the  change for the  different choices  (T1) - (T4) of the threshold function under model (M1).
For its computation we ignore runs without a rejection and with a rejection before the actual change.
The results basically show that the rejections occur earlier when using threshold (T3) or (T2) instead of the constant threshold (T1).
The threshold (T3) also yields slightly earlier rejections compared to threshold (T2).
This effect may be caused by a value  of $\gamma$ in the parametric family in \eqref{eq:thgamma} ($\gamma=0$ and $\gamma=0.25$), which corresponds to the observations made by \cite{Horvath2004} or \cite{Wied2013}, where a more detailed discussion is given.
 Basically, the plots illustrate a decreasing rejection with an increasing size of  the change provided that this is larger than $0.5$. 
This corresponds to intuition.
The slight decrease in the empirical time of rejection for small values of $\mu$ can be explained by the fact that this case is close to the null hypotheses.
As a consequence the rejection times are more uniformly distributed, with a greater portion close to the time $m+k^*$.

For the sake of brevity we will only consider the constant threshold (T1) in the remaining part of this paper (note also that the
results in Figure \ref{fig:2} and \ref{fig:3} show no substantial differences between the different threshold functions).
In Figure \ref{fig:4} we display the power of the tests based on the self-normalized statistics $\hat{D}^{\rm SN}$ and $\hat{P}^{\rm SN}$ and non self-normalized statistics $\hat{D}$ and $\hat{P}$.
The results are similar as before and the empirical power obtained by the use of $\hat{D}^{\rm SN}$ is considerably higher than those of $\hat{P}^{\rm SN}$.
Further a comparison between the results of $\hat{D}^{\rm SN}$ and $\hat{P}^{\rm SN}$ to those of $\hat{D}$ and $\hat{P}$ indicates, that self-normalization yields a substantial loss of power in the sequential detection schemes.

On the other hand the approximation of the nominal level is more stable with respect to different dependence structures for self-normalized methods.
To illustrate this fact we display in Table \ref{table:1} the type-I error for all five sequential monitoring schemes based on the statistics  $\hat D$, $\hat P$ and $\hat Q$.
The results provide some empirical evidence that the self-normalized statistics yield a more stable approximation of the nominal level.
In particular for model (M4), which has a stronger dependence structure, the approximation of the self-normalized methods is clearly superior.
This effect is even more visible for stronger dependencies (these results are not displayed for the sake of brevity).

As pointed out by a reviewer, the method $\hat{D}$ and the self-normalized methods $\hat{D}^{SN}, \hat{P}^{SN}$ should require higher computational effort.
To offer a comparison of the computational complexity, we illustrate the run time of the different procedures in Table \ref{tab:runtime}, where we display the CPU-time of each procedure for for one run.
We observe that the computation of the statistic $\hat{D}$ is more expensive than those of $\hat{P}$ since it requires updating $\hat \theta_1^{m+j}$ and $\hat \theta_{m+j+1}^{m+k}$ for all possible change positions $j \in \{0,\dots,k-1\}$ when a new observation $X_{m+k+1}$ arrives.
In contrast to this, $\hat P$ only requires renewing $\hat \theta_{m+j+1}^{m+k}$ (for all possible choices of $j$) while the historical estimate $\hat \theta_1^m$ is fixed.
Note that the procedure based on $\hat Q$ is the fastest procedure since there is no maximum included and only the estimation $\hat \theta_m^{m+k}$ has to be updated when moving to the next observation $X_{m+k+1}$.
However, one has to keep in mind that the advantage of $\hat Q$ and $\hat P$ over $\hat D$ 
with respect to computation time comes with the price of  power inferiority described above.
As expected, the computation time for the self-normalized procedures $\hat P$ and $\hat D$ is larger, as the process $\mathbb{V}$ defined in \eqref{eq:meta-snprocess} is of a more complicated structure.
However, all procedures are quite fast, such that the observed time disparities might have an impact in a large simulation study but will not be of importance in practice when a data set has to be evaluated.

\begin{figure}[H]
\begin{tabular}{b{1cm}ccc}
& (T1) & (T2) & (T3)\\
(M1) \tabj &
\includegraphics[width=4.5cm,height=3.5cm]{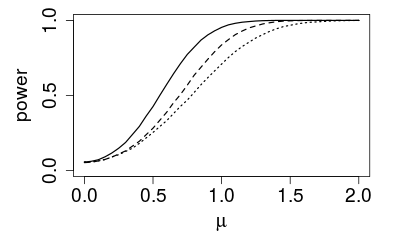} &
\includegraphics[width=4.5cm,height=3.5cm]{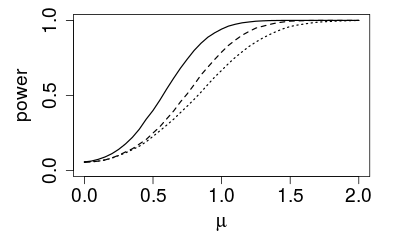} &
\includegraphics[width=4.5cm,height=3.5cm]{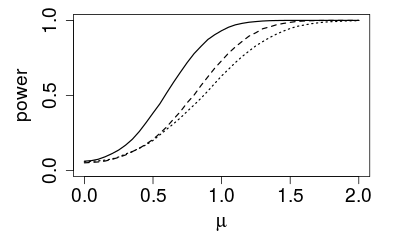} \vspace{-0.3cm}
\\
(M2) \tabj &
\includegraphics[width=4.5cm,height=3.5cm]{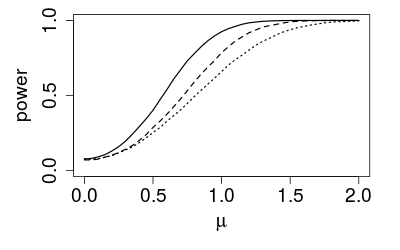} &
\includegraphics[width=4.5cm,height=3.5cm]{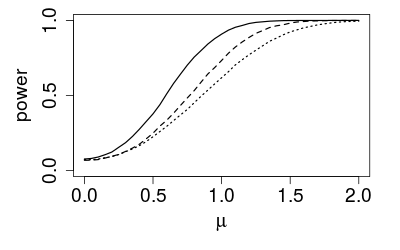} &
\includegraphics[width=4.5cm,height=3.5cm]{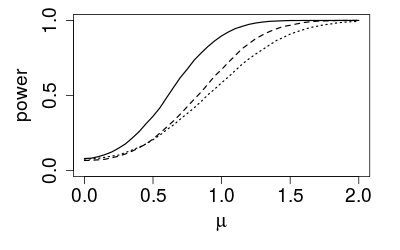} \vspace{-0.3cm}
\\
(M3) \tabj &
\includegraphics[width=4.5cm,height=3.5cm]{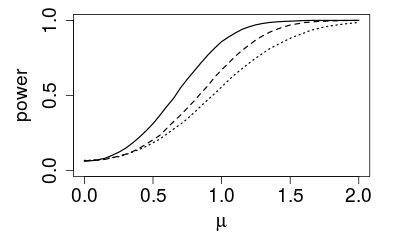} &
\includegraphics[width=4.5cm,height=3.5cm]{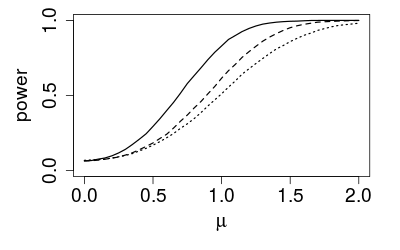} &
\includegraphics[width=4.5cm,height=3.5cm]{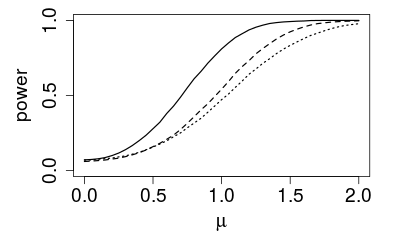} \vspace{-0.3cm}
\\
(M4) \tabj &
\includegraphics[width=4.5cm,height=3.5cm]{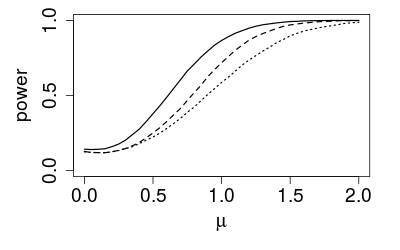} &
\includegraphics[width=4.5cm,height=3.5cm]{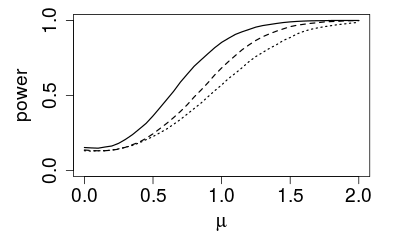} &
\includegraphics[width=4.5cm,height=3.5cm]{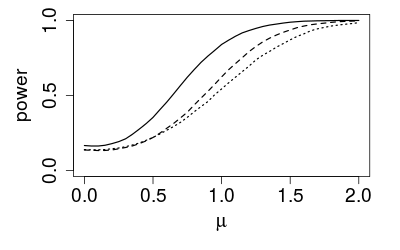} \vspace{-0.3cm}
\end{tabular}
\caption{\it Empirical rejection probabilities of the sequential tests for a change in the mean
based on the statistics $\hat{D}$ (solid line), $\hat{P}$ (dashed line) , $\hat{Q}$ (dotted line).
The initial and total sample size are $m=50$ and $m(T+1)=100$, respectively, and
the change occurs at observation $75$.
The level is $\alpha=0.05$.
Different rows correspond to different models, while different columns correspond to different threshold functions.\label{fig:2}}
\end{figure}

\begin{figure}[H]
\begin{tabular}{b{1cm}ccc}
& (T1) & (T2) & (T3)\\
(M1) \tabj &
\includegraphics[width=4.5cm,height=3.5cm]{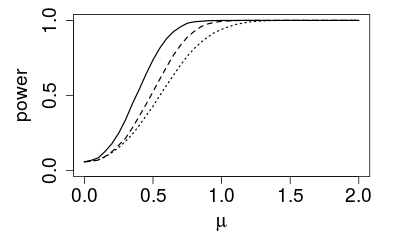} &
\includegraphics[width=4.5cm,height=3.5cm]{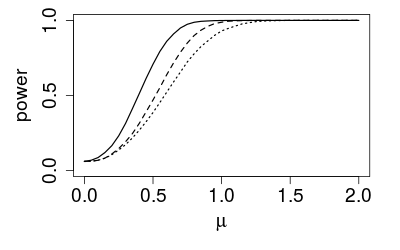} &
\includegraphics[width=4.5cm,height=3.5cm]{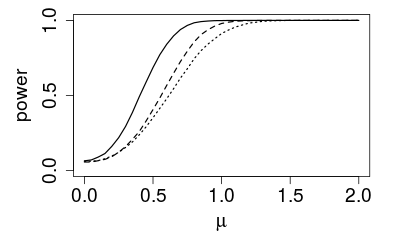} \vspace{-0.3cm}
\\
(M2) \tabj &
\includegraphics[width=4.5cm,height=3.5cm]{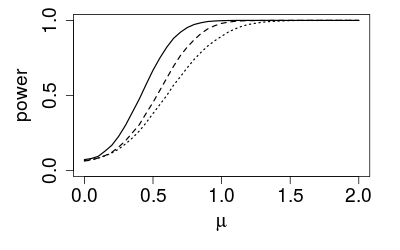} &
\includegraphics[width=4.5cm,height=3.5cm]{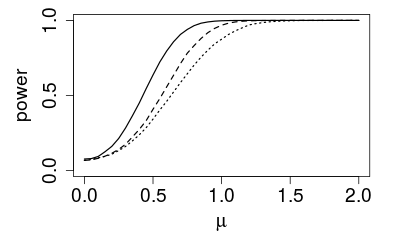} &
\includegraphics[width=4.5cm,height=3.5cm]{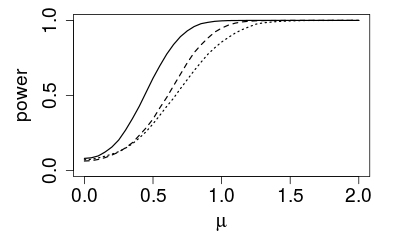} \vspace{-0.3cm}
\\
(M3) \tabj &
\includegraphics[width=4.5cm,height=3.5cm]{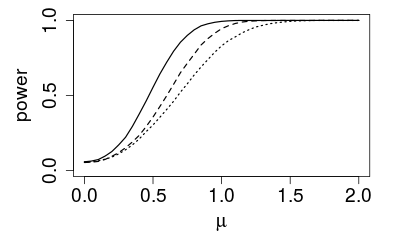} &
\includegraphics[width=4.5cm,height=3.5cm]{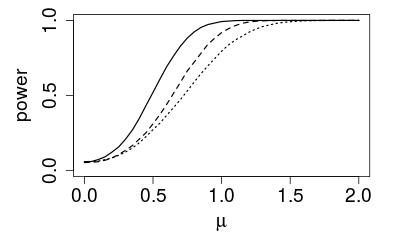} &
\includegraphics[width=4.5cm,height=3.5cm]{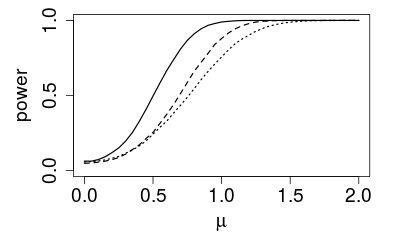} \vspace{-0.3cm}
\\
(M4) \tabj &
\includegraphics[width=4.5cm,height=3.5cm]{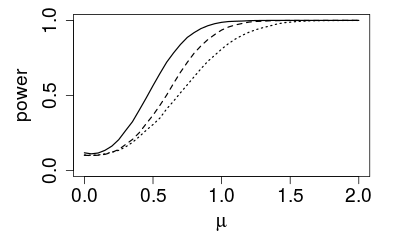} &
\includegraphics[width=4.5cm,height=3.5cm]{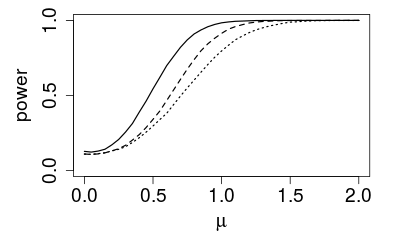} &
\includegraphics[width=4.5cm,height=3.5cm]{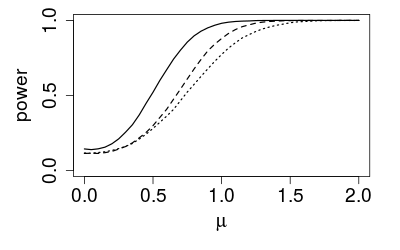} \vspace{-0.3cm}
\end{tabular}
\caption{\it Empirical rejection probabilities of the sequential tests
for a change in the mean based on the statistics $\hat{D}$ (solid line), $\hat{P}$ (dashed line), $\hat{Q}$ (dotted line).
The initial and total sample size are $m=100$ and $m(T+1)=200$, respectively, and
the change occurs at observation $150$.
The level is $\alpha=0.05$.
Different rows correspond to different models, while different columns correspond to different threshold functions.
\label{fig:3}}
\end{figure}

\begin{figure}[H]
\begin{tabular}{b{1.54cm}ccc}
& (T1) & (T2) & (T3)\\
$m=50$ $T=1$ \tabj &
\includegraphics[width=4.5cm,height=3.5cm]{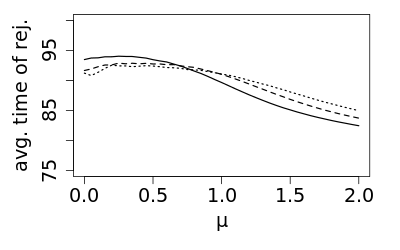} &
\includegraphics[width=4.5cm,height=3.5cm]{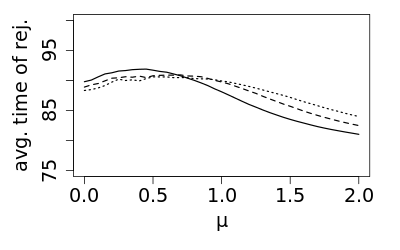} &
\includegraphics[width=4.5cm,height=3.5cm]{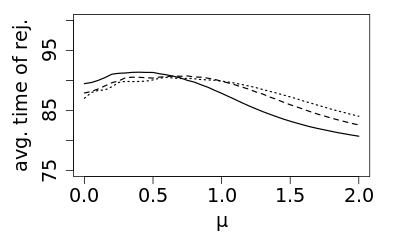} \\
$m=100$ $T=1$ \tabj &
\includegraphics[width=4.5cm,height=3.5cm]{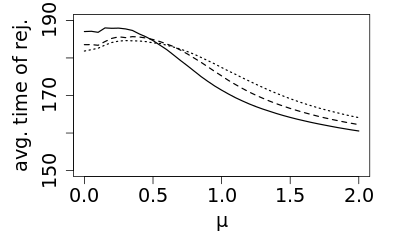} &
\includegraphics[width=4.5cm,height=3.5cm]{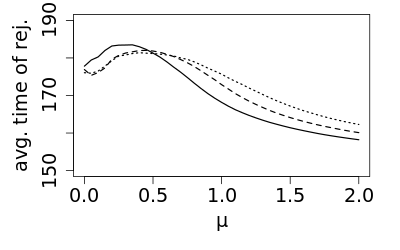} &
\includegraphics[width=4.5cm,height=3.5cm]{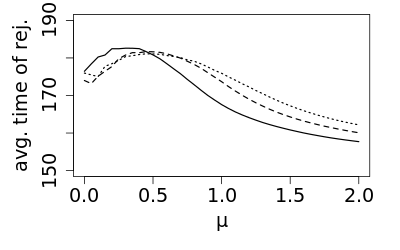} \\
\end{tabular}
\caption{\it Empirical time of rejection of the sequential tests for a change in the mean based on the statistics $\hat{D}$ (solid line), $\hat{P}$ (dashed line), $\hat{Q}$ (dotted line).
The model is chosen as (M1). 
The change occurs at observation $75$ (first row) and $150$ (second row), respectively.
The level is $\alpha=0.05$.
Different columns correspond to different threshold functions.
\label{fig:rejTime}}
\end{figure}

\begin{figure}[H]
\centering
\begin{tabular}{b{1cm}ccc}
& $m = 50,\;\;T=1$ & $m=100,\;\;T=1$\\
(M1) \tabj &
\includegraphics[width=4.5cm,height=3.5cm]{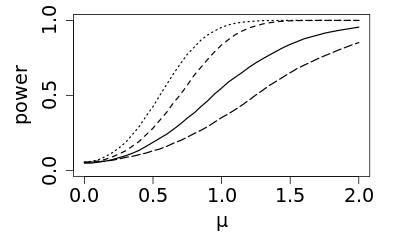} &
\includegraphics[width=4.5cm,height=3.5cm]{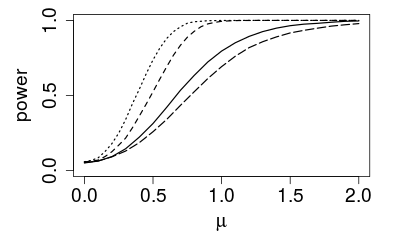}
\\
(M2) \tabj &
\includegraphics[width=4.5cm,height=3.5cm]{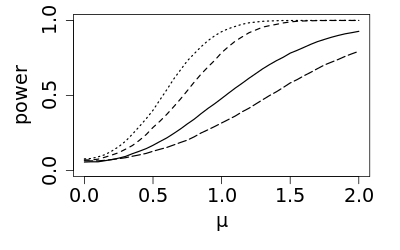} &
\includegraphics[width=4.5cm,height=3.5cm]{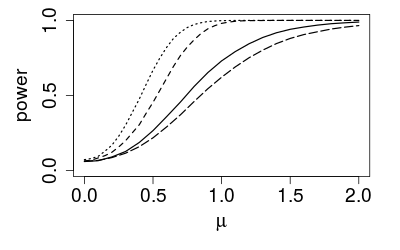}
\\
(M3) \tabj &
\includegraphics[width=4.5cm,height=3.5cm]{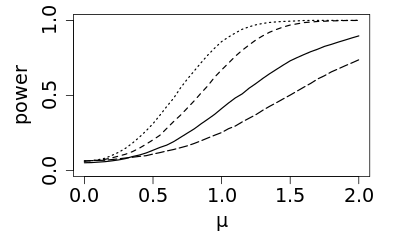} &
\includegraphics[width=4.5cm,height=3.5cm]{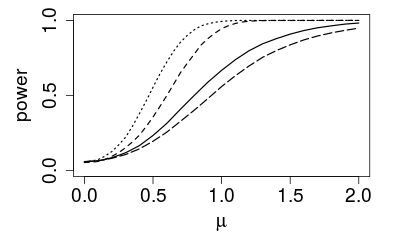}
\\
(M4) \tabj &
\includegraphics[width=4.5cm,height=3.5cm]{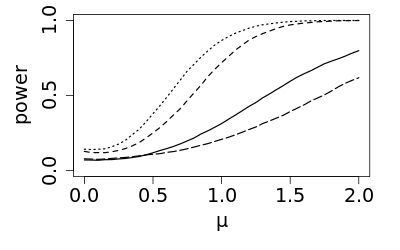} &
\includegraphics[width=4.5cm,height=3.5cm]{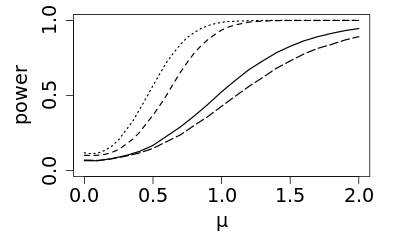}
\\
\end{tabular}
\caption{\it Empirical rejection probabilities of the sequential tests for a change in the mean based on the self-normalized statistics $\hat{D}^{\rm SN}$ (solid line), $\hat{P}^{\rm SN}$ (long dashed line) compared to the non self-normalized statistics $\hat{D}$ (dotted line) and $\hat{P}$ (dashed lined). 
The initial and total sample size are $m=50$ and $m(T+1)=100$ (upper panel, change at observation $75$) and $m=100$ and $m(T+1)=200$ (lower panel, change at observation $150$).
The level is $\alpha=0.05$ and the threshold function is given by {\rm (T1)}.
\label{fig:4}}
\end{figure}

\begin{table}[H]
\centering
\begin{tabular}{|c|c|c|c|c|c|c|c|}
\hline
$m$ & model \textbackslash \;statistic & $\hat{D}$ & $\hat{P}$ & $\hat{Q}$ & $\hat{D}^{\rm SN}$ & $\hat{P}^{\rm SN}$ \\
\hline
\multirow{3}{*}{$50$}
& (M1) &  5.6\% &  5.3\% &  5.8\% & 5.0\% & 5.7\%\\
\cline{2-7}
& (M2) &  7.8\% &  7.1\% &  7.6\% & 5.7\% & 6.5\%\\
\cline{2-7}
& (M3) &  6.4\% &  6.3\% &  6.8\% & 5.2\% & 6.3\%\\
\cline{2-7}
& (M4) & 14.2\% & 12.8\% & 12.5\% & 7.0\% & 7.8\%\\
\hline
\hline
\multirow{3}{*}{$100$}
& (M1) &  5.9\% &  5.8\% &  5.9\% & 5.1\% & 5.7\%\\
\cline{2-7}
& (M2) &  7.3\% &  6.4\% &  6.6\% & 6.0\% & 6.4\%\\
\cline{2-7}
& (M3) &  5.8\% &  5.4\% &  5.7\% & 5.9\% & 6.1\%\\
\cline{2-7}
& (M4) & 11.8\% & 10.2\% & 10.3\% & 6.6\% & 7.1\%\\
\hline
\end{tabular}
\caption{\it Simulated type I error (level $\alpha=0.05$) of the sequential tests for a change in the mean based on the statistics $\hat{D}$, $\hat{P}$, $\hat{Q}$, $\hat{D}^{SN}$ and $\hat{P}^{SN}$.
The threshold function is (T1) and the factor $T$ is again set to $T=1$.\label{table:1}}
\end{table}

\begin{table}[H]
\centering
\begin{tabular}{|c|c|c|c|}
\hline 
$\hat{D}$ & 0.0112 sec & $\hat{D}^{SN}$ & 0.3194 sec\\ 
\hline 
$\hat{P}$ & 0.0080 sec & $\hat{P}^{SN}$ & 0.3173 sec\\ 
\hline 
$\hat{Q}$ & 0.0075 sec & &\\ 
\hline 
\end{tabular} 
\caption{\it Computation time for one simulation run of the procedures $\hat D, \hat P, \hat Q, \hat{D}^{SN}, \hat{P}^{SN}$ 
The scenario is the same as in the left column of Figure \ref{fig:4}.\label{tab:runtime}}
\end{table}

\subsection{Changes in the variance}\label{sec52}
In this subsection we present a small simulation study investigating the performance of the detection schemes for a change in the variance matrix.
We consider the following models
\begin{enumerate}[(V1)]
\item $X_{t} = \varepsilon_{t}$~,
\item $X_{t} = A_1X_{t-1} + \varepsilon_t$~,
\item $X_{t} = \varepsilon_t + A_2\varepsilon_{t-1} + A_3\varepsilon_{t-2}$~,
\item $X_{t} = A_4X_{t-1} + \varepsilon_{t}$~,
\end{enumerate}
where $\{\varepsilon_t\}_{t \in \mathbb{Z}}=\{(\varepsilon_{t,1},\dots,\varepsilon_{t,d})^\top\}_{t \in \mathbb{Z}}$ denotes an i.i.d. sequence of centered $d$-dimensional Gaussian distributed random variables $d$ is chosen accordingly to the dimension of the involved matrices, which are defined by
\begin{align}\label{eq:matrixsim}
\begin{split}
A_1 &=  \begin{pmatrix}
0.2 & 0.1 \\
0.1 & 0.2
\end{pmatrix},\quad
A_2 =  \begin{pmatrix}
0.3 & 0.1 \\
0.1 & 0.3
\end{pmatrix},\quad
A_3  = \begin{pmatrix}
0.1  & 0.05 \\
0.05 & 0.1
\end{pmatrix}~,\\[10pt]
A_4 &= \begin{pmatrix}
0.1  & 0.05 & 0.05\\
0.05 & 0.1  & 0.05\\
0.05 & 0.05 & 0.1
\end{pmatrix}~.
\end{split}
\end{align}
For the alternative, we proceed similarly as in \cite{Aue2009b} and define
\begin{align}\label{eq:altvar}
\Cov(\varepsilon_t,\varepsilon_t)
= \varepsilon_t\cdot \varepsilon_t^\top
= \begin{cases}
I_d\; &\text{if}\; t \leq m + \lfloor {m \over 2} \rfloor~\\
I_d + \delta \cdot I_d\;&\text{if}\; t > m + \lfloor {m \over 2} \rfloor
\end{cases}~,
\end{align}
where $I_d$ denotes the $d$-dimensional identity matrix (the case $\delta=0$ corresponds to the null hypothesis of no change).
For the sake of brevity we will focus on the non-self-normalized statistics $\hat{D}$, $\hat{P}$ and $\hat{Q}$ here.
In Figure \ref{fig:6} we display the empirical power for the three data generating processes and the threshold function (T1).
The results are similar to those presented in Section \ref{sec51}.
The test based on the statistic $\hat{D}$ is more powerful than the tests based on $\hat P$ and $\hat Q$.
It should be mentioned that the approximation of the nominal level is less accurate  for  model (V4).
\begin{figure}
\centering
\begin{tabular}{b{1cm}ccc}
& (V1) & (V2) \\
(T1) \tabj &
\includegraphics[width=4.5cm,height=3.5cm]{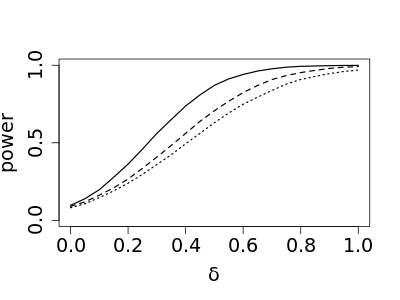} &
\includegraphics[width=4.5cm,height=3.5cm]{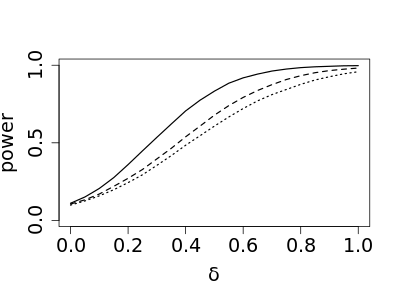} \\
& (V3) & (V4) \\
(T1) \tabj & 
\includegraphics[width=4.5cm,height=3.5cm]{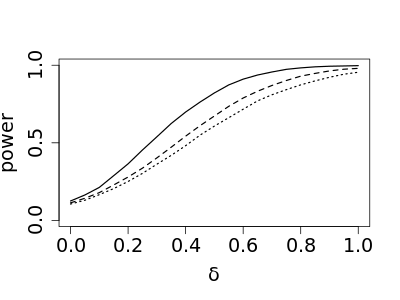} & 
\includegraphics[width=4.5cm,height=3.5cm]{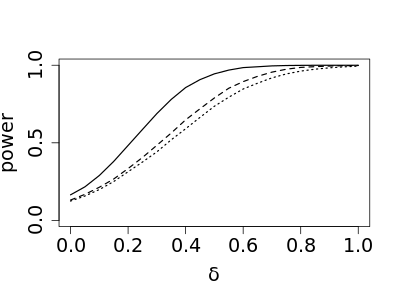} \\
\end{tabular}
\caption{\it Empirical rejection probabilities of the sequential tests for a change in the variance matrix based on the statistics $\hat{D}$ (solid line), $\hat{P}$ (dashed line), $\hat{Q}$ (dotted line).
The initial and total sample size are $m=200$ and $m(T+1)=400$, respectively, and the change occurs at observation $300$.
The level is $\alpha=0.05$.
\label{fig:6} }
\end{figure}

\subsection{Changes in the correlation} \label{sec53}
We conclude this paper with a brief empirical comparison of the three methods for the detection of a change in the correlation, which has been considered in \cite{Wied2013}.
For the sake of brevity we do not provide a detailed proof that the assumptions of Section \ref{sec3} are satsified, but restrict ourselves to the numerical comparison.
For the definition of the data generating processes, we use the models (V1) - (V3) introduced in Section \ref{sec52} but with a different process $\{\varepsilon_j\}_{j \in \mathbb{Z}} = \{ (\varepsilon_{j,1}, \varepsilon_{j,2})^\top\}_{j \in \mathbb{Z}}$.
In this section $\{\varepsilon_j\}_{j \in \mathbb{Z}}$ is a sequence of independent two-dimensional Gaussian random variables such that
\begin{align*}
\Cor(\varepsilon_{j,1}, \varepsilon_{j,2})
= \begin{cases}
c_1 \quad \text{if} \quad j \leq m + \lfloor {m \over 2} \rfloor~,\\
c_2 \quad \text{if} \quad j > m + \lfloor {m \over 2} \rfloor
\end{cases}
\end{align*}
and $\Var(\epsilon_{j,1}) = \Var(\epsilon_{j,2}) = 1$.
We use $c_1=0.3$ for the correlation before the change and consider different values of $c_2$.
For estimation of the long-run variance matrix we use the estimator proposed in \cite{Wied2013} (the explicit formula for the estimator is given in the appendix of the referenced paper and omitted here for the sake of brevity).
Figure \ref{fig:corr} now compares the power of the non-self-normalized methods for the three models defined above.
As in the previous sections the sequential detection scheme based on $\hat{D}$ yields substantially better results.

\begin{figure}
\begin{tabular}{b{1cm}ccc}
& (V1) & (V2) & (V3)\\
(T1) \tabj &
\includegraphics[width=4.5cm,height=3.5cm]{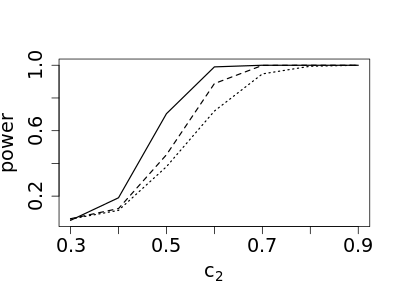} &
\includegraphics[width=4.5cm,height=3.5cm]{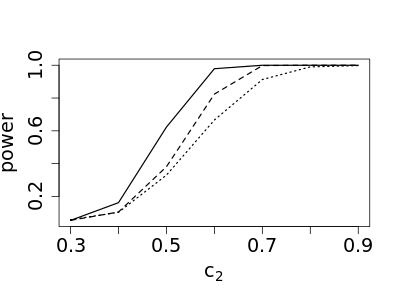} &
\includegraphics[width=4.5cm,height=3.5cm]{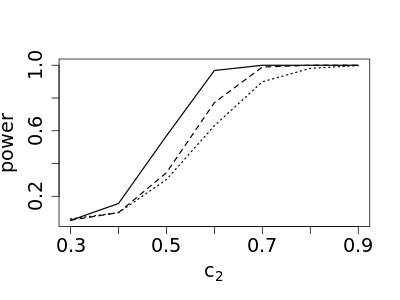} \\
\end{tabular}
\caption{\it Empirical rejection probabilities of the sequential tests for a change in the correlation based on the statistics $\hat{D}$ (solid line), $\hat{P}$ (dashed line) , $\hat{Q}$ (dotted line).
The initial and total sample size are $m=500$ and $n=1000$, respectively, and the change occurs at observation $750$. The level is $\alpha=0.05$.\label{fig:corr} }
\end{figure}

\bigskip
\noindent

\subsection{Data example} \label{data}
In this section we provide a small data example to illustrate potential applications  of the new method based on the statistic $\hat{D}$.
For this purpose we consider the log-returns of the NASDAQ Composite Index and the Sandard \& Poor's 500 Index in the period from 1997-01-02 to 2002-12-31 (rise and burst of dot-com bubble) and investigate this data for potential changes in the covariance matrix. 
The log-returns and corresponding prices are shown in Figure \ref{fig:realData}.\\
If the monitoring rejects the null hypothesis of a stable covariance matrix (at time $k$), we directly obtain an estimate of the change point location via the formula:
\begin{align}\label{eq:posestimate}
\hat{\ell} =m+ \argmax_{j=0}^{k-1}
(m+j)^2(k-j)^2
\big( \hat{V}_{1}^{m+j} - \hat{V}_{m+j+1}^{m+k}\big)^\top
\hat \Sigma_m^{-1}
\big( \hat{V}_{1}^{m+j} - \hat{V}_{m+j+1}^{m+k} \big) ~,
\end{align}
where $\hat{V}_i^j := \vech(V(\hatFij))$ is the vectorized covariance functional [see equation \eqref{eq:varfunc}].
Note that \eqref{eq:posestimate} can be directly derived from formula \eqref{genstat}, where we use the constant threshold function.
To verify stability of the historical/training data set we employ the retrospective test based on the statistic $\tilde{\Lambda}_n$ defined on page 6 in \cite{Aue2009b} with the Newey-West estimator for the log-run variance using the automatic bandwidth selection contained in the R-package 'sandwich' [see \cite{Zeileis2004}].\\
To be precise, the methodology at hand is applied as follows:
If the test of \cite{Aue2009b} rejects for the potential initial stable sample of $m$ observations (where $m=255$ corresponds approximately to one year), we use the corresponding estimator, say $\hat \ell_A$, from this reference  to estimate the change point and consider (next) the set $ \{ X_{\hat\ell_A+1},\dots, X_{\hat \ell_A + m-1}\} $ to investigate stability.
Otherwise, we start the monitoring procedure proposed in this paper with the observations $\{X_1,\ldots ,X_m\}$ as the stable initial data set.
If this monitoring does not reject, we would simply stop monitoring after $m+mT$ observations (2002-12-31) and conclude that there was no change in the considered time frame.
If sequential monitoring rejects, we report the day of rejection and the corresponding location estimate $\hat{\ell}$ defined in \eqref{eq:posestimate}.
Next we define the $m$ ($\approx$ one year) data points $\{ X_{\hat \ell+1}, \dots, X_{m+\hat\ell +1}\}$ subsequent of the location estimate $\hat{\ell}$ as a new (potential) historic data set and restart with the retrospective analysis as described above.
The level of all tests is $5\%$.
The results of this approach are listed in Table \ref{tab:realData} and  the estimated change points are displayed in Figure \ref{fig:realData}.\\
For example the first column in Table \ref{tab:realData} is obtained as follows.
We start by applying the retrospective test to the first $255$ observations (1997-01-02 to 1997-12-31). Since this test does not reject the null hypothesis of a stable covariance matrix, we use this data  as  first initial sample  and start our monitoring procedure with end point 2002-12-31 ($m=255, T=4.92)$.
This sequential monitoring then rejects with observation $534$ (1999-02-12) and the corresponding location estimate for the change is $397$ (1998-07-29) [see also Table \ref{tab:realData}].
Next we check (covariance) stability of the observations $398$ to $652(=398+254)$ with the retrospective test. 
This test does not reject, which means that this set will be the new historic period.  
Now the sequential procedure is relaunched with the same end point (2002-12-31) but an adapted $T=3.36$.
This procedure is continued until the end of the period under investigation.
\begin{table}
\begin{tabular}{cc|cc|cc}
\multicolumn{2}{c}{(T1)} & \multicolumn{2}{c}{(T2)} & \multicolumn{2}{c}{(T3)} \\ 
change point & found at & change point & found at  &change point & found at  \\ 
\hline
1998-07-29 (S) & 1999-02-12 & 1998-07-29 (S) & 1998-10-01 & 1998-05-29 (S) & 1998-09-08 \\ 
1999-12-31 (S) & 2000-04-03 & 1999-10-14 (S) & 2000-03-15 & 1999-10-14 (S) & 2000-03-14 \\ 
2000-05-03 (R) & 2001-01-04 & 2001-04-23 (S) & 2002-01-31 & 2000-10-18 (S) & 2001-01-03 \\ 
2000-10-11 (R) & 2001-05-08 & 2002-05-15 (S) & 2002-07-22 & 2001-04-23 (R) & 2001-10-29 \\
2001-04-23 (R) & 2001-10-22 &                &            & 2002-05-14 (S) & 2002-07-19 \\ 
2002-06-14 (S) & 2002-07-10 &                &            &                &
\end{tabular} 
\caption{ \it 
\rev{Rejection dates (found at) and corresponding estimates of the change point locations for the (log-return) covariance matrix of NASDAQ Composite and Standard and Poor's 500 indices in the time frame 1997-01-02 to 2002-12-31 for different choices of threshold functions.\label{tab:realData}.}}
\end{table}
\begin{figure}
\centering
\begin{tabular}{cc}
\includegraphics[width=8cm,height=5.9cm]{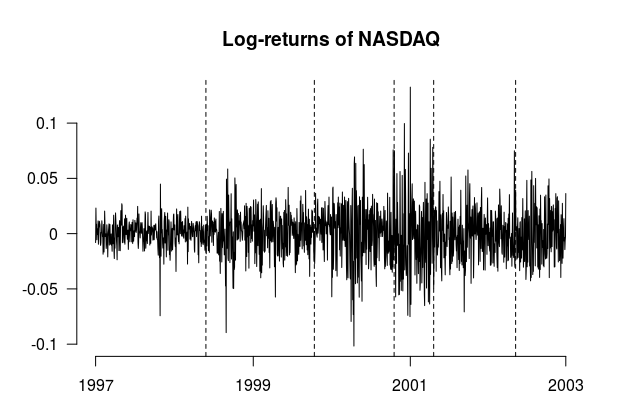} &
\includegraphics[width=8cm,height=5.9cm]{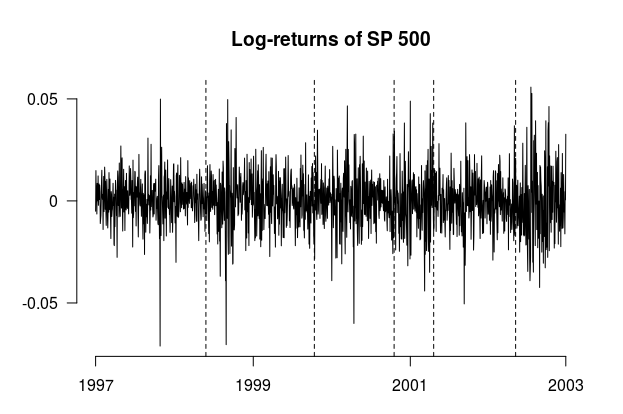} \\
\includegraphics[width=8cm,height=5.9cm]{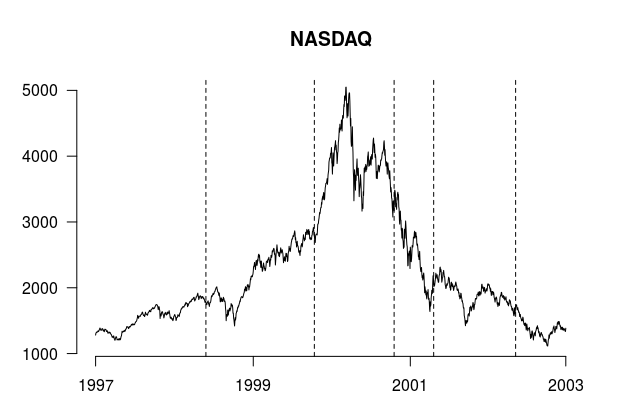} & 
\includegraphics[width=8cm,height=5.9cm]{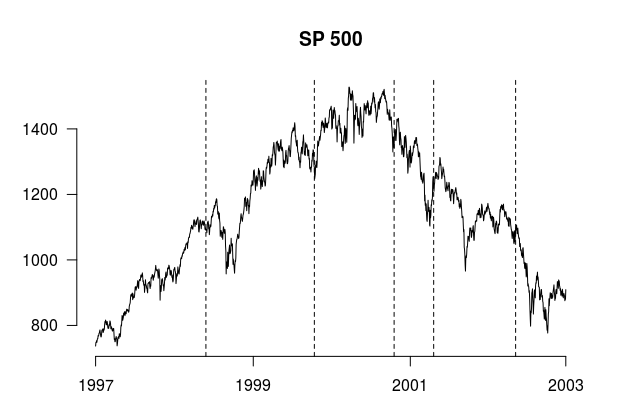} \\
\end{tabular}
\caption{\it Log-returns (upper row) and prices (lower row) of NASDAQ Composite and Standard and Poor's 500 indices. The vertical dashed lines mark positions of estimated change points found by our analysis with threshold function (T3).
In brackets we denote if a change point was detected by the sequential (S) or by the retrospective method (R).\label{fig:realData}}
\end{figure}

\section{Conclusion and outlook}\label{sec:concout}
In this paper we have proposed a new closed end sequential monitoring procedure for changes in the parameter of a $d$ dimensional time series.
We have proved that for large sample sizes the new tests keep its pre-specified level and are consistent.
Moreover, we introduce the concept of self-normalization for sequential change point detection to avoid the estimation of the long-run variance.
Our approach is motivated by the concept of likelihood ratio tests and is generally applicable, whenever the estimate of the parameter can be represented as a linear function of the empirical distribution function with a remainder satisfying several regularity conditions.
These assumptions have to be verified for each case individually and we do this here for the mean, variance and quantile.
An interesting direction for future research is the investigation if these conditions are also satisfied for other functionals such skewness, kurtosis or correlation (for the latter we give some numerical results).
\\
Our empirical findings provide strong evidence that compared to the currently available methodology the new sequential detection scheme based on the likelihood ratio approach yields to a substantial improvement with respect to power.
The improvement was observed in all examples under consideration.
It is less visible if the change points occur shortly after the initial stable sequence but substantial for all other cases, in particular, if the change point is close to the end of the monitoring period.
Thus - although the new approach is computationally more expensive - it should be preferred to the currently available methodology.
\\
An important direction for future research is the development of a corresponding methodology for open end procedures as considered by \cite{Chu1996}, \cite{Horvath2004} or \cite{Kirch2018}, for example.
Moreover, it remains to find suitable threshold functions for the monitoring procedure based on $\hat{D}$.
The functions considered in this work have been originally developed for other monitoring procedures [see \cite{Horvath2004} or \cite{Wied2013} among others], and it is likely that the performance of the new procedure can be further improved by choosing an alternative weight function.

\bigskip
\noindent

\if0\blind
{
{\bf Acknowledgments.}
This work has been supported in part by the Collaborative Research Center ``Statistical modeling of nonlinear dynamic processes'' (SFB 823, Teilprojekt  A1, C1) and the Research Training Group 'High-dimensional phenomena in probability - fluctuations and discontinuity' (RTG 2131) of the German Research Foundation (DFG).
The authors would like to thank Claudia Kirch, Dominik Wied and Wei Biao Wu for some helpful discussions on this subject.
We are also grateful to the three unknown referees and the associate editor for their constructive comments on an earlier version of the paper.
} \fi

\bibliographystyle{chicago}
\setlength{\bibsep}{1pt}
\begin{small}
\bibliography{literature}
\end{small}
\newpage

\appendix
\setcounter{page}{1}
\title{Online Supplement to: A likelihood ratio approach to sequential change point detection for a general class of parameters}
\if0\blind
{
\author{ {\small Holger Dette, Josua G\"osmann}\\
 {\small Ruhr-Universit\"at Bochum }\\
 {\small Fakult\"at f\"ur Mathematik} \\
 {\small44780 Bochum} \\
 {\small Germany }
}
} \fi
\maketitle

\section{Technical details}
\label{sec6}
\begin{proof}[\bf Proof of Theorem \ref{thm:meta-main}]
From \eqref{eq:meta-remainder}, \eqref{utilde} we obtain the representation
\begin{align*}
&m^{-3/2} \tilde \U(\floor{mr},\floor{ms},\floor{mt})
= m^{-3/2}(\floor{mt}-\floor{ms})(\floor{ms}-\floor{mr})
\Big( \hattheta_{\floor{mr+1}}^{\floor{ms}} - \hattheta_{\floor{ms}+1}^{\floor{mt}} \Big)\\
&= \dfrac{\floor{mt}-\floor{ms}}{m^{3/2}} \sum_{i=\floor{mr}+1}^{\floor{ms}} \IF(X_i,F,\theta)
- \dfrac{\floor{ms}-\floor{mr}}{m^{3/2}}\sum_{t=\floor{ms}+1}^{\floor{mt}} \IF(X_i,F,\theta)\\
&+ \dfrac{(\floor{mt}-\floor{ms})(\floor{ms}-\floor{mr})}{m^{3/2}}
\big( R_{\floor{mr}+1,\floor{ms}} - R_{\floor{ms}+1,\floor{mt}} \big)~.
\end{align*}
By Assumption \ref{assump:meta-1} we have
\begin{align*}
\Big \{
\dfrac{\floor{mt}-\floor{ms}}{m^{3/2}} &\sum_{i=\floor{mr}+1}^{\floor{ms}} \IF(X_i,F,\theta)
- \dfrac{\floor{ms}-\floor{mr}}{m^{3/2}}\sum_{i=\floor{ms}+1}^{\floor{mt}} \IF(X_i,F,\theta)
\Big \}_{(r,s,t) \in \Delta_3}\\
&\convd
\Sigma^{1/2}_F\Big \{\big(t-s\big)\big(W(s)-W(r)\big) - \big(s-r\big)\big(W(t)-W(s)\big)\Big \}_{(r,s,t) \in \Delta_3} \\
&~~~~= \Sigma^{1/2}_F\big \{   B(s,t) + B(r,s) -B(r,t)  \big \}_{(r,s,t) \in \Delta_3}~~,
\end{align*}
where we use the definition of the process $B$ in \eqref{eq:Bst} and the fact
\begin{align*}
\sup_{(s,t) \in \Delta_2}
\Big| \dfrac{\floor{mt}-\floor{ms}}{m} - (t-s) \Big|
\leq \dfrac{2}{m} = o(1)~.
\end{align*}
Finally, Assumption \ref{assump:meta-2} yields
\begin{align*}
\dfrac{(\floor{mt}-\floor{ms})(\floor{ms}-\floor{mr})}{m^{3/2}}
\big( R_{\floor{mr}+1,\floor{ms}} - R_{\floor{ms}+1,\floor{mt}} \big)
= o_p(1)~,
\end{align*}
uniformly with respect to $(r,s,t) \in \Delta_3$ so that the proof of Theorem \ref{thm:meta-main} is finished by Slutsky's Theorem.
\end{proof}

\medskip
\begin{proof}[\bf Proof of Corollary \ref{cor:meta-lvl1}]
Define
\begin{align}
\label{dhat}
D_m(k)
= m^{-3} \max_{j=0}^{k-1}
| \U^\top(m+j,m+k)  \Sigma_F^{-1} \U(m+j,m+k) |~.
\end{align}
Using the fact, that the detection scheme $\{D_m(\floor{mt})\}_{t\in [0,T]}$ is piecewise constant (with respect to $t$) and the monotonicity of the threshold function we obtain the representation
\begin{align*}
\max_{k=1}^{Tm} &\dfrac{D_m(k)}{w(k/m)}
= \sup_{t \in [0,T]} \dfrac{D_m(\floor{mt})}{w(t)}\
\
&= \sup_{t \in [1,T+1]}\sup_{s \in [1,t]}
\dfrac{m^{-3}\big| \U^\top(\floor{ms},\floor{mt})  \Sigma_F^{-1}
\U(\floor{ms},\floor{mt}) \big|}
{w(t-1)}~.
\end{align*}
By Remark \ref{rem:Bst} and the continuous mapping theorem we have
\begin{align*}
\max_{k=1}^{Tm} \dfrac{ D_m(k)}{w(k/m)} \convd
\sup_{t \in [1,T+1]} \sup_{s\in [1,t]}
\dfrac{B(s,t)^{\top} B(s,t) }{w(t-1)}~,
\end{align*}
where the process $B$ is defined in \eqref{eq:Bst}.
The result now follows from Remark \ref{rem:Bst},
the fact, that $w_{\alpha}$ has a lower bound and that $\hat \Sigma_m$ is a consistent
estimate of the matrix $\Sigma_F$, which implies (observing  the definition  of $\hat D$ in \eqref{dhat})
\begin{align*}
\Big|
\max_{k=1}^{Tm} \dfrac{ D_m(k)}{w(k/m)} - \max_{k=1}^{Tm} \dfrac{ \hatD_m(k)}{w(k/m)} \Big| & \leq
\max_{k=1}^{Tm} \big |  D_m(k)  - \hatD_m(k) \big |
\\
& \leq \| \hat \Sigma_m^{-1} -
\Sigma_F^{-1} \|_{op}
\sup_{t \in [1,T+1]}\sup_{s \in [1,t]} | m^{-3/2} \U
(\floor{ms},\floor{mt}) |^2   \\
&= o_\Pb(1)~.
\end{align*}
Here $\Vert \cdot \Vert_{op}$ denotes the operator norm and we have  used
the estimate $\| \hat \Sigma_m^{-1} -\Sigma_F^{-1} \|_{op}
= o_\Pb(1)$, which is a consequence
of  the Continuous Mapping Theorem.
\end{proof}

\medskip

\begin{proof}[\bf Proof of Theorem \ref{thm:meta-power}]
By the definition of the statistic $\hatD$ in \eqref{genstat}, we obtain
\begin{align}\label{ineq:powerlowerbound}
\max_{k=0}^{Tm} &\dfrac{\hatD_m(k)}{w_{\alpha}(k/m)}
\geq m^{-3}\dfrac{\big| \U^\top(\floor{mc},m(T+1))
\hat \Sigma_m^{-1}
\U(\floor{mc},m(T+1) \big|}
{w_{\alpha}(T)}~,
\end{align}
where $\floor{mc}$ denotes the (unknown) location of the change.
We can apply expansion \eqref{eq:meta-remainder} to $X_1, \ldots, X_{\floor{mc}}$ and
$X_{\floor{mc}+1}, \ldots ,X_{\floor{mT}}$ and obtain
\begin{align*}
m^{-3/2} \U(\floor{mc},m(T+1))
&= \dfrac{\floor{mc}\big( m(T+1) - \floor{mc} \big)}{m^{3/2}}
\Big( \hattheta_{1}^{\floor{mc}} - \hattheta_{\floor{mc}+1}^{m(T+1)} \Big)\\
&= \dfrac{m(T+1) - \floor{mc}}{m^{3/2}}
\sum_{i=1}^{\floor{mc}} \IF(X_i, F^{(1)} , \theta_{F^{(1)} })  \\
& - \dfrac{\floor{mc}}{m^{3/2}}\sum_{i=\floor{mc}+1}^{m(T+1)} \IF(X_i, F^{(2)} , \theta_{F^{(2)} })\\
&+ \dfrac{\floor{mc} \big( m(T+1) - \floor{mc} \big)}{m^{3/2}}
\Big( \theta_{F^{(1)} } - \theta_{F^{(2)}} + R_{1,\floor{mc}}^{(F^{(1)} )} - R_{\floor{mc}+1,m(T+1)}^{(F^{(2)})} \Big)~,
\end{align*}
where $\theta_{F^{(\ell)}} = \theta (F^{(\ell )})$ $(\ell=1,2)$.
Using Assumption \ref{assump:meta-power} we obtain the joint convergence of
\begin{align*}
 \frac{1}{m^{3/2}}
\begin{pmatrix}
\big( m(T+1) - \floor{mc} \big)
\sum_{i=1}^{\floor{mc}} \IF(X_i, {F^{(1)} }, \theta_{F^{(1)} }) \\
\floor{mc}\sum_{i=\floor{mc}+1}^{m(T+1)} \IF(X_i,{F^{(2)}}, \theta_{F^{(2)}})
\end{pmatrix}
\convd
\begin{pmatrix}
(T+1 - c ) \sqrt{\Sigma_{F^{(1)} }}W_1(c)
 \\
c \sqrt{\Sigma_{F^{(2)}}} \Big( W_2(T+1) - W_2(c) \Big)
\end{pmatrix}
\end{align*}
and
\begin{align*}
\dfrac{\floor{mc}m(T+1)}{m^{3/2}}\Big(R_{1,\floor{mc}}^{({F^{(1)} })} - R_{\floor{mc}+1,m(T+1)}^{({F^{(2)}})}\Big)
&\convp 0~.
\end{align*}
As $\theta_{F^{(1)} } \not = \theta_{F^{(2)} }$ this directly implies
$m^{-3/2} |\U(\floor{mc},m(T+1)) |\convp \infty~,$
and the assertion follows from \eqref{ineq:powerlowerbound} and the assumption that
$\hat \Sigma_m$ is a consistent estimate for $\Sigma_{F^{(1)} }$.
\end{proof}
\bigskip

\begin{proof}[\bf Proof of Theorem \ref{cor:meta-lvl1SN}]
Recalling the definition of $ \tilde{\U}$ and $ \U$ in \eqref{utilde} and \eqref{eq:meta-cusum}, respectively, we obtain for the normalizing process $\V$ in \eqref{eq:meta-snprocess} the representation
\begin{align*}
m^{-4}
\V(\floor{ms},\floor{mt})
&= m^{-4} \sum_{j=1}^{\floor{ms}}j^2(\floor{ms}-j)^2
\Big(\hattheta_1^{j} - \hattheta_{j+1}^{\floor{ms}}\Big)
\Big(\hattheta_1^{j} - \hattheta_{j+1}^{\floor{ms}}\Big)^\top\\
&+ m^{-4} \sum_{j=\floor{ms}+1}^{\floor{mt}} (\floor{mt}-j)^2(j-\floor{ms})^2
\Big(\hattheta_{\floor{ms}+1}^{j} - \hattheta_{j+1}^{\floor{mt}} \Big)
\Big(\hattheta_{\floor{ms}+1}^{j} - \hattheta_{j+1}^{\floor{mt}} \Big)^\top\\
&= m^{-4}\sum_{j=1}^{\floor{ms}} \U(j,\floor{ms})\U^\top(j,\floor{ms})\\
&+ m^{-4}\sum_{j=\floor{ms}+1}^{\floor{mt}}
\tilde  \U(\floor{ms},j,\floor{mt})\tilde  \U^\top(\floor{ms},j,\floor{mt})\\
&= m^{-3}\int_0^s
\U(\floor{mr},\floor{ms})\U^\top(\floor{mr},\floor{ms})dr\\
&+ m^{-3}\int_s^t
\tilde \U(\floor{ms},\floor{mr},\floor{mt})\tilde  \U^\top(\floor{ms},\floor{mr},\floor{mt}) dr~.
\end{align*}
By Theorem \ref{thm:meta-main} we have
\begin{align}\label{conv:kern}
\big \{
m^{-3/2} \tilde{\U}_m(\floor{mr}, \floor{ms}, \floor{mt}) \big \}_{(r,s,t) \in \Delta_3}
&\convd  \Sigma^{1/2}_F\big \{   B(s,t) + B(r,s) -B(r,t)  \big \}_{(r,s,t) \in \Delta_3}
\end{align}
in the space $\ell^{\infty}(\Delta_3, \R^p)$, where the process $B$ is defined in \eqref{eq:Bst}.
Consequently, the Continuous Mapping Theorem yields (in the space $\ell^{\infty}(\Delta_2, \R^p \times \R^p)$)
\begin{align}\label{conv:toshow1}
\Big \{ \Big (
\begin{array}{c}
m^{-3/2}\cdot\U(\floor{ms},\floor{mt}) \\
m^{-4}\cdot \V(\floor{ms},\floor{mt})
\end{array}
\Big )\Big\}_{(s,t) \in \Delta_2}
\convd
\Big \{ \Big (
\begin{array}{c} \Sigma^{1/2}_FB(s,t) \\
\Sigma^{1/2}_F\big(N_1(s) + N_2(s,t)\big)\Sigma_F^{1/2}
\end{array}
\Big ) \Big\}_{(s,t) \in \Delta_2}~,
\end{align}
where $N_1$, $N_2$ are defined in \eqref{eq:denominator}.
Now the assertion of Theorem \ref{cor:meta-lvl1SN} follows by a further application of the  Continuous Mapping Theorem.
\end{proof}

\medskip

\begin{proof}[\bf Proof of Theorem \ref{thm:meta-powerSN}]
By definition of the self-normalized statistic $\hatDSN$  in \eqref{genstatSN}, we obtain
\begin{align}\label{ineq:powerlowerboundSN}
\max_{k=0}^{Tm} \dfrac{\hatDSN_m(k)}{w_{\alpha}(k/m)}
\geq m \cdot \dfrac{\big| \U^\top(\floor{mc},m(T+1)) \V^{-1}(\floor{mc},m(T+1))
\U(\floor{mc},m(T+1) \big|}
{w_{\alpha}(T)}~,
\end{align}
where $\floor{mc}$ denotes the (unknown) location of the change.
The discussion in the proof of Theorem \ref{thm:meta-power} shows
\begin{align*}
m^{-3/2} \U(\floor{mc},m(T+1))
\convp \infty~.
\end{align*}
The proof will be completed by inspecting the random variable $\V^{-1}(\floor{mc},m(T+1))$ in the lower bound in \eqref{ineq:powerlowerboundSN}.
Repeating again the arguments from the proof of Theorem \ref{thm:meta-main} we can rewrite
\begin{align}\label{eq:repSN}
\begin{split}
m^{-4}\cdot\V(\floor{mc},m(T+1))
&= m^{-3}\int_0^c
\U(\floor{mr},\floor{ms})\U^\top(\floor{mr},\floor{ms})dr\\
&+ m^{-3}\int_c^{T+1}
\tilde \U(\floor{ms},\floor{mr},\floor{mt})\tilde  \U^\top(\floor{ms},\floor{mr},\floor{mt}) dr~.
\end{split}
\end{align}
Using Assumption \ref{assump:meta-power} and employing the arguments from the proof of Theorem \ref{thm:meta-main} we obtain weak convergence of
\begin{align*}
\begin{pmatrix}
\{ \U(\floor{mr},\floor{ms}) \}_{0 \leq r \leq s \leq c}  \\
\{ \tilde \U(\floor{ms},\floor{mr},\floor{mt})\}_{c \leq s \leq r \leq t \leq T+1}
\end{pmatrix}
\convd
\begin{pmatrix}
\{ B^{(1)}(r,s) \}_{0 \leq r \leq s \leq c}\\
\{  B^{(2)}(r,t) + B^{(2)}(s,r) -B^{(2)}(s,t) \}_{c \leq s \leq r \leq t \leq T+1}
\end{pmatrix}~,
\end{align*}
where we use the extra definition
\begin{align*}
B^{(\ell)}(s,t) = tW_{\ell}(s) - sW_{\ell}(t)\quad\quad \ell = 1,2
\end{align*}
and $W_1$ and $W_2$ are defined in Assumption \ref{assump:meta-power}.
By the Continuous Mapping Theorem and the representation in \eqref{eq:repSN} this implies
\begin{align*}
m^{-4}\cdot\V(\floor{mc},m(T+1)) \convd \Sigma_{F^{(1)}}^{1/2} \big( N^{(1)}_1(c) \big)\Sigma_{F^{(2)}}^{1/2} + \Sigma_{F^{(2)}}^{1/2}\big( N^{(2)}_2(c,T+1)\big)\Sigma_{F^{(2)}}^{1/2}~,
\end{align*}
where the processes $N_1^{(1)}$ and $N_2^{(2)}$ are distributed like $N_1$ and $N_2$ in \eqref{eq:denominator} but with respect to the processes $B^{(1)}$ and $B^{(2)}$, respectively.
\end{proof}

\bigskip

\begin{proof}[\bf Proof of Proposition \ref{prop:var}]
For the sake of readability, we will give the proof only for the case $d=2$.
The arguments presented here  can be easily extended to higher dimension.
In view of the representation in \eqref{eq:compsvar}, we may also assume without loss of generality that $\mu = \mathbb{E} [X_t] =0$.
\\\\
\noindent
Part (a) of the proposition is a consequence
of the discussion after Corollary \ref{cor:meta-lvl1} provided that Assumptions \ref{assump:meta-1} and \ref{assump:meta-2} can be established.
For this purpose we introduce the notation
\begin{align*}
Z_t
:= \IF_v(X_t, F, V)
= \begin{pmatrix}
X_{t,1}^2 - \E[X_{t,1}^2] \\
X_{t,1}X_{t,2} - \E[X_{t,1}X_{t,2}] \\
X_{t,2}^2 - \E[X_{t,2}^2] \\
\end{pmatrix}~
\end{align*}
and note that the time series $\{Z_t\}_{t \in \mathbb{Z}}$ can be represented as a physical system, that is
\begin{align}\label{eq:newsystem}
Z_t
=  \begin{pmatrix}
g_{1}^2(\varepsilon_t,\dots) - \E[X_{1,1}^2] \\
g_{1}(\varepsilon_t,\dots)g_{2}(\varepsilon_t,\dots) - \E[X_{1,1}X_{1,2}] \\
g_{2}^2(\varepsilon_t,\dots)  - \E[X_{1,2}^2]
\end{pmatrix}
:= G(\varepsilon_t, \varepsilon_{t-1}, \dots) ~,
\end{align}
where $g_i$ denotes the $i$-th component of the function $g$ in \eqref{eq:physicalsys}.
In view of definition \eqref{eq:physicalcoeff} introduce the notation
\begin{align*}
X_t' = g(\varepsilon_t,\varepsilon_{t-1},\dots,\varepsilon_1,\varepsilon_0',\varepsilon_{-1},\dots)~.
\end{align*}
The corresponding physical dependence coefficients $\delta_{t,2}^Z$ in \eqref{eq:physicalcoeff} are then given by
\begin{align*}
\delta_{t,2}^Z
&= \Big\Vert
\sqrt{ (X_{t,1}^2 - (X_{t,1}')^2)^2
+ (X_{t,2}^2 - (X_{t,2}')^2 )^2
+ (X_{t,1}X_{t,2} - X_{t,1}'X_{t,2}')^2} \Big\Vert_2  \\
& \leq \Vert X_{t,1}^2 - (X_{t,1}')^2 \Vert_2 + \Vert X_{t,2}^2 - (X_{t,2}')^2 \Vert_2 + \Vert X_{t,1}X_{t,2} - X_{t,1}'X_{t,2}' \Vert_2 \\
&\leq 3 \cdot \max\Big\{\Vert X_{t,1}^2 - (X_{t,1}')^2 \Vert_2~,
\Vert X_{t,2}^2 - (X_{t,2}')^2 \Vert_2~,
\Vert X_{t,1}X_{t,2} - X_{t,1}'X_{t,2}' \Vert_2 \Big\}~,
\end{align*}
where we used the inequality $\sqrt{a+b} \leq \sqrt{a} + \sqrt{b}$ for $a,b >0$.
Now H\"older's inequality yields for an appropriate constant $C$
\begin{align*}
\Vert X_{t,1}^2 - (X_{t,1}')^2 \Vert_2
&\leq \Vert X_{t,1} + X_{t,1}' \Vert_4 \Vert X_{t,1} - X_{t,1}' \Vert_4
\leq C \cdot \delta_{t,4}~,
\\
\Vert X_{t,1}X_{t,2} - X_{t,1}'X_{t,2}' \Vert_2
&\leq \big\Vert X_{t,1}\big( X_{t,2} - X_{t,2}' \big) \big\Vert_2
+ \big\Vert X_{t,2}'\big( X_{t,1} - X_{t,1}' \big) \big\Vert_2\\
&\leq \big\Vert X_{t,1} \big\Vert_4 \big\Vert X_{t,2} - X_{t,2}' \big\Vert_4
+ \big\Vert X_{t,2}' \big\Vert_4 \big\Vert X_{t,1} - X_{t,1}' \big\Vert_4
\leq C \cdot \delta_{t,4}^{(1)}~.
\end{align*}
Combining these results gives
$\sum_{t=1}^\infty \delta_{t,2}^Z \leq C\cdot  \Theta_{4}^{(1)} < \infty$ and Theorem 3 from \citesuppl{Wu2005} implies the weak convergence
\begin{align*}
\dfrac{1}{\sqrt{m}} \sum_{t=1}^{\floor{ms}} \IF_v(X_t, F, V)
=  \dfrac{1}{\sqrt{m}} \sum_{t=1}^{\floor{ms}} Z_t
\convd \sqrt{\Sigma_F} W(s)
\end{align*}
in the space $\ell^{\infty}([0,T+1],\R^3)$ as $m \to \infty$, where $\Sigma_F$ is the long-run variance matrix defined in \eqref{eq:meta-covar}.
Therefore Assumption \ref{assump:meta-1} is satisfied.

\medskip
\noindent To finish part (a) it remains to show that Assumption \ref{assump:meta-2} holds.
Due to \eqref{eq:remaindervar} this
is a consequence of
\begin{align}\label{eq:varRemainder}
\sup_{1\leq i < j \leq n} \dfrac{1}{\sqrt{j-i+1}}
\Big|\sum_{t=i}^j X_{t,\ell} - \E[X_{t,\ell}] \Big|
= \op(n^{1/4})
\end{align}
for $\ell=1,2,3$.
Since the arguments are exactly the same, we will only elaborate the case $\ell=1$.
For this purpose let
\begin{align*}
S_i = \sum_{t=1}^i X_{t,1} - \E[X_{t,1}]~,
\end{align*}
and note that the left-hand side of \eqref{eq:varRemainder} can be rewritten as
\begin{align*}
\max_{1\leq j \leq n} \max_{ 1 \leq k \leq n-j } \dfrac{1}{\sqrt{k}} |S_{j+k} - S_j |
= \max\Big\{ \max_{1\leq j \leq n} \max_{ 1 \leq k \leq n-j } &\dfrac{1}{\sqrt{k}} (S_{j+k} - S_j)\; ,\;\\
&\max_{1\leq j \leq n} \max_{ 1 \leq k \leq n-j } \dfrac{-1}{\sqrt{k}} (S_{j+k} - S_j) \Big\}~.
\end{align*}
Thus it suffices to show that both terms inside the (outer) maximum are of order $\op(n^{1/4})$~.
For the sake of brevity, we will only prove that
\begin{align}\label{eq:nurpos}
\max_{1\leq j \leq n} \max_{ 1 \leq k \leq n-j } \dfrac{1}{\sqrt{k}} (S_{j+k} - S_j)
= \op(n^{1/4})
\end{align}
and the  other term can be treated in the same way.
Assertion \eqref{eq:nurpos} follows obviously from
the two estimates
\begin{align}
\max_{1\leq j \leq n} \max_{1 \leq k \leq (n-j) \wedge \floor{\log^2(n)}}
\dfrac{S_{j+k}-S_j}{\sqrt{k}n^{1/4}}
&= \op(1)~,\label{eq:remain1}\\
\max_{1\leq j \leq n} \max_{ \floor{\log^2(n)} \leq k \leq n-j }
\dfrac{S_{j+k}-S_j}{\sqrt{k}n^{1/4}}
&= \op(1)~.\label{eq:remain2}
\end{align}
Since the function $g$ is bounded, one directly obtains that there exists a constant $C$ such that $|X_{j,1} - \E[X_{j,1}]| \leq C$.
This gives
\begin{align*}
\bigg| \max_{1\leq j \leq n} \max_{1 \leq k \leq (n-j) \wedge \floor{\log^2(n)}}
\dfrac{S_{j+k}-S_j}{\sqrt{k}n^{1/4}} \bigg|
\leq \max_{1\leq j \leq n} \max_{1 \leq k \leq (n-j) \wedge \floor{\log^2(n)}}
\dfrac{\sqrt{k}C}{n^{1/4}} = o(1)
\end{align*}
and so \eqref{eq:remain1} is shown.
To establish \eqref{eq:remain2} we will use Corollary 1 from \citesuppl{Wu2011}, which implies, that (on a richer probability space) there exists a process $\{ \check{S}_i \}_{i=1}^n$ and a Gaussian process $\{ \check{G}_i \}_{i=1}^n$, such that
\begin{align*}
(\check{S}_1,\dots,\check{S}_n)
&\eqd (S_1,\dots,S_n)
\qquad
\text{and}
\qquad
\max_{1 \leq i \leq n} |\check{S}_i - \check{G}_i|
= \mathcal{O}_{\mathbb{P}}\big(n^{1/4}(\log n)^{3/2}\big)~.
\end{align*}
Additionally, (again on a richer probability space) there exists another Gaussian process $\{\hat{G}_i\}_{i=1}^n$ such that
\begin{align*}
(\check{G}_1,\dots,\check{G}_n)
&\eqd (\hat{G}_1,\dots,\hat{G}_n)
\qquad
\text{and}
\qquad
\max_{1 \leq i \leq n} |\hat{G}_i - G_i|
= \mathcal{O}_{\mathbb{P}}\big(n^{1/4}(\log n)^{3/2}\big)~,
\end{align*}
where the process $G$ is given by
\begin{align*}
\{ G_i \}_{i=1}^n
= \Big \{ \sum_{t=1}^i Y_{t} \Big \}_{i=1}^n~,
\end{align*}
with i.i.d. Gaussian distributed random variables $Y_1,\dots,Y_n \sim\mathcal{N}(0, (\Gamma(g))_{1,1})\big)$ with $\Gamma(g)$ defined in \eqref{eq:ordcovar}.
Therefore we obtain
\begin{align*}
\bigg\vert
&\max_{1\leq j \leq n} \max_{ \floor{\log^2(n)} \leq k \leq n-j }
\dfrac{\check{S}_{j+k}-\check{S}_j}{\sqrt{k}n^{1/4}}
- \max_{1\leq j \leq n} \max_{ \floor{\log^2(n)} \leq k \leq n-j }
\dfrac{\check{G}_{j+k}-\check{G}_j}{\sqrt{k}n^{1/4}}
\bigg\vert\\
&\leq
\max_{1\leq j \leq n} \max_{ \floor{\log^2(n)} \leq k \leq n-j }
\bigg\vert
\dfrac{\check{S}_{j+k}-\check{S}_j - \check{G}_{j+k} + \check{G}_j}{\sqrt{k}n^{1/4}}
\bigg\vert
\leq
2\max_{1\leq j \leq n}
\dfrac{| \check{S}_{j}-\check{G}_j |}{\log^2(n)n^{1/4}}
= \op(1)
\end{align*}
and by the same arguments 
\begin{align*}
\bigg\vert
&\max_{1\leq j \leq n} \max_{ \floor{\log^2(n)} \leq k \leq n-j }
\dfrac{\hat{G}_{j+k}-\hat{G}_j}{\sqrt{k}n^{1/4}}
- \max_{1\leq j \leq n} \max_{ \floor{\log^2(n)} \leq k \leq n-j }
\dfrac{G_{j+k}-G_j}{\sqrt{k}n^{1/4}}
\bigg\vert
= \op(1)~.
\end{align*}
Now Theorem 1 in \citesuppl{Shao1995} gives
\begin{align*}
\lim_{n \to \infty} \max_{1\leq j \leq n}
\max_{ \floor{\log^2(n)} \leq k \leq n-j }
\dfrac{G_{j+k}-G_j}{\sqrt{k}n^{1/4}} = 0
\end{align*}
with probability $1$, which completes the proof of Part (a).
\medskip\\
\noindent
For a proof of part (b) of Proposition \ref{prop:var} let $F^{(1)}$, $\Sigma_{F^{(1)}}$ and $F^{(2)}$, $\Sigma_{F^{(2)}}$ denote the distribution function and corresponding long-run variances in equation \eqref{eq:meta-covartwo} before and after the change point, respectively.
Note that $h = A \cdot g$ and consider the time series
\begin{align*}
\tilde{X}_t =
\begin{cases}
g(\varepsilon_t,\varepsilon_{t-1},\dots)\;\;&\text{if}\;\; t < \floor{mc}~,\\
A^{-1}\cdot h(\varepsilon_t,\varepsilon_{t-1},\dots)\;\;&\text{if}\;\; t \geq \floor{mc}~,
\end{cases}
\end{align*}
which is strictly stationary with distribution function $F^{(1)}$.
Using similar arguments as in the proof of part (a), one easily verifies that
\begin{align}\label{conv:oneprocess}
\Big \{
\dfrac{1}{\sqrt{m}} \sum_{t=1}^{\floor{ms}} \IF_v(\tilde{X}_t, F^{(1)}, V) \Big \}_{s \in [0,T+1]}
\convd \{ \sqrt{\Sigma_{F^{(1)}}} W(s) \}_{s \in [0,T+1]}~.
\end{align}
Next, observe that there exists a matrix $A^{(v)} \in \R^{3\times 3}$, such that for all symmetric matrices $M \in \R^{2\times 2}$, the following identity holds
\begin{align*}
\vech(A \cdot M \cdot A^{\top}) = A^{(v)} \cdot \vech(M)~.
\end{align*}
Further, using \eqref{eq:infvar} one observes
\begin{align*}
A \cdot \IF(\tilde{X}_t,F^{(1)}, V) \cdot A^\top
&= A  \tilde{X_t}  \tilde{X_t}^\top A^{\top}
- A\cdot V(F^{(1)})\cdot A^{\top} = \IF(X_t, F^{(2)}, V)~
\end{align*}
whenever $t \geq \floor{mc}$,
which yields
\begin{align}\label{eq:vector1}
A^{(v)} \IF_v(\tilde{X}_t,F^{(1)}, V)
= \IF_v(X_t, F^{(2)}, V)  \quad\text{for}\quad t \geq \floor{mc}~.
\end{align}
Similar arguments give
\begin{align}\label{eq:vector2}
A^{(v)} \Sigma_{F^{(1)}} (A^{(v)})^\top
= \Sigma_{F^{(2)}}~.
\end{align}
Now consider the mapping
\begin{align*}
\Phi_A:
\left\{
\begin{array}{l}
\;\ell^{\infty}([0,T+1],\R^3) \to
\ell^{\infty}([0,c],\R^3) \times \ell^{\infty}([c,T+1],\R^3)~,\vspace{0.2cm}\\
\{f(s)\}_{s \in [0,T+1]} \mapsto
\begin{pmatrix}
\{f(s)\}_{s \in [0,c]  }\vspace{0.2cm}\\
\big\{ A^{(v)} ( f(s) -f(c)) \big\}_{s \in [c,T+1]}
\end{pmatrix}~,
\end{array}
\right.
\end{align*}
then the Continuous Mapping, \eqref{conv:oneprocess} and \eqref{eq:vector1} yield
\begin{align*}
\left(
\begin{array}{c}
\big\{\frac{1}{\sqrt{m}} \sum_{t=1}^{\floor{ms}} \IF_v(X_t,F^{(1)},V)  \big\}_{s\in [0,c]} \\
\big\{\frac{1}{\sqrt{m}} \sum_{t=\floor{mc}+1}^{\floor{ms}} \IF_v(X_t, F^{(2)},V)  \big\}_{s\in [c,T+1]}
\end{array}
\right)~~&\convd ~
\left(
\begin{array}{c}
 \big\{\sqrt{\Sigma_{F^{(1)}}}
 W(s)\big\}_{s\in [0,c]} \\
 \big\{A^{(v)}\sqrt{\Sigma_{F^{(1)}}} \big( W(s) - W(c)\big) \big\}_{s\in [c,T+1]}
\end{array}
\right)\\
&\qquad\eqd
\left(
\begin{array}{c}
 \big\{\sqrt{\Sigma_{F^{(1)}}}
 W(s)\big\}_{s\in [0,c]} \\
 \big\{\sqrt{\Sigma_{F^{(2)}}} \big( W(s) - W(c)\big) \big\}_{s\in [c,T+1]}
\end{array}
\right)~,
\end{align*}
where the identity in distribution follows from the fact that both components are independent and the identity
\begin{align*}
\big( A^{(v)} \sqrt{\Sigma_{F^{(1)}}} \big)\big( A^{(v)} \sqrt{\Sigma_{F^{(1)}}} \big)^{\top}
= A^{(v)} \Sigma_{F^{(1)}} (A^{(v)})^\top
= \Sigma_{F^{(2)}}~.
\end{align*}
For the verification of Assumption \ref{assump:meta-power} it suffices to show that both, the phase before and after the change point satisfy Assumption \ref{assump:meta-2}.
This can be done using similar arguments as in the proof of part (a) of Proposition \ref{prop:var} and the details are omitted.
\end{proof}

\begin{proof}[\bf Proof of Theorem \ref{thm:quantremain}]
\noindent For a proof of Theorem \ref{thm:quantremain} we will require six Lemmas, that are stated below.
Lemmas \ref{lem:EDFbound}, \ref{lem:shorack}, \ref{lem:FFxy}, \ref{lem:quantas}, \ref{lem:combined} are partially adapted from Lemma 2 and the proof of Theorem 4 in \citesuppl{Wu2005b} but extended to hold uniformly in sample size.
Lemma \ref{lem:quantlem} controls the error of the quantile estimators in case of small samples, where the tail assumptions on the distribution function comes into play. 

\begin{lemma}\label{lem:EDFbound}
Under the assumptions of Theorem \ref{thm:quantremain} for all $0<r<1$ and $\vartheta > 1$, there exists a constant $C_{r,\vartheta}$, such that 
\begin{align*}
\Pb\bigg(\max_{\substack{1 \leq i < j \leq n \\ |j - i| \geq n^r }}\sup_{x \in \R} \;
|\hatFij(x) - F(x)| > C_{r,\vartheta} \dfrac{\sqrt{r\log(n)}}{n^{r/2}} \bigg) \lesssim n^{-\vartheta}~.
\end{align*}
\end{lemma}
\begin{proof}
We have the following upper bounds
\begin{align*}
\Pb\bigg(\max_{\substack{1 \leq i < j \leq n \\ |j - i| \geq n^r }}\sup_{x \in \R} \;
|\hatFij(x) - F(x)| &> C_{r,\vartheta} \dfrac{\sqrt{r\log(n)}}{n^{r/2}} \bigg)\\
&\leq \sum_{\substack{1 \leq i < j \leq n \\ |j - i| \geq n^r }} \Pb\bigg(\sup_{x \in \R} \;
|\hatFij(x) - F(x)| > C_{r,\vartheta} \dfrac{\sqrt{r\log(n)}}{n^{r/2}} \bigg)\\
&\leq \sum_{\substack{1 \leq i < j \leq n \\ |j - i| \geq n^r }} \Pb\bigg(\sup_{x \in \R} \;
|\hatFij(x) - F(x)| > C_{r,\vartheta} \dfrac{\sqrt{\log(j-i+1)}}{\sqrt{j-i+1}} \bigg)~.
\end{align*}
Now choose $\tau>0$ sufficiently large to fulfill $2-\tau r < -\vartheta$. 
Applying Lemma 2 from \citesuppl{Wu2005b}, we obtain that $C_{r,\vartheta}$ can be chosen, such that the last term is (up to a constant) bounded by
\begin{align*}
\sum_{\substack{1 \leq i < j \leq n \\ |j - i| \geq n^r }} |j-i+1|^{-\tau}
\leq \sum_{\substack{1 \leq i < j \leq n \\ |j - i| \geq n^r }} n^{-r\tau}
\leq n^{2- r\tau}
\leq n^{-\vartheta}~.
\end{align*}
\end{proof}

\noindent The following inequality is a (direct) consequence of inequality 14.0.9 from \citesuppl{shorack1986}.
\begin{lemma}\label{lem:shorack}
Under the assumptions of Theorem \ref{thm:quantremain}, let $L_F := \sup f(x)>0$.
It holds for all $0 < a \leq \dfrac{1}{2L_F}$, $s>0$, $n \in \N$ that
\begin{align*}
\Pb \bigg( \sup_{|x-y| \leq a} \big|\hat{F}_1^n (x) - F(x) - \big(\hat{F}_1^n (y) - F(y)\big)\big| \geq \dfrac{s\sqrt{L_Fa}}{\sqrt{n}}\bigg)
\leq \dfrac{c_1}{a} \exp \bigg(-c_2s^2\psi\Big(\dfrac{s}{\sqrt{nL_Fa}}\Big)\bigg)~,
\end{align*}
where $c_1$ and $c_2$ are positive constants (only depending on $F$) and $\psi$ is defined by
\begin{align*}
\psi(x) = 2\dfrac{(x+1)\log(x+1)-x}{x^2} \qquad \text{for}~ x>0~.
\end{align*}
\end{lemma}
\begin{proof}
Denote by $U_1,\dots ,U_n$ a sample of i.i.d. $\sim \mathcal{U}([0,1])$ random variables and note that by Lipschitz continuity $|x-y|\leq a$ implies $|F(x)-F(y)|\leq L_F\cdot a$.
Using also that $F$ is surjective and continuous by assumption, we obtain by quantile transformation
\begin{align*}
&\sup_{|x-y| \leq a} \big|\hat{F}_1^n (x) - F(x) - \big(\hat{F}_1^n (y) - F(y)\big)\big|\\
&= \sup_{|x-y| \leq a} \dfrac{1}{n}\bigg| \sum_{i=1}^n I\{X_i \leq x\} -F(x) - I\{X_i \leq y\}  + F(y) \bigg|\\
&\eqd \sup_{|x-y| \leq a} \dfrac{1}{n}\bigg| \sum_{i=1}^n I\{F^-(U_i) \leq x\} -F(x) - I\{F^-(U_i) \leq y\}  + F(y) \bigg|\\
&= \sup_{|x-y| \leq a} \dfrac{1}{n}\bigg| \sum_{i=1}^n I\{U_i \leq F(x)\} -F(x) - I\{U_i \leq F(y)\}  + F(y) \bigg|\\
&\leq \sup_{|F(x)-F(y)| \leq L_F\cdot a} \dfrac{1}{n}\bigg| \sum_{i=1}^n I\{U_i \leq F(x)\} -F(x) - I\{U_i \leq F(y)\}  + F(y) \bigg|\\
&= \sup_{\substack{x,y \in [0,1] \\|x-y| \leq L_F\cdot a}} \dfrac{1}{n}\bigg| \sum_{i=1}^n I\{U_i \leq x\} -x - I\{U_i \leq y\}  + y \bigg|~.
\end{align*} 
The claim now follows from inequality 14.0.9 in \citesuppl{shorack1986} for the uniform empirical process.
\end{proof}

\begin{lemma}\label{lem:FFxy}
Under the assumptions of Theorem \ref{thm:quantremain} for all $0<r<1$ and $\vartheta > 1$, there exists a constant $C_{r,\vartheta}$, such that for all positive sequences $\{a_n\}_{n \in \N}$ with 
\begin{align}\label{eq:assumpan}
\dfrac{r\log n}{n^r a_n} = o(1)
\;\;\;
\text{and}
\;\;\; a_n = o(1)
\end{align}
it holds that
\begin{align*}
\Pb\bigg(\max_{\substack{1 \leq i < j \leq n \\ |j - i| \geq n^r }}\sup_{|x - y| \leq a_n}\;
|\hatFij(x) - F(x) - (\hatFij(y) - F(y))| > C_{r,\vartheta} \sqrt{\dfrac{a_n r\log(n)}{n^r}} \bigg) \lesssim n^{-\vartheta}
\end{align*}
provided that $n$ is sufficiently large.
\end{lemma}
\begin{proof}
First consider $m=m_n\geq n^r$.
For $n$ sufficiently large we have $a_n \leq 1/(2L_F)$ and so choosing $a=a_{n}$ and $s=C_{r,\vartheta}\sqrt{\log(m)/L_F}$ in Lemma \ref{lem:shorack} we obtain that
\begin{align*}
&\Pb \bigg( \sup_{|x-y| \leq a_n} \big|\hat{F}_1^m (x) - F(x) - \big(\hat{F}_1^m (y) - F(y)\big)\big| \geq C_{r,\vartheta}\dfrac{\sqrt{a_n\log(m)}}{\sqrt{m}}\bigg)\\
&\leq \dfrac{c_1}{a_n} \exp \bigg( -c_2C_{r,\vartheta}^2\log(m)\psi\bigg( \dfrac{C_{r,\vartheta}}{L_F}\sqrt{\dfrac{\log(m)}{m a_n}}\bigg) \bigg)~.
\end{align*}
Using that $\psi$ is non-increasing the last expression can be bounded by
\begin{align}\label{psifallend}
\dfrac{c_1}{a_n} \exp \bigg( -c_2\log(n^{C_{r,\vartheta}^2r})\psi\bigg( \dfrac{C_{r,\vartheta}}{L_F}\sqrt{\dfrac{r\log(n)}{n^ra_n}}\bigg) \bigg)~.
\end{align}
Next note that by the assumption on $a_n$ we obtain
\begin{align*}
\lim_{n \to \infty} \psi\bigg( \dfrac{C_{r,2}}{L_F}\sqrt{\dfrac{r\log(n)}{n^ra_n}}\bigg)
= 1~.
\end{align*}
Thus for $n$ sufficiently large and with an adapted constant $\tilde{c_2}$ the term in \eqref{psifallend} is bounded by 
\begin{align*}
\dfrac{c_1}{a_n} \exp \bigg( -\tilde{c}_2\log(n^{C_{r,\vartheta}^2r})\bigg)
= \dfrac{c_1 n^{-\tilde{c}_2C_{r,\vartheta}^2 r}}{a_n}
= \dfrac{c_1 n^{-\tilde{c}_2C_{r,\vartheta}^2 r +r}}{n^ra_n}
\lesssim n^{-\vartheta-2}~,
\end{align*}
where we chose $C_{r,\vartheta}$ sufficiently large in the last estimate and used that $(n^ra_n)^{-1} = o(1)$ by assumption \eqref{eq:assumpan}.
Since the sequence $\{X_t\}_{t \in \Z}$ is i.i.d. we can now finish the proof
\begin{align*}
&\Pb\bigg(\max_{\substack{1 \leq i < j \leq n \\ |j - i| \geq n^r }}\sup_{|x - y| \leq a_n}\;
|\hatFij(x) - F(x) - (\hatFij(y) - F(y))| > C_{r,\vartheta} \dfrac{\sqrt{a_n r\log(n)}}{n^{r/2}} \bigg)\\
&\leq \sum_{\substack{1 \leq i < j \leq n \\ |j - i| \geq n^r }} \Pb\bigg(\sup_{|x - y| \leq a_n}\;
|\hatFij(x) - F(x) - (\hatFij(y) - F(y))| > C_{r,\vartheta} \dfrac{\sqrt{a_n r\log(n)}}{n^{r/2}} \bigg)\\
&\leq \sum_{\substack{1 \leq i < j \leq n \\ |j - i| \geq n^r }} \Pb\bigg(\sup_{|x - y| \leq a_n}\;
|\hatFij(x) - F(x) - (\hatFij(y) - F(y))| > C_{r,\vartheta} \dfrac{\sqrt{a_n \log(j-i+1)}}{\sqrt{j-i+1}} \bigg)\\
&\lesssim \sum_{\substack{1 \leq i < j \leq n \\ |j - i| \geq n^r }} n^{-\vartheta -2} \leq n^{-\vartheta}~.
\end{align*} 
\end{proof}

\begin{remark}\label{rem:BorelCantelli}
For the remainder of the chapter we can choose fixed $\vartheta>1$ and denote by $C_{r,1}$ and $C_{r,2}$ the corresponding constants
from Lemma \ref{lem:EDFbound} and \ref{lem:FFxy}, respectively.
Further define the sequence
\begin{align}\label{eq:bn}
b_{n,r} = C_{r,3}\sqrt{r\log(n)}/n^{r/2}~,
\end{align}
where $C_{r,3}$ is a constant such that $C_{r,3} > 2(C_{r,1}+1)/f(q_\beta)$.
Now let 'i.o.' be a shortcut for 'infinitely often' and note that Lemma \ref{lem:EDFbound} and \ref{lem:FFxy} together with the Borel-Cantelli Lemma imply
\begin{align*}
\Pb\bigg(\max_{\substack{1 \leq i < j \leq n \\ |j - i| \geq n^r }}\sup_{x \in \R} \;
|\hatFij(x) - F(x)| > C_{r,1} \dfrac{\sqrt{r\log(n)}}{n^{r/2}} \;\;\text{i.o.} \bigg)=0
\end{align*}
and
\begin{align*}
\Pb\bigg(\max_{\substack{1 \leq i < j \leq n \\ |j - i| \geq n^r }}\sup_{|x - y| \leq b_{n,r}}\;
|\hatFij(x) - F(x) - \big(\hatFij(y) - F(y)\big)| > C_{r,2} \dfrac{\sqrt{b_{n,r}r\log(n)}}{n^{r/2}} \;\;\text{i.o.} \bigg)=0~,
\end{align*}
which we require for the proof of the next Lemma.
\end{remark}

\begin{lemma}\label{lem:quantas}
Under the assumptions of Theorem \ref{thm:quantremain} it holds for all $r \in (0,1)$ that
\begin{align*}
\limsup_{n \to \infty} b_{n,r}^{-1}\max_{\substack{1 \leq i < j \leq n \\ |j - i| \geq n^r }} | \hatqijbeta - q_{\beta} |\leq 1
\end{align*}
with probability one.
\end{lemma}

\begin{proof}
The claim is equivalent to 
\begin{align}\label{eq:quantprob}
\Pb\bigg( \max_{\substack{1 \leq i < j \leq n \\ |j - i| \geq n^r }} | \hatqijbeta - q_{\beta} | > b_{n,r} \;\;\text{i.o.}\bigg) = 0~.
\end{align}
By definition of the empirical quantile $\max_{\substack{1 \leq i < j \leq n \\ |j - i| \geq n^r }} | \hatqijbeta - q_{\beta} | > b_{n,r}$ means that at least one of the considered e.d.f.s first exceeds the level $\beta$ outside of the interval $[q_\beta - b_{n,r},q_\beta + b_{n,r}]$.
Thus using also monotonicity of the e.d.f.s statement \eqref{eq:quantprob} follows if we can establish
\begin{enumerate}[(i)]
\item $\Pb\bigg( \min\limits_{\substack{1 \leq i < j \leq n \\ |j - i| > n^r }}  \hatFij(q_\beta + b_{n,r}) - \beta < 0 \;\;\text{i.o.}\bigg)=0$~,
\item $\Pb\bigg( \max\limits_{\substack{1 \leq i < j \leq n \\ |j - i| > n^r }}  \hatFij(q_\beta - b_{n,r}) - \beta > 0 \;\;\text{i.o.}\bigg)=0$~.
\end{enumerate}
Let us start with (i).
By a Taylor expansion we obtain
\begin{align*}
&\min_{\substack{1 \leq i < j \leq n \\ |j - i| \geq n^r }}  \hatFij(q_\beta + b_{n,r}) - \beta\\
&= \min_{\substack{1 \leq i < j \leq n \\ |j - i| \geq n^r }} \bigg[ F(q_\beta + b_{n,r}) - \beta - \hatFij(q_\beta ) + \beta
+ \hatFij(q_\beta + b_{n,r}) - F(q_\beta + b_{n,r}) + \hatFij(q_\beta) - F(q_\beta) \bigg]\\
&\geq F(q_\beta + b_{n,r}) - \beta - \max_{\substack{1 \leq i < j \leq n \\ |j - i| \geq n^r }}  |\hatFij(q_\beta) - \beta|
- \max_{\substack{1 \leq i < j \leq n \\ |j - i| > n^r }}\sup_{|x-y|\leq b_{n,r}} \big|\hatFij(x)-F(x) - \big(\hatFij(y) - F(y)\big) \big|\\
&\geq b_{n,r}f(q_\beta) + \frac{b_{n,r}^2}{2}\inf_{x \in \R}f'(x) - \max_{\substack{1 \leq i < j \leq n \\ |j - i| \geq n^r }}  |\hatFij(q_\beta) - \beta|\\
&\;\qquad- \max_{\substack{1 \leq i < j \leq n \\ |j - i| \geq n^r }}\sup_{|x-y|\leq b_{n,r}} \big|\hatFij(x)-F(x) - \big(\hatFij(y) - F(y)\big) \big|~.
\end{align*}
This yields
\begin{align*}
&\Pb\bigg( \min_{\substack{1 \leq i < j \leq n \\ |j - i| \geq n^r }}  \hatFij(q_\beta + b_{n,r}) - \beta \leq 0 \;\;\text{i.o.}\bigg)\\
&\leq \Pb\bigg(   \max_{\substack{1 \leq i < j \leq n \\ |j - i| \geq n^r }}\sup_{|x-y|\leq b_{n,r}} \big|\hatFij(x)-F(x) - \big(\hatFij(y) - F(y)\big) \big|
+ \max_{\substack{1 \leq i < j \leq n \\ |j - i| \geq n^r }}  |\hatFij(q_\beta) - \beta|\\
&\hspace{10.4cm}\geq b_{n,r}f(q_\beta) + \frac{b_{n,r}^2}{2}\inf_{x \in \R}f'(x) \;\;\text{i.o.}\bigg)\\
&\leq \Pb\bigg(\max_{\substack{1 \leq i < j \leq n \\ |j - i| \geq n^r }}  |\hatFij(q_\beta) - \beta|\geq \frac{b_{n,r}}{2}f(q_\beta) + \frac{b_{n,r}^2}{4}\inf_{x \in \R}f'(x) \;\;\text{i.o.}\bigg)\\
&\qquad+ \Pb\bigg(\max_{\substack{1 \leq i < j \leq n \\ |j - i| \geq n^r }}\sup_{|x-y|\leq b_{n,r}} \big|\hatFij(x)-F(x) - \big(\hatFij(y) - F(y)\big) \big|\geq \frac{b_{n,r}}{2}f(q_\beta) + \frac{b_{n,r}^2}{4}\inf_{x \in \R}f'(x) \;\;\text{i.o.}\bigg)~.
\end{align*}
By definition of $b_{n,r}$ in \eqref{eq:bn} we have $b_{n,r}f(q_\beta)/2 > (C_{r,1}+1)\sqrt{r\log(n)}/n^{r/2}$ and so Remark \ref{rem:BorelCantelli} yields that the last two probabilities are zero.
To achieve (ii), we proceed similar and obtain 
\begin{align*}
&\;\max_{\substack{1 \leq i < j \leq n \\ |j - i| \geq n^r }}  \hatFij(q_\beta - b_{n,r}) - \beta\\
&= \max_{\substack{1 \leq i < j \leq n \\ |j - i| \geq n^r }}  F(q_\beta - b_{n,r}) - \beta - \hatFij(q_\beta ) + \beta
+ \hatFij(q_\beta - b_{n,r}) - F(q_\beta - b_{n,r}) + \hatFij(q_\beta) - F(q_\beta)\\
&\leq F(q_\beta - b_{n,r}) - \beta + \max_{\substack{1 \leq i < j \leq n \\ |j - i| \geq n^r }} | \hatFij(q_\beta ) - \beta|\\
&\;\;+  \max_{\substack{1 \leq i < j \leq n \\ |j - i| \geq n^r }}\sup_{|x-y|\leq b_{n,r}} |\hatFij(x) - F(x) - \big(\hatFij(y) - F(y)\big)|\\
&\leq -f(q_\beta)b_{n,r} + \sup_{x \in \R}f'(x)\frac{b_{n,r}^2}{2} + \max_{\substack{1 \leq i < j \leq n \\ |j - i| > n^r }} | \hatFij(q_\beta ) - \beta|\\
&\;\;+  \max_{\substack{1 \leq i < j \leq n \\ |j - i| \geq n^r }}\sup_{|x-y|\leq b_{n,r}} |\hatFij(x) - F(x) - \big(\hatFij(y) - F(y)\big)|
\end{align*}
This leads to
\begin{align*}
&\Pb\bigg( \max_{\substack{1 \leq i < j \leq n \\ |j - i| \geq n^r }}  \hatFij(q_\beta - b_{n,r}) - \beta \geq 0 \;\;\text{i.o.}\bigg)\\
&\leq \Pb\bigg( -f(q_\beta)b_{n,r} + \sup_{x \in \R}f'(x)\frac{b_{n,r}^2}{2} + \max_{\substack{1 \leq i < j \leq n \\ |j - i| \geq n^r }} | \hatFij(q_\beta ) - \beta|\\
&\hspace{5cm}+  \max_{\substack{1 \leq i < j \leq n \\ |j - i| \geq n^r }}\sup_{|x-y|\leq b_{n,r}} |\hatFij(x) - F(x) - \big(\hatFij(y) - F(y)\big)| \geq 0 \;\;\text{i.o.}\bigg)\\
&\leq \Pb\bigg( \max_{\substack{1 \leq i < j \leq n \\ |j - i| \geq n^r }} | \hatFij(q_\beta ) - \beta| \geq  f(q_\beta)b_{n,r} - \sup_{x \in \R}f'(x)\frac{b_{n,r}^2}{2} \;\;\text{i.o.}\bigg)\\
&\;\;+\Pb\bigg( \max_{\substack{1 \leq i < j \leq n \\ |j - i| \geq n^r }}\sup_{|x-y|\leq b_{n,r}} |\hatFij(x) - F(x) - \big(\hatFij(y) - F(y)\big)| \geq  f(q_\beta)b_{n,r} - \sup_{x \in \R}f'(x)\frac{b_{n,r}^2}{2} \;\;\text{i.o.}\bigg)~.
\end{align*}
Using again the definition of $b_{n,r}$ and Remark \ref{rem:BorelCantelli} the two probabilities are zero, which finishes the proof of Lemma \ref{lem:quantas}.
\end{proof}

\begin{remark}\label{rem:quantlarge}
Note that Lemma \ref{lem:quantas} in particular implies that (for all $0 < r <1$) and $c_0<1$
\begin{align*}
n^{c_0r/2} \max_{ \substack{1 \leq i < j \leq n \\ |j - i| \geq n^r }} \big| \hatqijbeta - q_\beta \big| = \op(1)~,
\end{align*}
which we require later on.
\end{remark}

\begin{lemma}\label{lem:quantlem}
Under the assumptions of Theorem \ref{thm:quantremain} it holds for all $0<r<1/2-1/\lambda$
\begin{align}\label{eq:quantlem}
\dfrac{1}{\sqrt{n}} \max_{ \substack{1 \leq i < j \leq n \\ |j - i| < n^r }} (j-i+1) \big| \hatqijbeta - q_\beta \big| = \op(1)~.
\end{align}
\end{lemma}
\begin{proof}
Due to $\min_{t=1}^n X_t \leq \hatqijbeta \leq \max_{t=1}^n X_t$ for all $i,j \in \{1,\dots,n\}$, we observe that the term on the left-hand side of \eqref{eq:quantlem} is bounded by
\begin{align*}
n^{r-1/2} \Big( |\max_{t=1}^n X_t - q_\beta| + |\min_{t=1}^n X_t - q_\beta|\Big)
\leq 2 n^{r-1/2}\max_{t=1}^n | X_t| +  o(1)~.
\end{align*}
Now for $\varepsilon>0$ we can employ the independence and obtain 
\begin{align*}
\Pb \Big( n^{r-1/2} \max_{i=1}^n |X_t| > \varepsilon \Big)
= 1 - \Pb \Big( n^{r-1/2} \max_{i=1}^n |X_t| \leq \varepsilon \Big)
&= 1 - \Pb \Big( |X_1| \leq n^{1/2-r}\varepsilon \Big)^n\\
&= 1 - \bigg( 1- \Pb \Big( |X_1| > n^{1/2-r}\varepsilon \Big) \bigg)^n~.
\end{align*}
By the assumption \eqref{ineq:tail} on the tails of the distribution of $|X_1|$ this is now bounded by
\begin{align*}
1 - \Big(1 - \varepsilon^{-\lambda} n^{-\lambda(1/2-r)}  \Big)^n = o(1)~,
\end{align*}
where we used that by assumption $\lambda(1/2-r) > 1$.
\end{proof}

\begin{lemma}\label{lem:combined}
Under the assumptions of Theorem \ref{thm:quantremain} it holds for all $2/9<r<1$
\begin{align*}
\dfrac{1}{\sqrt{n}} \max_{\substack{1 \leq i < j \leq n \\ |j - i| > n^r }} (j-i+1) \Big|\hatFij(\hatqijbeta) - \hatFij(q_\beta) - (F(\hatqijbeta) - F(q_\beta))\Big|
= o_p(1)~.
\end{align*}
\end{lemma}
\begin{proof}
Fix $\varepsilon >0$ and choose $\delta$ such that $2/3 < \delta < \tfrac{3}{4}r + 1/2$ , then it holds that
\begin{align}
&\Pb\bigg( \dfrac{1}{\sqrt{n}} \max_{\substack{1 \leq i < j \leq n \\ |j - i| \geq n^r }}  (j-i+1) \Big|\hatFij(\hatqijbeta) - \hatFij(q_\beta) - \big(F(\hatqijbeta) - F(q_\beta)\big)\Big|>\varepsilon \bigg)\nonumber\\
&\leq \Pb\bigg( \dfrac{1}{\sqrt{n}} \max_{\substack{1 \leq i < j \leq n \\ n^\delta > |j - i| \geq n^r }} (j-i+1)\Big|\hatFij(\hatqijbeta) - \hatFij(q_\beta) - \big(F(\hatqijbeta) - F(q_\beta)\big)\Big|>\varepsilon \bigg)\label{ineq:combsep}\\
&\qquad\qquad\qquad+ \Pb\bigg( \dfrac{1}{\sqrt{n}} \max_{\substack{1 \leq i < j \leq n \\ |j - i| \geq n^\delta }}(j-i+1) \Big|\hatFij(\hatqijbeta) - \hatFij(q_\beta) - \big(F(\hatqijbeta) - F(q_\beta)\big)\Big|>\varepsilon \bigg)~.\nonumber
\end{align}
We will treat the two summands on the right-hand side separately.\\
\textbf{First summand of (\ref{ineq:combsep}):}
Using $\delta - 1/2 < 3/4r$, we can choose a constant $0<c_0<1$ sufficiently large, such that $\delta -1/2 < (c_0/4+1/2)r$.
Further choose $a_n = n^{-c_0r/2}$.
The first summand of \eqref{ineq:combsep} is then bounded by
\begin{align}\label{ineq:combsep2}
\begin{split}
&\Pb\bigg( \max_{\substack{1 \leq i < j \leq n \\ |j - i| \geq n^r }} n^{\delta-1/2}\Big|\hatFij(\hatqijbeta) - \hatFij(q_\beta) - \big(F(\hatqijbeta) - F(q_\beta)\big)\Big|>\varepsilon \bigg)\\
&\leq \Pb\bigg( \max_{\substack{1 \leq i < j \leq n \\ |j - i| \geq n^r }} \sup_{|x-y|\leq a_{n,r}} \Big|\hatFij(x) - F(x) - \hatFij(y) + F(y)\Big|> \dfrac{\varepsilon}{n^{\delta-1/2}}\bigg)\\
&\hspace{8cm}+ \Pb\bigg(\max_{\substack{1 \leq i < j \leq n \\ |j - i| \geq n^r }} |\hatqijbeta - q_\beta| > a_{n,r}\bigg)~.
\end{split}
\end{align}
By Remark \ref{rem:quantlarge} the second summand of the right-hand side of \eqref{ineq:combsep2} converges to zero.
For the first summand of \eqref{ineq:combsep2} note that using $\delta -1/2 < (c_0/4+1/2)r$, we obtain (provided that $n$ is sufficiently large)
\begin{align*}
\dfrac{\varepsilon}{n^{\delta - 1/2 }}
\geq C_{r,2}\dfrac{\sqrt{r\log(n)}}{n^{c_0r/4+r/2}}
= C_{r,2}\dfrac{\sqrt{a_{n,r}r\log(n)}}{n^{r/2}}
\end{align*}
 and so the first summand of \eqref{ineq:combsep2} converges to zero by Lemma \ref{lem:FFxy}.\\
\textbf{Second summand of (\ref{ineq:combsep}):}
Due to $\delta>2/3$, we can choose a constant $0<c_0<1$ sufficiently large, such that $1/2 < \delta/2 +c_0\delta/4$.
Next define $a_{n,\delta} = n^{-c_0\delta/2}$ and obtain the bound
\begin{align}\label{ineq:combsep3}
&\Pb\bigg( \dfrac{1}{\sqrt{n}} \max_{\substack{1 \leq i < j \leq n \\ |j - i| \geq n^\delta }}(j-i+1) \Big|\hatFij(\hatqijbeta) - \hatFij(q_\beta) - \big(F(\hatqijbeta) - F(q_\beta)\big)\Big|>\varepsilon \bigg)\nonumber\\
\begin{split}
&\leq\Pb\bigg( \max_{\substack{1 \leq i < j \leq n \\ |j - i| \geq n^\delta }} \sup_{|x-y|\leq a_{n,\delta}} \Big|\hatFij(x) - F(x) - \hatFij(y) + F(y)\Big|> \dfrac{\varepsilon}{n^{1/2}}\bigg)\\
&\hspace{8cm}+ \Pb\bigg(\max_{\substack{1 \leq i < j \leq n \\ |j - i| \geq n^\delta }} |\hatqijbeta - q_\beta| > a_{n,\delta}\bigg)~.
\end{split}
\end{align}
Employing again Remark \eqref{rem:quantlarge} the second summand of \eqref{ineq:combsep3} converges to zero.
For the first summand of \eqref{ineq:combsep3}, note that we have (for sufficiently large $n$)
\begin{align*}
\dfrac{\varepsilon}{n^{1/2}}
\geq C_{\delta,2} \dfrac{\sqrt{\delta\log(n)}}{n^{c_0\delta/4+\delta/2}}
= C_{\delta,2} \dfrac{\sqrt{a_{n,\delta}\delta\log(n)}}{n^{\delta/2}}
\end{align*}
and so Lemma \ref{lem:FFxy} finishes the proof.
\end{proof}

\noindent Now we are able to proceed to the actual proof of Theorem \ref{thm:quantremain}.\\[10pt]
\textit{Proof of Theorem \ref{thm:quantremain}:}\\
Since $\beta$ is fixed, it is easy to see that the claim is equivalent to 
\begin{align}\label{eq:qremain1}
\dfrac{1}{\sqrt{n}} \max_{1\leq i < j \leq n} (j-i+1) \bigg| f(q_\beta) \big(\hatqijbeta - q_\beta\big) - \beta + \hatFij(q_\beta) \bigg| = \op(1)~.
\end{align}
Further note that (since $F$ is continuous) $|\hatFij(\hatqijbeta) - \beta | \leq (j-i+1)^{-1}$ almost surely, which yields
\begin{align*}
\dfrac{1}{\sqrt{n}} \max_{1\leq i < j \leq n} (j-i+1) |\hatFij(\hatqijbeta) - \beta |
= \op(1)
\end{align*} 
and so it remains to prove
\begin{align}\label{eq:qremain2}
\dfrac{1}{\sqrt{n}} \max_{1\leq i < j \leq n} (j-i+1) \bigg| f(q_\beta) \big(\hatqijbeta - q_\beta\big) - \hatFij(\hatqijbeta) + \hatFij(q_\beta) \bigg|
= \op(1)~.
\end{align}
Now due to $\lambda > 18/5$, we can choose $r_1$ with $2/9 < r_1 < 1/2 - 1/ \lambda<1/2$.
By Lemma \ref{lem:quantlem} and $\hatFij(\hatqijbeta), \hatFij(q_\beta) \in [0,1]$ we only have to verify
\begin{align*}
\dfrac{1}{\sqrt{n}} \max_{ \substack{1 \leq i < j \leq n \\ |j - i| \geq n^{r_1} }}
(j-i+1) \bigg| f(q_\beta) \big(\hatqijbeta - q_\beta\big) - \hatFij(\hatqijbeta) + \hatFij(q_\beta) \bigg| = \op(1)~.
\end{align*}
Now employing Lemma \ref{lem:combined}, the statement above follows if we can establish
\begin{align}\label{eq:2last}
\dfrac{1}{\sqrt{n}} \max_{ \substack{1 \leq i < j \leq n \\ |j - i| \geq n^{r_1} }}
(j-i+1) \bigg| f(q_\beta) \big(\hatqijbeta - q_\beta\big) - F(\hatqijbeta) + F(q_\beta) \bigg| = \op(1)~.
\end{align}
By means of a Taylor expansion the term on the left-hand side is (up to a constant almost surely) bounded by
\begin{align*}
\dfrac{1}{\sqrt{n}} &\max_{ \substack{1 \leq i < j \leq n \\ |j - i| \geq n^{r_1} }}(j-i+1)\sup_{x \in \R}|f'(x)|\big(\hatqijbeta - q_\beta\big)^2~,
\end{align*}
where the factor $\sup_{x \in \R} |f'(x)|$ is bounded by assumption.
Now since $2/9 < r_1 < 1/2$, it is easy to see, that we can choose $0<c_0<1$ (sufficiently large), such that
\begin{align*}
\dfrac{1}{2c_0} \leq r_1c_0 + 1/2~.
\end{align*}
Thus we can select $\delta$ that fulfills
\begin{align*}
r_1 < \dfrac{1}{2c_0} \leq \delta \leq r_1c_0 + 1/2~.
\end{align*}
We consider the cases $n^{r_1}\leq|i-j|\leq n^{\delta}$ and $n^{\delta}\leq|i-j|$ separately.
For the first one we obtain
\begin{align*}
&\dfrac{1}{\sqrt{n}} \max_{ \substack{1 \leq i < j \leq n \\ n^{r_1} \leq |j - i| < n^{\delta} }}(j-i+1)|\big(\hatqijbeta - q_\beta\big)^2
\leq \max_{ \substack{1 \leq i < j \leq n \\ n^{r_1} \leq |j - i| < n^{\delta }}}n^{\delta-1/2}\big(\hatqijbeta - q_\beta\big)^2\\
&\leq \max_{ \substack{1 \leq i < j \leq n \\ n^{r_1} \leq |j - i| < n}} n^{c_0r_1}\big(\hatqijbeta - q_\beta\big)^2
= \max_{ \substack{1 \leq i < j \leq n \\ n^{r_1} \leq |j - i| < n}} \big(n^{c_0r_1/2}(\hatqijbeta - q_\beta)\big)^2
= \op(1)~,
\end{align*}
where we used Remark \ref{rem:quantlarge} for the last estimate.
For the other case we obtain similarly
\begin{align*}
\dfrac{1}{\sqrt{n}} \max_{ \substack{1 \leq i < j \leq n \\ n^{\delta} \leq |j - i| < n }}(j-i+1)|\big(\hatqijbeta - q_\beta\big)^2
&\leq \max_{ \substack{1 \leq i < j \leq n \\ n^{\delta} \leq |j - i| < n}}n^{1/2}\big(\hatqijbeta - q_\beta\big)^2\\
\leq \max_{ \substack{1 \leq i < j \leq n \\ n^{\delta} \leq |j - i| < n}} n^{c_0\delta}\big(\hatqijbeta - q_\beta\big)^2
&= \max_{ \substack{1 \leq i < j \leq n \\ n^{\delta} \leq |j - i| < n}} \big(n^{c_0\delta/2}(\hatqijbeta - q_\beta)\big)^2
= \op(1)~,
\end{align*}
where we again employed Remark \ref{rem:quantlarge}.
\end{proof}

\section{Additional simulation results} \label{addsim}
In this section we provide some additional simulation results to allow a more detailed analysis of the presented detection schemes.
We will focus on changes in the mean as presented in Section \ref{sec51} and study the following aspects:
\begin{enumerate}[{Section B.}1:]
\item The influence of the actual change point locations on the power.
\item Other choices of the factor $T$, that controls the monitoring window length.
\end{enumerate}

\subsection{Influence of change point locations}
In this section we report simulation results for the situation considered in Figure \ref{fig:3} except for the change point locations, for which we consider rather early and late locations.
Figure \ref{fig:earlychange} displays the power of the non self-normalized procedures for the different choices of the model and the threshold considered in Section \ref{sec51}, where the change occurs already at observation $X_{120}$ and a historical training data ending at $X_{100}$.
This can be considered as a situation of an early change and the displayed plots can be explained as follows.
In all combinations, the detection scheme based on $\hat{D}$ still has a slightly larger power compared to the methods based on $\hat{P}$ and $\hat{Q}$, while $\hat{P}$ slightly outperforms $\hat{Q}$.
Compared to Figure \ref{fig:3} the differences with respect to the different schemes are considerably smaller.
These observations may be explained by the different constructions of the detection schemes, that are described at the end of Section \ref{sec2}.
In particular the performance of the monitoring schemes based on $\hat{Q}$ and $\hat{P}$ improves if the change occurs closer to the monitoring start, see also the discussion at the end of Section \ref{sec2}.

\noindent In Figure \ref{fig:latechange} we report the power for a change located close to the end of the monitoring period.
Here the break occurs at observation $X_{180}$, while the monitoring window ends with observation $X_{200}$.
Concerning the small number of 20 observations after the change, such an event is certainly harder to detect.
Consequently, all schemes perform inferior compared to the situations considered in Figure \ref{fig:3} and \ref{fig:earlychange}.
However, the power superiority of the methods based on the statistics $\hat{D}$ over $\hat{P}$ and $\hat{Q}$ is even more significant now.
These results support our initial conjecture: While all schemes behave more or less equivalent for changes close to the start, $\hat{D}$ offers better characteristics, if changes are located closer to the end.

\begin{figure}[H]
\begin{tabular}{b{1cm}ccc}
& (T1) & (T2) & (T3)\\
(M1) \tabj &
\includegraphics[width=4.5cm,height=3.5cm]{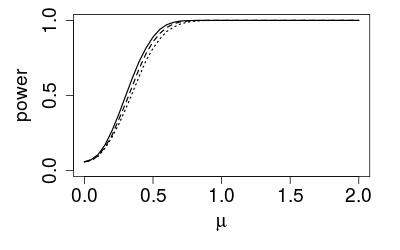} &
\includegraphics[width=4.5cm,height=3.5cm]{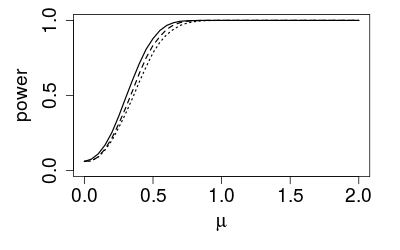} &
\includegraphics[width=4.5cm,height=3.5cm]{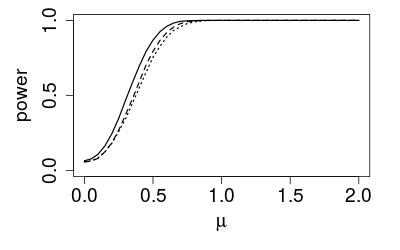} \vspace{-0.3cm}
\\
(M2) \tabj &
\includegraphics[width=4.5cm,height=3.5cm]{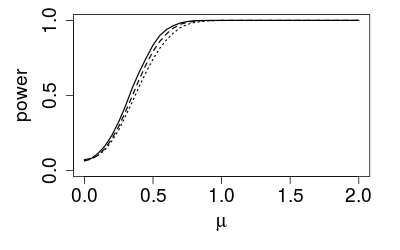} &
\includegraphics[width=4.5cm,height=3.5cm]{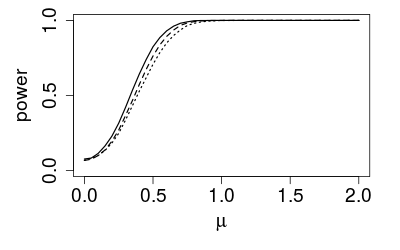} &
\includegraphics[width=4.5cm,height=3.5cm]{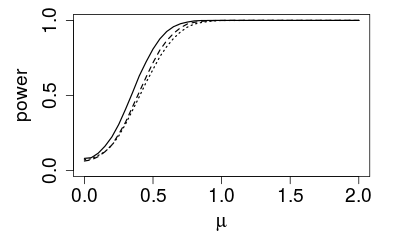} \vspace{-0.3cm}
\\
(M3) \tabj &
\includegraphics[width=4.5cm,height=3.5cm]{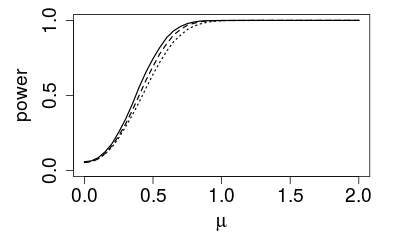} &
\includegraphics[width=4.5cm,height=3.5cm]{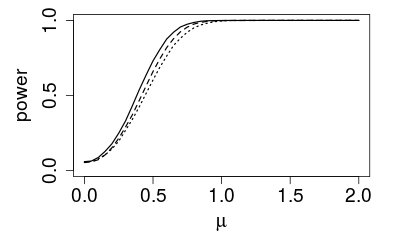} &
\includegraphics[width=4.5cm,height=3.5cm]{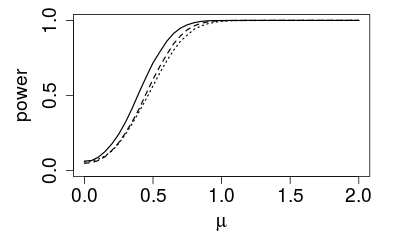} \vspace{-0.3cm}
\\
(M4) \tabj &
\includegraphics[width=4.5cm,height=3.5cm]{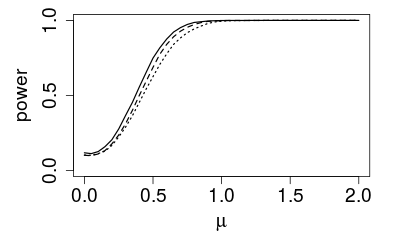} &
\includegraphics[width=4.5cm,height=3.5cm]{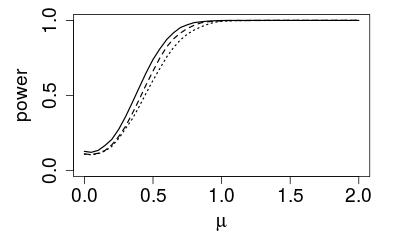} &
\includegraphics[width=4.5cm,height=3.5cm]{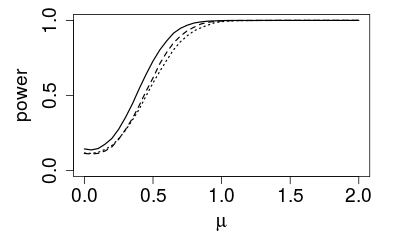} \vspace{-0.3cm}
\end{tabular}
\caption{\it Empirical rejection probabilities of the sequential tests for a change in the mean based on the statistics $\hat{D}$ (solid line), $\hat{P}$ (dashed line) , $\hat{Q}$ (dotted line).
The initial and total sample size are $m=100$ and $m(T+1)=200$, respectively, and the change occurs at observation $120$.
The level is $\alpha=0.05$.
Different rows correspond to different threshold functions, while different columns correspond to different models.\label{fig:earlychange}}
\end{figure}

\begin{figure}[H]
\begin{tabular}{b{1cm}ccc}
& (T1) & (T2) & (T3)\\
(M1) \tabj &
\includegraphics[width=4.5cm,height=3.5cm]{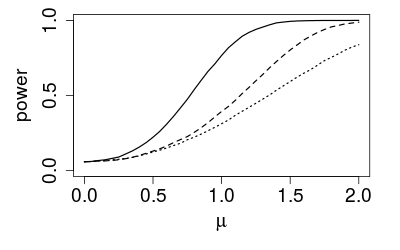} &
\includegraphics[width=4.5cm,height=3.5cm]{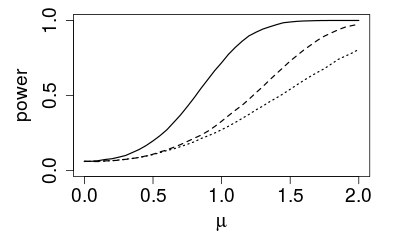} &
\includegraphics[width=4.5cm,height=3.5cm]{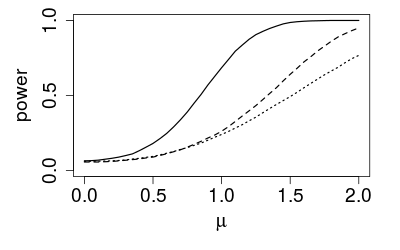} \vspace{-0.3cm}
\\
(M2) \tabj &
\includegraphics[width=4.5cm,height=3.5cm]{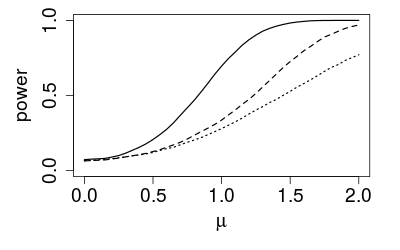} &
\includegraphics[width=4.5cm,height=3.5cm]{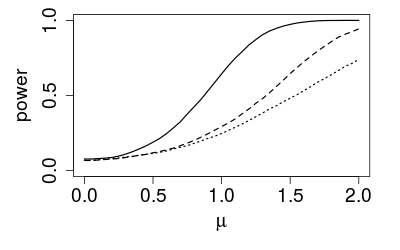} 
&
\includegraphics[width=4.5cm,height=3.5cm]{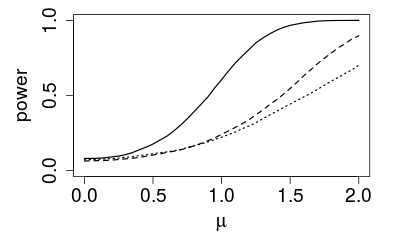} \vspace{-0.3cm}
\\
(M3) \tabj &
\includegraphics[width=4.5cm,height=3.5cm]{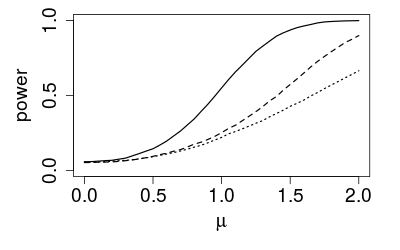} &
\includegraphics[width=4.5cm,height=3.5cm]{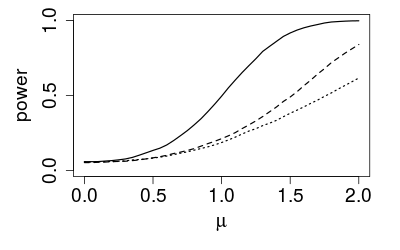} &
\includegraphics[width=4.5cm,height=3.5cm]{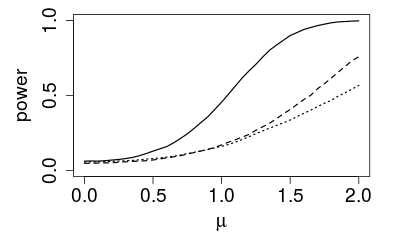} \vspace{-0.3cm}
\\
(M4) \tabj &
\includegraphics[width=4.5cm,height=3.5cm]{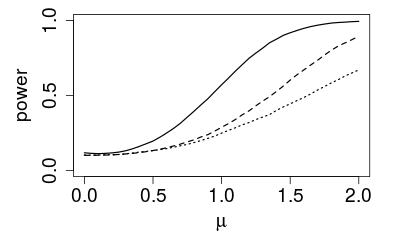} &
\includegraphics[width=4.5cm,height=3.5cm]{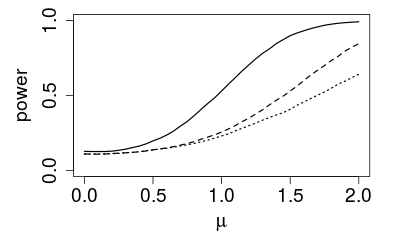} &
\includegraphics[width=4.5cm,height=3.5cm]{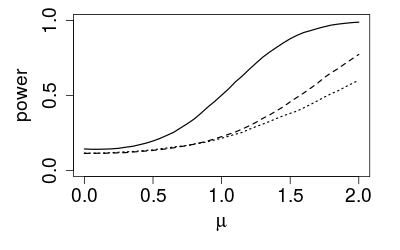} \vspace{-0.3cm}
\end{tabular}
\caption{\it Empirical rejection probabilities of the sequential tests for a change in the mean based on the statistics $\hat{D}$ (solid line), $\hat{P}$ (dashed line), $\hat{Q}$ (dotted line).
The initial and total sample size are $m=100$ and $m(T+1)=200$, respectively, and the change occurs at observation $180$.
The level is $\alpha=0.05$.
Different rows correspond to different models, while different columns correspond to different threshold functions.\label{fig:latechange}}
\end{figure}

\subsection{Larger monitoring windows}
In this section we report simulations with the same settings as in Figure \ref{fig:3} but with a larger monitoring window.
More precisely, we operate again with a set of $m=100$ stable observations, while the factor $T$ is set to $2$ and $3$ for the simulations in Figure \ref{fig:T2} and \ref{fig:T3}, respectively.
The change point is again located at the middle of the monitoring period.
The obtained results are similar to those for the case $T=1$ given in Section \ref{sec51} and for this reason we omit a detailed discussion here. 

\begin{figure}[H]
\begin{tabular}{b{1cm}ccc}
& (T1) & (T2) & (T3)\\
(M1) \tabj &
\includegraphics[width=4.5cm,height=3.5cm]{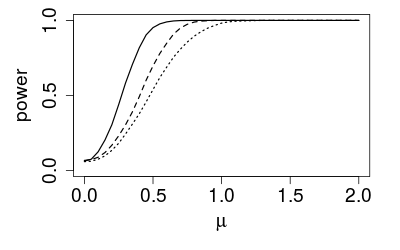} &
\includegraphics[width=4.5cm,height=3.5cm]{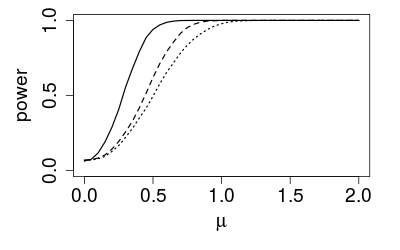} &
\includegraphics[width=4.5cm,height=3.5cm]{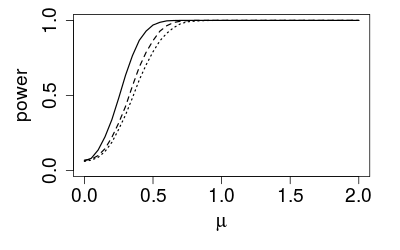} \vspace{-0.3cm}
\\
(M2) \tabj &
\includegraphics[width=4.5cm,height=3.5cm]{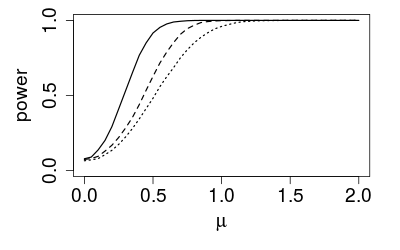} &
\includegraphics[width=4.5cm,height=3.5cm]{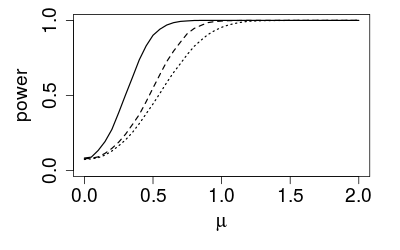} &
\includegraphics[width=4.5cm,height=3.5cm]{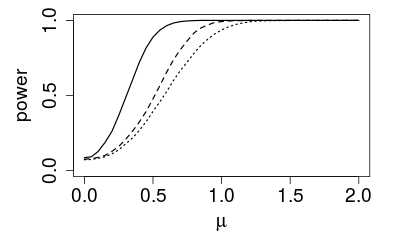} \vspace{-0.3cm}
\\
(M3) \tabj &
\includegraphics[width=4.5cm,height=3.5cm]{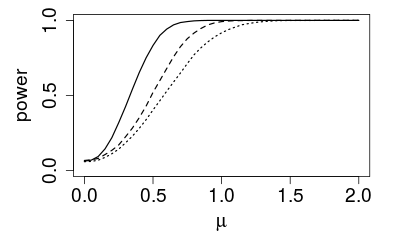} &
\includegraphics[width=4.5cm,height=3.5cm]{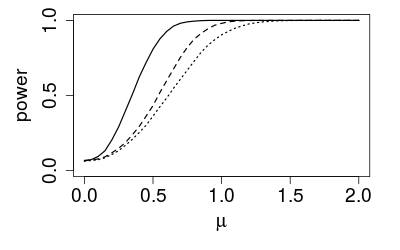} &
\includegraphics[width=4.5cm,height=3.5cm]{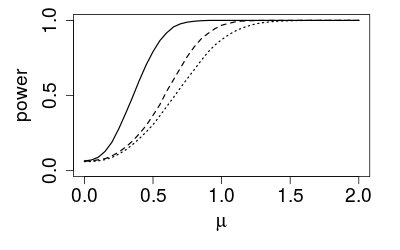} \vspace{-0.3cm}
\\
(M4) \tabj &
\includegraphics[width=4.5cm,height=3.5cm]{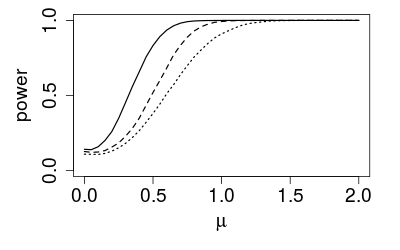} &
\includegraphics[width=4.5cm,height=3.5cm]{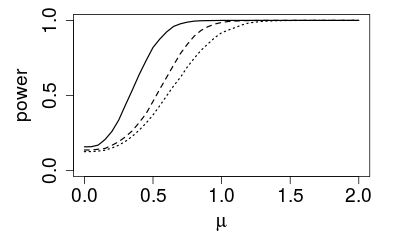} &
\includegraphics[width=4.5cm,height=3.5cm]{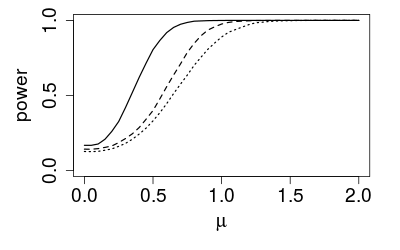} \vspace{-0.3cm}
\end{tabular}
\caption{\it Empirical rejection probabilities of the sequential tests for a change in the mean based on the statistics $\hat{D}$ (solid line), $\hat{P}$ (dashed line), $\hat{Q}$ (dotted line).
The initial and total sample size are $m=100$ and $m(T+1)=300$, respectively, and the change occurs at observation $200$.
The level is $\alpha=0.05$.
Different rows correspond to different models, while different columns correspond to different threshold functions.\label{fig:T2}}
\end{figure}

\begin{figure}[H]
\begin{tabular}{b{1cm}ccc}
& (T1) & (T2) & (T3)\\
(M1) \tabj &
\includegraphics[width=4.5cm,height=3.5cm]{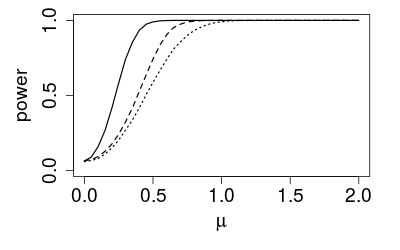} &
\includegraphics[width=4.5cm,height=3.5cm]{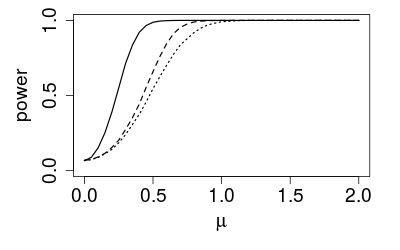} &
\includegraphics[width=4.5cm,height=3.5cm]{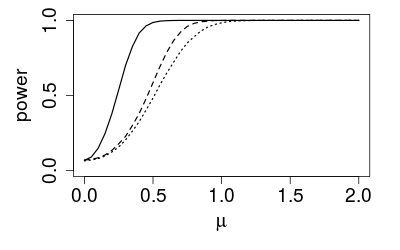} \vspace{-0.3cm}
\\
(M2) \tabj &
\includegraphics[width=4.5cm,height=3.5cm]{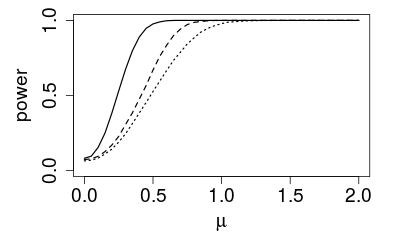} &
\includegraphics[width=4.5cm,height=3.5cm]{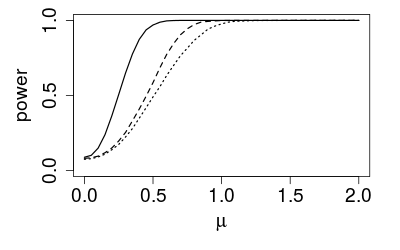} &
\includegraphics[width=4.5cm,height=3.5cm]{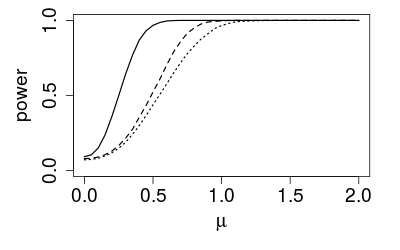} \vspace{-0.3cm}
\\
(M3) \tabj &
\includegraphics[width=4.5cm,height=3.5cm]{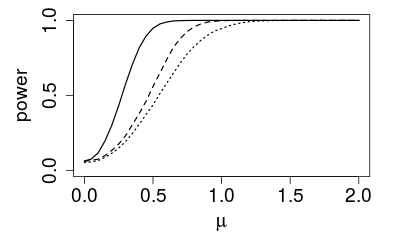} &
\includegraphics[width=4.5cm,height=3.5cm]{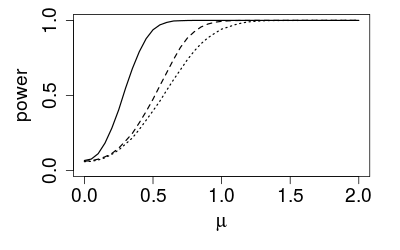} &
\includegraphics[width=4.5cm,height=3.5cm]{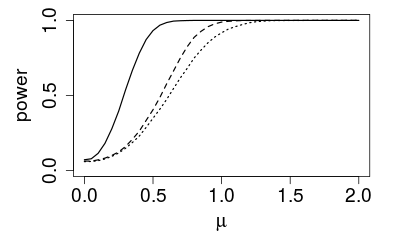} \vspace{-0.3cm}
\\
(M4) \tabj &
\includegraphics[width=4.5cm,height=3.5cm]{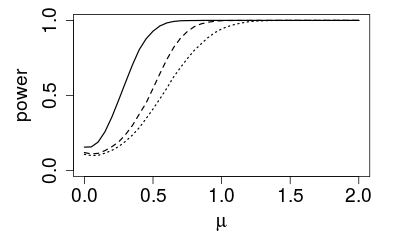} &
\includegraphics[width=4.5cm,height=3.5cm]{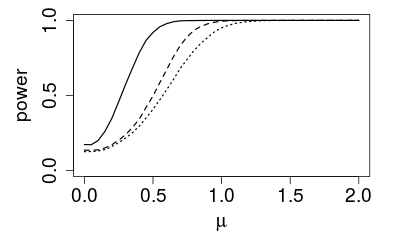} &
\includegraphics[width=4.5cm,height=3.5cm]{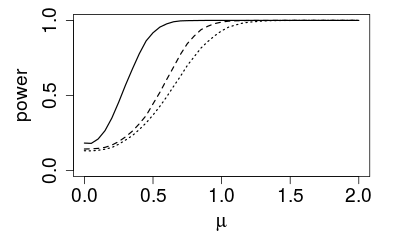} \vspace{-0.3cm}
\end{tabular}
\caption{\it Empirical rejection probabilities of the sequential tests for a change in the mean based on the statistics $\hat{D}$ (solid line), $\hat{P}$ (dashed line) , $\hat{Q}$ (dotted line).
The initial and total sample size are $m=100$ and $m(T+1)=400$, respectively, and the change occurs at observation $250$.
The level is $\alpha=0.05$.
Different rows correspond to different models, while different columns correspond to different threshold functions.\label{fig:T3}}
\end{figure}
\bibliographystylesuppl{apalike}
\bibliographysuppl{literature}

\end{document}